\documentclass[oneside,article]{memoir}
\setsecnumdepth{subsection}
\usepackage{amsmath}         
\usepackage{amsthm,thmtools,thm-restate} 
\usepackage{enumitem}        

\usepackage{tikz}
  \usetikzlibrary{matrix,arrows,positioning,decorations.markings}
  
\usepackage{tikz-cd}

\usepackage{float} 
\usepackage{caption, subcaption} 

\usepackage[nobysame,alphabetic,initials]{amsrefs} 
\usepackage[hidelinks]{hyperref}                   
\usepackage[capitalize,nosort]{cleveref}           

\usepackage{amssymb}
\usepackage[mathscr]{euscript}
\usepackage{relsize}
\usepackage{stmaryrd}
\usepackage{stackengine}
\usepackage{mathdots} 

\usepackage{indentfirst} 
\usepackage{multicol}

\usepackage{lmodern}
\usepackage{libertine}
\usepackage[libertine]{newtxmath}

\usepackage{booktabs}

\usepackage[inline,nomargin]{fixme} 
  \fxsetup{targetlayout=color}

\usepackage{float} 
\usepackage{wrapfig} 

\isopage[12]

\setlrmargins{*}{*}{1}
\checkandfixthelayout

\counterwithout{section}{chapter}
\usepackage{appendix}

\newcommand{\Cat}{\mathsf{Cat}}
\newcommand{\cSet}{\mathsf{cSet}}
\newcommand{\mcSet}{\cSet^+}

\newcommand{\Set}{\mathsf{Set}}
\newcommand{\msSet}{\sSet^+}

\newcommand{\Map}{\operatorname{Map}}

\newcommand{\face}[2]{\partial^{#1}_{#2}} 
\newcommand{\degen}[2]{\sigma^{#1}_{#2}} 

  \newcommand{\boxcat}{\mathord{\square}} 
  
  \newcommand{\conn}[2]{\gamma^{#1}_{#2}} 
  \newcommand{\cube}[1]{\mathord{\square^{#1}}} 
  \newcommand{\obox}[2]{\mathord{\sqcap^{#1}_{#2}}} 
  \newcommand{\dfobox}[1][n]{\mathord{\sqcap^{#1}_{i,\varepsilon}}} 
  \newcommand{\pp}[1][\gprod]{\mathbin{\hat{#1}}} 

  \newcommand{\iobox}[2]{\widehat{\mathord{\sqcap}}^{#1}_{#2}} 
  \newcommand{\icube}[2]{\widehat{\mathord{\square}}^{#1}_{#2}} 

  \newcommand{\minm}[1]{{#1}^{\flat}}
  \newcommand{\maxm}[1]{{#1}^{\sharp}}
  \newcommand{\natm}[1]{{#1}^{\natural}}

\newcommand{\restr}[2]{{#1}\vert{#2}} 
\newcommand{\pbtick}{\lrcorner}
\newcommand{\potick}{\ulcorner}

  \DeclareFontFamily{U}{dmjhira}{}
  \DeclareFontShape{U}{dmjhira}{m}{n}{ <-> dmjhira }{}

  \newcommand{\adj}{\dashv}

  \newcommand{\Ho}{\operatorname{Ho}}

  \newcommand{\op}{{\mathord\mathrm{op}}}

  \DeclareMathOperator*{\push}{\cup}

  \newcommand{\colim}{\operatorname*{colim}}
  \newcommand{\hocolim}{\operatorname*{hocolim}}

  \newcommand{\Lan}{\operatorname{Lan}}

  \newcommand{\bd}{\partial}

  \newcommand{\gprod}{\otimes}

  \newcommand{\id}[1][]{\operatorname{id}_{#1}}
  
  \newcommand{\simp}[1]{\mathord\Delta^{#1}}

  \newcommand{\uvar}{\mathord{\relbar}}

  \renewcommand{\tilde}{\widetilde}

  \renewcommand{\hat}{\widehat}

  \newcommand{\cat}[1]{\mathscr{#1}}

  \newcommand{\ncat}[1]{\mathsf{#1}}

  \newcommand{\sSet}{\ncat{sSet}}

  \newcommand{\from}{\colon}

  \newcommand{\ito}{\hookrightarrow}

  \declaretheorem[style=definition,within=section]{definition}
  \declaretheorem[style=definition,numberlike=definition]{example}
  \declaretheorem[style=definition,numberlike=definition]{remark}

  \declaretheorem[style=plain,numberlike=definition]{corollary}
  \declaretheorem[style=plain,numberlike=definition]{lemma}
  \declaretheorem[style=plain,numberlike=definition]{proposition}
  \declaretheorem[style=plain,numberlike=definition]{theorem}

  \declaretheorem[style=plain,numbered=no,name=Theorem]{theorem*}
  \declaretheorem[style=plain,numbered=no,name=Conjecture]{conjecture*}

  \Crefname{corollary}{Corollary}{Corollaries}
  \Crefname{definition}{Definition}{Definitions}
  \Crefname{lemma}{Lemma}{Lemmas}
  \Crefname{proposition}{Proposition}{Propositions}
  \Crefname{remark}{Remark}{Remarks}
  \Crefname{theorem}{Theorem}{Theorems}
  \Crefname{conjecture}{Conjecture}{Conjectures}
  \Crefname{notation}{Notation}{Notations}
  \Crefname{convention}{Convention}{Conventions}

  \numberwithin{equation}{section}
  \numberwithin{figure}{section}
  \theoremstyle{plain}
  
  \theoremstyle{definition}
  
  \theoremstyle{remark}
  
  \theoremstyle{plain}
  
  \theoremstyle{plain}
  
  \theoremstyle{plain}

  \providecommand{\corollaryname}{Corollary}
  \providecommand{\definitionname}{Definition}
  \providecommand{\lemmaname}{Lemma}
  \providecommand{\propositionname}{Proposition}
  \providecommand{\remarkname}{Remark}
  \providecommand{\theoremname}{Theorem}

  \newlist{axioms}{enumerate}{1}
  \Crefname{axiomsi}{}{}

  \newenvironment{tikzeq*}
  {
    \begingroup
    \begin{equation*}
    \begin{tikzpicture}[baseline=(current bounding box.center)]
  }
  {
    \end{tikzpicture}
    \end{equation*}
    \endgroup
    \ignorespacesafterend
  }

  \tikzset
  {
    diagram/.style=
    {
      matrix of math nodes,
      column sep={4.3em,between origins},
      row sep={4em,between origins},
      text height=1.5ex,
      text depth=.25ex
    },
    over/.style={preaction={draw=white,-,line width=6pt}},
    every to/.style={font=\footnotesize},
    inj/.style={right hook->},
    surj/.style={-{Latex[open]}},
    cof/.style={>->},
    fib/.style={->>},
  }

  \DeclareFontFamily{U}{mathx}{\hyphenchar\font45}

  \DeclareFontShape{U}{mathx}{m}{n}{
    <5> <6> <7> <8> <9> <10>
    <10.95> <12> <14.4> <17.28> <20.74> <24.88>
    mathx10}{}

  \DeclareSymbolFont{mathx}{U}{mathx}{m}{n}

  \DeclareFontFamily{U}{mathb}{\hyphenchar\font45}

  \DeclareFontShape{U}{mathb}{m}{n}{
    <5> <6> <7> <8> <9> <10>
    <10.95> <12> <14.4> <17.28> <20.74> <24.88>
    mathb10}{}

  \DeclareSymbolFont{mathb}{U}{mathb}{m}{n}

  \DeclareMathSymbol{\Rsh}{\mathrel}{mathb}{"E9}

  \DeclareFontFamily{U}{MnSymbolA}{}

  \DeclareFontShape{U}{MnSymbolA}{m}{n}{
    <-6> MnSymbolA5
    <6-7> MnSymbolA6
    <7-8> MnSymbolA7
    <8-9> MnSymbolA8
    <9-10> MnSymbolA9
    <10-12> MnSymbolA10
    <12-> MnSymbolA12}{}

  \DeclareSymbolFont{MnSyA}{U}{MnSymbolA}{m}{n}

  \DeclareMathSymbol{\twoheaddownarrow}{\mathrel}{MnSyA}{27}

  \newcommand{\MSC}[1]{%
    \let\thempfn\relax
    \footnotetext[0]{2020 Mathematics Subject Classification: #1.}
  }
\usepackage{stackrel}
\usepackage{mleftright}
\mleftright
\DeclareMathSymbol{:}{\mathpunct}{operators}{"3A}
\newcommand{\notehelper}[3]{\textcolor{#3}{$\blacksquare$}\marginpar{\ifodd\thepage\raggedright\else\raggedleft\fi\color{#3}\tiny \textbf{#2:} #1}}

\tikzcdset{scale cd/.style={every label/.append style={scale=#1},
    cells={nodes={scale=#1}}}}
\usetikzlibrary{calc}
\usetikzlibrary{decorations.pathmorphing}
\usetikzlibrary{spath3}

\tikzset{curve/.style={settings={#1},to path={(\tikztostart)
    .. controls ($(\tikztostart)!\pv{pos}!(\tikztotarget)!\pv{height}!270:(\tikztotarget)$)
    and ($(\tikztostart)!1-\pv{pos}!(\tikztotarget)!\pv{height}!270:(\tikztotarget)$)
    .. (\tikztotarget)\tikztonodes}},
    settings/.code={\tikzset{quiver/.cd,#1}
        \def\pv##1{\pgfkeysvalueof{/tikz/quiver/##1}}},
    quiver/.cd,pos/.initial=0.35,height/.initial=0}

\tikzset{between/.style n args={2}{/tikz/spath/at end path construction={
    \tikzset{spath/split at keep middle={current}{#1}{#2}}
}}}

\tikzset{tail reversed/.code={\pgfsetarrowsstart{tikzcd to}}}
\tikzset{2tail/.code={\pgfsetarrowsstart{Implies[reversed]}}}
\tikzset{2tail reversed/.code={\pgfsetarrowsstart{Implies}}}
\tikzset{no body/.style={/tikz/dash pattern=on 0 off 1mm}}

\global\long\def\pr#1{\left(#1\right)}%
\global\long\def\CS{\mathsf{cSet}}%
\global\long\def\SS{\mathsf{sSet}}%
\global\long\def\Fun{\operatorname{Fun}}%
\global\long\def\hocolim{\operatorname{hocolim}}%
\global\long\def\msf#1{\mathsf{#1}}%
\global\long\def\mscr#1{\mathscr{{#1}}}%
\global\long\def\mcal#1{\mathcal{#1}}%
\global\long\def\mbb#1{\mathbb{#1}}%
\global\long\def\mbf#1{\mathbf{#1}}%
\global\long\def\mrm#1{\mathrm{#1}}%
\global\long\def\pr#1{\left(#1\right)}%
\global\long\def\hat#1{\widehat{#1}}%
\global\long\def\QC{\mathsf{QCat}}%
\global\long\def\curlyCat{\mathcal{C}\msf{at}}%
\global\long\def\id{\operatorname{id}}%
\global\long\def\Hom{\operatorname{Hom}}%
\global\long\def\ot{\leftarrow}%
\global\long\def\lim{\operatorname{lim}}%
\global\long\def\colim{\operatorname{colim}}%
\global\long\def\hocolim{\operatorname{hocolim}}%
\global\long\def\cc{\operatorname{cc}}%
\global\long\def\Rect{\operatorname{Rec}}%
\global\long\def\Side{\operatorname{Face}}%
\global\long\def\LV{\operatorname{LV}}%
\global\long\def\KQ{\text{Kan-Quillen}}%
\global\long\def\plain{\mathrm{plain}}%
\global\long\def\Joyal{\mathrm{Joyal}}%
\global\long\def\cov{\mathrm{cov}}%
\global\long\def\proj{\mathrm{proj}}%
\global\long\def\ev{\operatorname{ev}}%
\global\long\def\rcone{\triangleright}%
\global\long\def\S{\mathsection}%
\global\long\def\p{\prime}%
\global\long\def\bigadj{\stackrel[\longleftarrow]{\longrightarrow}{\bot}}%

\renewcommand{\colim}{\operatorname*{colim}}

\tikzstyle{vertex}=[circle, fill, minimum size=5pt, inner sep=0pt]
\tikzstyle{openvertex}=[circle, draw, fill=white, line width=0.85pt, opacity=0.6, minimum size=5pt, inner sep=0pt]
\tikzstyle{colorvertex}=[circle, draw, minimum size=6pt, inner sep=0pt]
\definecolor{vertex1}{RGB}{255,0,0}
\definecolor{vertex2}{RGB}{0,0,255}
\definecolor{vertex3}{RGB}{0,255,0}
\definecolor{vertex4}{RGB}{255,255,0}
\definecolor{vertex5}{RGB}{255,0,255}

\newcommand{\eps}{\varepsilon}

\renewcommand{\cat}[1]{\mathcal{#1}}

\newcommand{\ho}{{\mathrm{ho} \,}} 

\newcommand{\catname}[1]{\mathsf{#1}} 

\newcommand{\sRect}{\Rect_{\Delta}} 

\thanksmarkseries{arabic}

\begin{document}

\author{Kensuke Arakawa \and Daniel Carranza \and Krzysztof Kapulkin}
\title{Cubical models of $\infty$-presheaves and the Bousfield--Kan formula}

\maketitle

 \begin{abstract} 
   We construct the covariant and the cocartesian model structures on the slice categories of cubical sets and marked cubical sets, respectively.
   As an application, we derive a version of the Bousfield--Kan formula for arbitrary cofibrantly generated monoidal model categories satisfying Muro's axiom.
 \end{abstract}

\section*{Introduction}

Cubical sets have long played a central role in homotopy theory and higher category theory.
In homotopy theory, they were among the first combinatorial models for spaces, introduced by Kan \cite{kan:abstract-homotopy-i,kan:abstract-homotopy-ii} prior to his development of simplicial sets \cite{kan:css}.
Over the years, cubical sets re-emerged thanks to the ideas of Grothendieck \cite{grothendieck:pursuing-stacks}, the work of Cisinski \cite{cisinski:presheaves,cisinski:elegant}, Jardine \cite{jardine:a-beginning}, and Maltsiniotis \cite{maltsiniotis:test,maltsiniotis:connections} on test categories, and the introduction of a new cubical site with connections by Brown and Higgins \cite{brown-higgins:connections}, which provides a cubical version of the Dold--Kan correspondence, as well as a setting for which cubical groups are Kan complexes \cite{tonks:cubical-groups-kan}.
Most recently, this approach has seen a great deal of interest because of its seemingly endless applications to homotopy type theory, e.g., \cite{bazem-coquand-huber,cchm,cavallo-mortberg-swan,angiuli-et-al,coquand-huber-sattler}.
Cubical sets were also extensively studied as models of strict $\omega$-categories \cite{crans:thesis,al-agl-brown-steiner,brown-higgins,steiner:cubical}, before being used in higher category theory as new models of $(\infty, 1)$-categories \cite{doherty-kapulkin-lindsey-sattler,doherty:symmetry,curien-livernet-saadia}, e.g., cubical quasicategories and marked cubical quasicategories, as well as models of more general $(\infty, n)$-categories \cite{campion-kapulkin-maehara,doherty-kapulkin-maehara}, e.g., comical sets.

The importance of cubical methods lies in the following simple observation: while simplicies are built out of the interval by taking iterated \emph{joins}, cubes are built out of the interval by taking iterated \emph{products}.
Consequently, while simplicial methods excel in situations where one needs to consider joins, cubical sets provide a natural home to all product-like constructions.
The latter abound in higher category theory and include: the homotopy coherent nerve (since the mapping spaces in $\mathfrak{C}[n]$ are cubes) \cite{kapulkin-voevodsky}, the Gray tensor product \cite{campion-kapulkin-maehara}, and the cobar construction \cite{rivera-zeinalian:cubical-rigidification}, to name a few.

In this paper, we take an important step in the development of cubical higher categories by constructing two new model structures: the covariant and the cocartesian model structures on the slice categories of cubical sets and marked cubical sets, respectively.
Their simplicial analogs are of paramount importance in higher category theory, as the former models the $(\infty, 1)$-category of functors from the base $(\infty, 1)$-category $\cat{B}$ to $\infty$-groupoids, while the latter presents the $(\infty, 1)$-category of functors from $\cat{B}$ to $(\infty, 1)$-categories \cite{lurie:htt}.
The essential idea is that the fibers of the map $\cat{X} \to \cat{B}$ encode the values of a functor on a given object via the Grothendieck construction.
Although it is in principle possible to consider other models of these $(\infty, 1)$-categories, e.g., the exponential object in simplicial sets from $\cat{B}$ to the homotopy coherent nerve of Kan complexes/quasicategories, or simplicial functors from the rigidification $\mathfrak{C}\cat{B}$ of $\cat{B}$ to (the simplicial category of) Kan complexes/quasicategories \cite{lurie:htt}, the covariant and cocartesian model structures are much easier to work with instead.
In fact, the contravariant model structure, a closely related analog of the covariant model structure but for functors from $\cat{B}^\op$ to $\infty$-groupoids, appeared already in the seminal notes by Joyal \cite{joyal:theory-of-qcats} that established the early foundations of the subject of higher category theory.

Our first results can be summarized in the following theorem:
\begin{theorem*}[cf.\ \cref{prop:cov,prop:fiberwise_eq_cov,prop:cc,prop:fiberwise_eq_cc,TUcomparison_for_all}]\leavevmode
\begin{enumerate}
	\item The slice category $\cSet_{/B}$ of cubical sets carries the covariant model structure, whose cofibrations are monomorphisms, and whose fibrant objects are left fibrations with codomain $B$.
		Weak equivalences of fibrant objects are the fiberwise homotopy equivalences.
    This model structure is Quillen equivalent to the covariant model structure on simplicial sets via the slice adjunction induced by triangulation.
	\item The slice category $\mcSet_{/\maxm{B}}$ of marked cubical sets carries the cocartesian model structure, whose cofibrations are monomorphisms, and whose fibrant objects are cocartesian fibrations over $B$ with cocartesian edges marked. 
		Weak equivalences of fibrant objects are the fiberwise categorical equivalences.
    This model structure is Quillen equivalent to the cocartesian model structure on marked simplicial sets via the slice adjunction induced by triangulation.
\end{enumerate}
\end{theorem*}

We further investigate this model structure, proving several results:
\begin{enumerate}
  \item \emph{Straightening--unstraightening} over the nerve of a $1$-category.
  We construct a cubical version of the straightening functor over a $1$-category $\cat{C}$, establishing Quillen equivalences between $\Fun(\cat{C}, \cSet)$ with the projective model structure and the covariant model structure over $N\cat{C}$ (\cref{thm:str}), as well as between $\Fun(\cat{C},\cSet^{+})$ with the projective model structure and the cocartesian model structure over $N \cat{C}$ (\cref{thm:str_marked}).
  Note that these equivalences exist in addition to (and are distinct from) the Quillen equivalences given by composing triangulation with simplicial straightening.

  \item \emph{Invariance} of the covariant/cocartesian model structure with respect to a weak equivalence in the base, i.e., if $f \colon B \to B'$ is a weak categorical equivalence, then the base change adjunctions $f_! \adj f^*$
yield Quillen equivalences with respect to both the covariant and cocartesian model structures.
This is proven in \cref{thm:cat_inv,thm:cat_inv_cc} and constitutes the main topic of \cref{sec:invariance}.
These result are of crucial importance, especially in the cubical setting, since many important phenomena require working with slices over representables (e.g., the Cisinski--Nguyen construction of the universal left fibration \cite{cisinski-nguyen:universal}), but the representables are not cubical quasicategories.
They are also key ingredients in proving the equivalence with simplicial sets when the base is non-fibrant.
  \item \emph{Pullbacks of cocartesian fibrations} are homotopy pullbacks even if the base is not a cubical quasicategory (\cref{cor:cfib_hpb}).
  This mirrors an analogous result in the simplicial setting, showcasing that (co)cartesian fibrations correct the failure of categorical fibrations to guarantee right properness.
\end{enumerate} 

Moreover, the present paper highlights another feature of cubical methods that was known to and used by experts since the early beginnings of the subject, but was only recently formally observed by Lawson \cite{Law17}.
Namely, it is very easy to construct monoidal left Quillen functors out of the category of cubical sets.
Being a presheaf category, to define such a left adjoint, one only needs to specify its values on the representables.
If this left adjoint is to be monoidal, for cubes this is immediately reduced to the value on a single object (the interval).

The way we use this property in practice is by first deriving the cubical analog of the Bousfield--Kan formula \cite{bousfield-kan} for cofibrantly generated model categories equipped with a cubical action, and then, using the aforementioned property of cubical sets, we construct such an action on every monoidal model category satisfying Muro's axiom \cite{Mur15}.
In other words, we derive a variant of the Bousfield--Kan formulas for monoidal model categories, which is therefore applicable, for example, to the Joyal model structure on simplicial sets \cite{joyal:theory-of-qcats}, a case not covered by the original Bousfield--Kan formula.
This fact depends crucially on our ability to define monoidal left adjoints out of the category of cubical sets, and our methods would not work simplicially.

\begin{theorem*}[\cref{cor:BK_mon}]
Let $\mbf M$ be a combinatorial monoidal model category satisfying Muro's unit axiom, and let $I$ be a cylinder object for the monoidal unit.
Given a small diagram $F \colon \mcal C\to\mbf M$ taking values in the full
subcategory of cofibrant objects, the homotopy colimit of $F$ may be computed by the formula
\[
	\int^{[1]^n\in\square}I^{\otimes n}\otimes B_{n}\pr{\ast,\mcal C,F}\text{,}
\]
where $B_n$ denotes the cubical bar construction of \cref{sec:formula}.
\end{theorem*}

Let us now discuss the contents and methods of the paper in more detail.
We begin by reviewing basic facts about cubical sets and cubical quasicategories in \cref{sec:prelims}, largely following \cite{doherty-kapulkin-lindsey-sattler}.
We then start off our development in \cref{sec:ccfib_lfib} by defining left and cocartesian fibrations of cubical sets and establishing their basic properties.
In \cref{sec:model}, we construct model structures for left fibrations (the covariant model structure) and cocartesian fibrations (the cocartesian model structure) over a fixed base, using techniques developed recently by Nguyen \cite{Ngu19}. 
In \cref{sec:triangulation,sec:straightening}, we study these model structures when the base is a cubical quasicategory.
In particular, we show that: left fibrations and cocartesian fibrations of cubical quasicategories are equivalent to their quasicategorical counterparts; and the ordinary Grothendieck construction extends to a straightening--unstraightening
equivalence over the cubical nerve of any $1$-category.
In \cref{sec:invariance}, we show that the covariant and cocartesian model structures are invariant under categorical equivalence in the base, meaning that pullback along a weak categorical equivalence of cubical sets forms a Quillen equivalence of model categories.
In \cref{sec:applications}, we use results from \cref{sec:invariance} to prove generalize various theorems on left/cocartesian fibrations to the case of a non-fibrant base. 
We conclude in \cref{sec:formula} by establishing the cubical Bousfield--Kan formula and applying it to a wide class of monoidal model categories.

\textbf{A note on terminology.}
Throughout the paper, we use the term ``$\infty$-category'' as a synonym of ``quasicategory,''
but we also make use of the term ``quasicategory''.
The reason for this is that we use these terms to evoke different connotations: the former a model-free representation, while the latter a concrete model of $\pr{\infty,1}$-categories.
For example, if we wish to regard the nerve of a (1-)category $\mcal C$ as an $\infty$-category, we often suppress the nerve from its notation and simply write it
as $\mcal C$.
However, when we want to emphasize that the nerve of $\mcal C$ is a quasicategory, we keep the nerve notation.
Accordingly, we write $\Cat_{\infty}$ for the $\infty$-category of $\infty$-categories and $\QC$ for the $\infty$-category of quasicategories, even though
both of them are defined as the $\infty$-categorical localization of (the nerve of) the ordinary category $\msf{QCat}$ of (small) quasicategories at categorical equivalences.

\section{\label{sec:prelims}Preliminaries}
In this section, we recall some basic definitions and constructions of cubical sets, the most important of which is the \emph{cubical Joyal} model structure on cubical sets (and marked cubical sets), which models the homotopy theory of $(\infty, 1)$-categories \cite{doherty-kapulkin-lindsey-sattler}.
Throughout, we assume familiarity with simplicial sets, quasicategories, and the Joyal model structure.

We first define the cube category $\Box$.
The objects of $\Box$ are posets of the form $[1]^n = \{0 \leq 1\}^n$ and the maps are generated (inside the category of posets) under composition by the following four special classes:
\begin{itemize}
  \item \emph{faces} $\partial^n_{i,\varepsilon} \colon [1]^{n-1} \to [1]^n$ for $i = 1, \ldots , n$ and $\varepsilon = 0, 1$ given by:
  \[ \partial^n_{i,\varepsilon} (x_1, x_2, \ldots, x_{n-1}) = (x_1, x_2, \ldots, x_{i-1}, \varepsilon, x_i, \ldots, x_{n-1})\text{;}  \]
  \item \emph{degeneracies} $\sigma^n_i \colon [1]^n \to [1]^{n-1}$ for $i = 1, 2, \ldots, n$ given by:
  \[ \sigma^n_i ( x_1, x_2, \ldots, x_n) = (x_1, x_2, \ldots, x_{i-1}, x_{i+1}, \ldots, x_n)\text{;}  \]
  \item \emph{negative connections} $\gamma^n_{i,0} \colon [1]^n \to [1]^{n-1}$ for $i = 1, 2, \ldots, n-1$ given by:
  \[ \gamma^n_{i,0} (x_1, x_2, \ldots, x_n) = (x_1, x_2, \ldots, x_{i-1}, \max\{ x_i , x_{i+1}\}, x_{i+2}, \ldots, x_n) \text{.} \]
  \item \emph{positive connections} $\gamma^n_{i,1} \colon [1]^n \to [1]^{n-1}$ for $i = 1, 2, \ldots, n-1$ given by:
  \[ \gamma^n_{i,1} (x_1, x_2, \ldots, x_n) = (x_1, x_2, \ldots, x_{i-1}, \min\{ x_i , x_{i+1}\}, x_{i+2}, \ldots, x_n) \text{.} \]
\end{itemize}

These maps obey the following \emph{cubical identities}:

\[ \begin{array}{l l}
    \partial_{j, \varepsilon'} \partial_{i, \varepsilon} = \partial_{i+1, \varepsilon} \partial_{j, \varepsilon'} \quad \text{for } j \leq i; & 
    \sigma_j \partial_{i, \varepsilon} = \begin{cases}
        \partial_{i-1, \varepsilon} \sigma_j & \text{for } j < i; \\
        \id                                                       & \text{for } j = i; \\
        \partial_{i, \varepsilon} \sigma_{j-1} & \text{for } j > i;
    \end{cases} \\
    \sigma_i \sigma_j = \sigma_j \sigma_{i+1} \quad \text{for } j \leq i; &
    \gamma_{j,\varepsilon'} \gamma_{i,\varepsilon} = \begin{cases}
    \gamma_{i,\varepsilon} \gamma_{j+1,\varepsilon'} & \text{for } j > i; \\
    \gamma_{i,\varepsilon}\gamma_{i+1,\varepsilon} & \text{for } j = i, \varepsilon' = \varepsilon;
    \end{cases} \\
    \gamma_{j,\varepsilon'} \partial_{i, \varepsilon} = \begin{cases} 
        \partial_{i-1, \varepsilon} \gamma_{j,\varepsilon'}   & \text{for } j < i-1 \text{;} \\
        \id                                                         & \text{for } j = i-1, \, i, \, \varepsilon = \varepsilon' \text{;} \\
        \partial_{j, \varepsilon} \sigma_j         & \text{for } j = i-1, \, i, \, \varepsilon = 1-\varepsilon' \text{;} \\
        \partial_{i, \varepsilon} \gamma_{j-1,\varepsilon'} & \text{for } j > i;
    \end{cases} &
    \sigma_j \gamma_{i,\varepsilon} = \begin{cases}
        \gamma_{i-1,\varepsilon} \sigma_j  & \text{for } j < i \text{;} \\
        \sigma_i \sigma_i           & \text{for } j = i \text{;} \\
        \gamma_{i,\varepsilon} \sigma_{j+1} & \text{for } j > i \text{.} 
    \end{cases}
\end{array} \]

A \emph{cubical set} is a functor $\boxcat^\op \to \Set$.
The category of cubical sets is denoted by $\cSet$, and we use the term \emph{cubical map} to mean a morphism of cubical sets.
We adopt the standard convention of writing $X_n$ for the value of a cubical set $X$ at an object $[1]^n$.
Moreover, we write cubical operators on the right e.g.~given an $n$-cube $x \in X_n$ of a cubical set $X$, we write $x\face{}{1,0}$ for the $\face{}{1,0}$-face of $x$. 
We typically denote cubical sets using capital letters $X, Y, \dots$.
\begin{definition} \leavevmode
    \begin{enumerate}
        \item For $n \geq 0$, the \emph{combinatorial $n$-cube} $\cube{n}$ is the representable functor $\boxcat(-, [1]^n) \from \boxcat^\op \to \Set$.
        \item The \emph{boundary of the $n$-cube} $\bd \cube{n}$ is the subobject of $\cube{n}$ defined by
        \[ \bd \cube{n} := \bigcup\limits_{\substack{j=1,\dots,n \\ \eta = 0, 1}} \operatorname{Im} \face{}{j,\eta}. \]
\item given $i = 1, \dots, n$ and $\eps = 0, 1$, the $(i, \eps)$-open box  $\dfobox$ is the subobject of $\bd \cube{n}$ defined by
        \[ \dfobox := \bigcup\limits_{(j,\eta) \neq (i,\varepsilon)} \operatorname{Im} \face{}{j,\eta}. \]
    \end{enumerate}
\end{definition}

A \emph{marked cubical set} is a pair $(X, S)$ where $X$ is a cubical set and $S \subseteq X_1$ is a subset of 1-cubes which contains all degenerate 1-cubes.
We refer to elements of $S$ as \emph{marked 1-cubes}, or \emph{marked edges}.
A morphism of marked cubical sets is a cubical map which sends marked edges to marked edges.
We denote the category of marked cubical sets by $\mcSet$, and remark that this category is a reflective subcategory of a presheaf category \cite[§1.4]{doherty-kapulkin-lindsey-sattler}.
We may denote marked cubical sets as tuples $(X, S), (Y, T), \dots$ or using overline notation $\overline{X}, \overline{Y}, \dots$, leaving the set of marked edges implicit.

The forgetful functor $\mcSet \to \cSet$ admits both adjoints.
\[ \begin{tikzcd}[sep = large]
	\mcSet \ar[r, "U" description] & \cSet \ar[l, "{\minm{(-)}}"', bend right] \ar[l, "{\maxm{(-)}}", bend left]
\end{tikzcd} \]
The left adjoint is the \emph{minimal marking} functor, it sends a cubical set $X$ to its marking at degenerate 1-cubes $\minm{X} = (X, \degen{}{1}(X_0))$; the right adjoint is the \emph{maximal marking} functor, it sends a cubical set $X$ to its marking at all 1-cubes $\maxm{X} = (X, X_1)$.

The category of cubical sets admits a monoidal product given by the \emph{geometric product}, which is defined as a Day convolution.
That is, consider the functor $\uvar \gprod \uvar \from \boxcat \times \boxcat \to \boxcat$ on the cube category defined by $[1]^m \gprod [1]^n = [1]^{m+n}$.
The geometric product is defined as the left Kan extension in the diagram
\[ \begin{tikzcd}
    \boxcat \times \boxcat \ar[r, "\gprod"] \ar[d] & \boxcat \ar[r] & \cSet \\
    \cSet \times \cSet \ar[urr, "\gprod"']
\end{tikzcd} \]
This product is \emph{not} symmetric in general, but it is biclosed.
That being said, we only make use of the right closed structure.
Given a cubical set $X$, we write $\Fun\pr{X,-}$ for the right adjoint
of $- \otimes A$. 
More generally, given a commutative square
\[\begin{tikzcd}
	A & X \\
	B & Y
	\arrow[from=1-1, to=1-2]
	\arrow[from=1-1, to=2-1]
	\arrow[from=1-2, to=2-2]
	\arrow[from=2-1, to=2-2]
\end{tikzcd}\]
of cubical sets, we write $\Fun_{A//Y}\pr{B,X}$ for the fiber of
the map
\[
\Fun\pr{B,X}\to\Fun\pr{A,X}\times_{\Fun\pr{A,Y}}\Fun\pr{B,Y}
\]
over the vertex corresponding to the above square. When $A=\emptyset$,
we simply write $\Fun_{Y}\pr{B,X} = \Fun_{\emptyset//Y}\pr{-,-}$.


The following result gives an element-wise description of cubes in the geometric product, which we use implicitly going forward.
\begin{proposition}[{\cite[Prop.~1.24]{doherty-kapulkin-lindsey-sattler}}] \label{thm:gprod_cube}
    Let $X, Y$ be cubical sets.
    \begin{enumerate}
        \item For $k \geq 0$, the $k$-cubes of $X \gprod Y$ consist of all pairs $(x \in X_m, y \in Y_n)$ such that $m + n = k$, subject to the identification $(x\degen{}{m+1}, y) = (x, y\degen{}{1})$.
        \item For $x \in X_m$ and $y \in Y_n$, the faces, degeneracies, and connections of the $(m+n)$-cube $(x, y)$ are computed by
        \begin{align*}
            (x, y)\face{}{i,\varepsilon} &= \begin{cases}
                (x\face{}{i,\varepsilon}, y) & 1 \leq i \leq m \\
                (x, y\face{}{i-m, \varepsilon}) & m+1 \leq i \leq m+n;
            \end{cases} \\
            (x, y)\degen{}{i} &= \begin{cases}
                (x\degen{}{i}, y) & 1 \leq i \leq m+1 \\
                (x, y\degen{}{i-m}) & m+1 \leq i \leq m+n;
            \end{cases} \\
            (x, y)\conn{}{i,\varepsilon} &= \begin{cases}
                (x\conn{}{i,\varepsilon}, y) & 1 \leq i \leq m \\
                (x, y\conn{}{i-m,\varepsilon}) & m+1 \leq i \leq n. \tag*{\qedsymbol}
            \end{cases}
        \end{align*}
    \end{enumerate}
\end{proposition}

The geometric product naturally extends to a product on the category of marked cubical sets \cite[§1.4, Prop.~1.37]{doherty-kapulkin-lindsey-sattler}: the \emph{marked geometric product} of $(X, S)$ and $(Y, T)$ is the marking on the geometric product $X \gprod Y$ given by
\[ \{ (f, y) \mid f \in S, \ y \in Y_0 \} \cup \{ (x, g) \mid x \in X_0, \ g \in T \}. \]
Each functor in the adjoint triple $\minm{(-)} \adj U \adj \maxm{(-)}$ is strong monoidal with respect to the marked and unmarked geometric products.

Cubical sets are related to simplicial sets via the \emph{triangulation} functor, which is defined as the free cocompletion of the cocubical object $[1]^n \mapsto (\simp{1})^n$ in $\sSet$.
\[ \begin{tikzcd}[column sep = large]
	\boxcat \ar[r, "{[1]^n \mapsto (\simp{1})^n}"] \ar[d] & \sSet \\
	\cSet \ar[ur, "T"']
\end{tikzcd} \]
This functor is strong monoidal with respect to the geometric product of cubical sets and the cartesian product of simplicial sets \cite[Ex.~8.4.24]{cisinski:presheaves}.
The triangulation functor admits a right adjoint $U \from \sSet \to \cSet$, defined by
\[ (UX)_n := \sSet \left( (\simp{1})^n, X \right) . \] 

The triangulation adjunction ascends to an adjunction between marked cubical sets and marked simplicial sets; the \emph{marked triangulation} functor $T^+ \from \mcSet \to \msSet$ sends a marked cubical set $(X, S)$ to $(TX, S)$ (here, we identify 1-cubes of $X$ with 1-simplices of $TX$).
The right adjoint $U^+ \from \msSet \to \mcSet$ is defined by a similar formula $U^+(K, S) := (UK, S)$ (as before, we identify 1-simplices of $K$ with 1-cubes of $UK$).

Just as simplicial sets form a model for the theory of $(\infty, 1)$-categories via the Joyal model structure, cubical sets form a model of $(\infty, 1)$-categories via the \emph{cubical Joyal} model structure.
Before we describe this model structure, we first recall the definitions of \emph{cubical quasicategories} and categorical equivalences of cubical sets.
\begin{definition}[{\cite[Def.~1.21, before Def.~4.4]{doherty-kapulkin-lindsey-sattler}}] \label{def:qcat}
	Let $n \geq 2$, $i \in \{ 1, \dots, n \}$, and $\eps = \{ 0, 1 \}$.
	\begin{enumerate}
		\item The \emph{critical edge} of $\cube{n}$ is the unique 1-cube which intersects both the vertex $(1 - \eps, \dots, 1 - \eps)$ and the $\face{}{i,\eps}$-face.
		\item The \emph{$(i, \eps)$-inner open box} $\iobox{n}{i, \eps}$ is the quotient of the $(i, \eps)$-open box $\obox{n}{i, \eps}$ obtained by quotienting the critical edge to a 0-cube; the \emph{$(i, \eps)$-inner open box inclusion} is the corresponding inclusion map $\iobox{n}{i, \eps} \ito \icube{n}{i, \eps}$ of the $(i, \eps)$-inner open box into the quotient of $\cube{n}$ identifying the critical edge to a 0-cube
		\[ \begin{tikzcd}
			\cube{1} \ar[r, hook] \ar[d] \ar[rd, phantom, "\potick" very near end] & \obox{n}{i, \eps} \ar[r, hook] \ar[d] \ar[rd, phantom, "\potick" very near end] & \cube{n} \ar[d] \\
			\cube{0} \ar[r, hook] & \iobox{n}{i, \eps} \ar[r, hook] & \icube{n}{i, \eps}
		\end{tikzcd} \]
		\item A cubical map is \emph{inner anodyne} if it is in the saturation of the inner open box inclusions; it is an \emph{inner fibration} if it has the right lifting property against all inner open box inclusions (or equivalently, against all inner anodyne morphisms).
		\item A \emph{cubical quasicategory} is a cubical set such that the unique map to the terminal object $\cube{0}$ is an inner fibration.
	\end{enumerate}
\end{definition}
We often denote cubical quasicategories using calligraphic letters, e.g. $\mathcal{X}, \mathcal{Y}, \dots$; we also occasionally use the letters $\mathcal{C}, \mathcal{D}, \dots$, mimicking the standard notation for ordinary categories.

One thinks of the right lifting property against inner open boxes as analogous to ``special horn fillings'' in the simplicial setting, i.e.\ outer horn filling problems where a specific edge is sent to an equivalence (in this case, we use the stronger condition that a certain edge is sent to a \emph{degeneracy}).
This analogy will be reinforced when we define cartesian and cocartesian edges (and fibrations) in the setting of cubical quasicategories.

Recall that ordinary categories embed into (simplicial) quasicategories via the nerve functor $N \from \Cat \to \sSet$.
By post-composing with the right adjoint $U \from \sSet \to \cSet$, we obtain a full and faithful embedding of ordinary categories into \emph{cubical} quasicategories (\cref{cubical-nerve-full-faithful}).
We refer to this functor as the \emph{cubical nerve} of a 1-category.
\begin{definition}
	The \emph{cubical nerve} functor $N \from \Cat \to \cSet$ is the composite
	\[ \Cat \xrightarrow{N} \sSet \xrightarrow{U} \cSet. \]
	Explicitly, it is defined by
	\[ (N \mathcal{C})_n = \Cat ([1]^n, \mathcal{C}). \]
\end{definition}
As a composite of right adjoints, this functor admits a left adjoint, which we denote by $\Ho \from \cSet \to \Cat$. 
Given a cubical set $X$, we refer to $\Ho X$ as the \emph{homotopy category} of $X$.
For a general cubical set $X$, morphisms of its homotopy category $\Ho X$ are formal composites of 1-cubes in $X$, modulo an equivalence relation induced from the 2-cubes of $X$.
\begin{proposition} \label{cubical-nerve-full-faithful}
	The cubical nerve functor is full and faithful, and takes values in cubical quasicategories.
\end{proposition}
\begin{proof}
	The fact that $N$ takes values in quasicategories follows from \cite[Lem.~4.30]{doherty-kapulkin-lindsey-sattler}.
	It remains to show it is full and faithful.

	For any category $\mathcal{C}$, the counit $\Ho (N\cat{C}) \to \mathcal{C}$
	admits an inverse which sends a morphism $f \in \cat{C}$ to the morphism represented by $[1] \cong \Ho (\cube{1}) \xrightarrow{\Ho (f)} \Ho (N \mathcal{C})$.
	We verify this assignment respects composition since, given a composable pair of morphisms $g, f \in \cat{C}$, the 2-cube $\cube{2} \to N \cat{C}$ depicted as
	\[ \begin{tikzcd}
		c \ar[d, "gf"'] \ar[r, "f"] & d \ar[d, "g"] \\
		e \ar[r, equal] & e
	\end{tikzcd} \]
	witnesses that $\eps^{-1}(gf) = \eps^{-1}(g) \circ \eps^{-1}(f)$ in $\Ho (N \mathcal{C})$.
\end{proof}
We remark that, for a cubical quasicategory $\mathcal{X}$, every morphism in $\Ho \mathcal{X}$ admits a representative as a single 1-cube of $\mathcal{X}$ \cite[Lem.~2.23]{doherty-kapulkin-lindsey-sattler}.

\begin{definition} 
	A 1-cube $f$ in a cubical set $X$ is an \emph{equivalence} if the induced morphism $[f]$ in $\Ho X$ is an isomorphism.
\end{definition}
Note this definition of equivalence is more general than the notion of equivalence used in \cite[before Def.~4.4]{doherty-kapulkin-lindsey-sattler}, where an equivalence is a 1-cube which admits explicit 2-cubes witnessing the existence of a left and right inverse. 
More precisely, let $K$ denote the cubical set
\[ K = \begin{tikzcd}
	1 \ar[r] \ar[d, equal] & 0 \ar[d] \ar[r, equal] & 0 \ar[d, equal] \\
	1 \ar[r, equal] & 1 \ar[r] & 0
\end{tikzcd} \]
where the edge $0 \to 1$ is the $\face{}{2,1}$-face of the left square and the $\face{}{1,0}$-face of the right square, and $(\eps = 0)$-faces are always drawn on top (these indices are not a typo, see \cref{rem:faces-of-K}). 
We say a 1-cube $f$ in a cubical set $X$ is a \emph{$K$-equivalence} if the map $f \from \cube{1} \to X$ factors through the inclusion of the middle edge $0 \to 1$.
\begin{proposition} \label{ho-equiv-is-K-equiv}
	Let $f$ be a 1-cube in a cubical quasicategory $\mathcal{X}$.
	The following are equivalent:
	\begin{enumerate}
		\item $f$ is an equivalence;
		\item the map $f \from \cube{1} \to \mathcal{X}$ factors through the inclusion of the middle edge $\cube{1} \to K$.
	\end{enumerate}
\end{proposition}
\begin{proof}
	The implication (2) $\implies$ (1) is immediate.
	To show (1) $\implies$ (2), use \cite[Lem.~2.21]{doherty-kapulkin-lindsey-sattler}.
\end{proof}
We strengthen this result in \cref{all-equivs-agree}.
\begin{remark} \label{rem:faces-of-K}
	Note that, in our depiction of the cubical set $K$, the coordinate directions of the faces of each square disagree; i.e.\ the $i = 1$ faces of the left square are drawn on the top and bottom, whereas the $i = 1$ faces of the right square are drawn on the left and right.
	While it would be more intuitive to define the middle edge $0 \to 1$ to be the $\face{}{1,1}$-face of the left square instead (so that the coordinate directions are consistent in the drawing), this would violate \cite[Lem.~6.14]{doherty-kapulkin-lindsey-sattler} that $K$ is the image of the simplicial interval $J$ under the functor $Q$ (we define this functor before \cref{Q-int-equiv-marked-Joyal}).
	That said, this result is only used in the proof of \cite[Lem.~6.11]{doherty-kapulkin-lindsey-sattler} to show the inclusion $\cube{0} \to K$ is an acyclic cofibration, which is true for any (reasonable) definition of $K$.
\end{remark}


As in the Joyal model structure on simplicial sets, the notion of categorical equivalence between cubical sets is detected by ``homotopy'' classes of maps into a cubical quasicategory.
\begin{definition}[{cf.\ \cite[Def.~4.7]{doherty-kapulkin-lindsey-sattler}}]
	For cubical maps $f, g \from X \to Y$, an (elementary left) \emph{homotopy} from $f$ to $g$ is a map $H \from \cube{1} \otimes X \to Y$ such that $\restr{H}{\{ 0 \} \gprod X} = f$, $\restr{H}{\{ 1 \} \gprod X} = g$, and for any 0-cube $x \in X_0$, the 1-cube $\restr{H}{\cube{1} \gprod \{ x \}} \in Y_1$ is an equivalence.
\end{definition}
When $Y$ is a cubical quasicategory, homotopies between maps $X \to Y$ can be ``composed'', yielding the following:
\begin{proposition}[{\cite[Cor.~2.16]{doherty-kapulkin-lindsey-sattler}}]
	If $\mathcal{Y}$ is a quasicategory then the notion of homotopy is an equivalence relation on the set of maps $X \to \mathcal{Y}$. \qed
\end{proposition}

We now describe the cubical Joyal model structure on $\cSet$.
For a quasicategory $\mathcal{Y}$, we write $[X, \mathcal{Y}]$ for the set of homotopy classes of maps $X \to \mathcal{Y}$.
\begin{theorem}[{c.f.\ \cite[Thm.~4.2, Thm.~4.16, Prop.~4.22 \& Cor.~4.11]{doherty-kapulkin-lindsey-sattler}}]
	The category of cubical sets $\cSet$ admits a model structure, called the \emph{cubical Joyal model structure}, whose:
	\begin{itemize}
		\item cofibrations are the monomorphisms;
		\item fibrant objects are the quasicategories; 
		\item weak equivalences are the maps $f \from X \to Y$ such that, for any cubical quasicategory $\mathcal{Z}$, the induced map $f^* \from [Y, \mathcal{Z}] \to [X, \mathcal{Z}]$ is a bijection.
	\end{itemize}
	Moreover, this model structure is monoidal with respect to the geometric product. \qed
\end{theorem}

There is an analog of this model structure on the category of \emph{marked} cubical sets, where weak equivalences are detected using \emph{marked homotopies}.
\begin{definition}[{cf.\ \cite[Def.~2.14]{doherty-kapulkin-lindsey-sattler}}]
	For marked cubical maps $f, g \from \overline{X} \to \overline{Y}$, an (elementary left) \emph{marked homotopy} from $f$ to $g$ is a map $H \from \maxm{(\cube{1})} \otimes \overline{f} \to \overline{g}$ such that $\restr{H}{\{ 0 \} \gprod X} = f$, and $\restr{H}{\{ 1 \} \gprod X} = g$.
\end{definition}
Note the condition that the 1-cube $\restr{H}{\cube{1} \gprod \{ x \}}$ is an equivalence has been replaced with the condition that it be a marked edge.
\begin{proposition}[{\cite[Cor.~2.16]{doherty-kapulkin-lindsey-sattler}}]
	If $\overline{\mathcal{Y}}$ is a quasicategory marked at equivalences then the notion of homotopy is an equivalence relation on the set of maps $\overline{X} \to \overline{\mathcal{Y}}$. \qed
\end{proposition}
Given a cubical quasicategory $\mathcal{Y}$, we will write $\natm{\mathcal{Y}}$ for the marking of $\mathcal{Y}$ at equivalences.
For a quasicategory $\mathcal{Y}$, we write $[\overline{X}, \natm{\mathcal{Y}}]$ for the set of homotopy classes of maps $\overline{X} \to \natm{\mathcal{Y}}$.

\begin{theorem}[{c.f.\ \cite[Thm.~2.44, Prop.~2.36, Prop.~4.3 \& Cor.~2.49]{doherty-kapulkin-lindsey-sattler}}] \label{marked-cset-quillen-equiv}
	The category of marked cubical sets $\mcSet$ admits a model structure whose:
	\begin{itemize}
		\item cofibrations are the monomorphisms;
		\item fibrant objects are the quasicategories marked at equivalences;
		\item weak equivalences are the maps $f \from \overline{X} \to \overline{Y}$ such that, for any cubical quasicategory $Z$, the induced map $f^* \from [\overline{Y}, \natm{Z}] \to [\overline{X}, \natm{Z}]$ is a bijection.
	\end{itemize}
	Moreover, this adjunction is monoidal with respect to the marked geometric product, and the adjunction
	\[ \begin{tikzcd}
		\cSet \ar[r, bend left, "{\minm{(-)}}"{name=Upper}] & \mcSet \ar[l, bend left, "U"{name=Lower}] \ar[from=Lower, to=Upper, phantom, "\perp"]
	\end{tikzcd} \]
	is a Quillen equivalence. \qed
\end{theorem}

We use the following result to exhibit fibrations in the marked model structure.
\begin{proposition}
	\label{prop:catfib}Let $p \from \mcal C\to\mcal D$ be an inner fibration
	of cubical quasicategories. The following conditions are equivalent:
	\begin{enumerate}
	\item The map $p$ is a fibration in the cubical Joyal model structure.
	\item The map $p^{\natural} \from \mcal C^{\natural}\to\mcal D^{\natural}$ is
	a fibration in the model structure for marked cubical quasicategories.
	\item The map $p^{\natural} \from \mcal C^{\natural}\to\mcal D^{\natural}$ has
	the right lifting property for the inclusion $\{0\}^{\sharp}\subset\pr{\square^{1}}^{\sharp}$.
	\item The map $p^{\natural} \from \mcal C^{\natural}\to\mcal D^{\natural}$ has
	the right lifting property for the inclusion $\{1\}^{\sharp}\subset\pr{\square^{1}}^{\sharp}$.
	\end{enumerate}
\end{proposition}

\begin{proof}
The implication (2)$\implies$(1) follows from \cref{marked-cset-quillen-equiv}. 
The implications (2)$\implies$(3), (4) are straightforward. 
Therefore, it
suffices to show that each of the conditions (1), (3), (4) implies
(2).

Suppose first that condition (1) is satisfied. We must show that the
map $p^{\natural}$ is a fibration in the model structure for marked
cubical quasicategories. By \cite[Prop.~2.50]{doherty-kapulkin-lindsey-sattler}, this
is equivalent to the condition that $p^{\natural}$ has the right
lifting property against the following maps:

\begin{enumerate}[label=(\roman*)]

\item The inclusion $\pr{\sqcap_{i,\varepsilon}^{n}}^{!}\subset\pr{\square_{i,\varepsilon}^{n}}^{!}$
for $n\geq2$, $1\leq i\leq n$, and $\varepsilon\in\{0,1\}$, where
the exclamation signs indicate marking the critical edge for the $\pr{i,\varepsilon}$-face.

\item The inclusion $K^{\flat}\subset K^{!}$, where the only nondegenerate
marked edge of $K^{!}$ is the middle edge in the diagram
\[\begin{tikzcd}
	1 & 0 & 0 \\
	1 & 1 & 0.
	\arrow[from=1-1, to=1-2]
	\arrow[equal, from=1-1, to=2-1]
	\arrow[equal, from=1-2, to=1-3]
	\arrow[from=1-2, to=2-2]
	\arrow[equal, from=1-3, to=2-3]
	\arrow[equal, from=2-1, to=2-2]
	\arrow[from=2-2, to=2-3]
\end{tikzcd}\]

\item The inclusions of the form
\[
\pr{\square^{2},\{\text{all edges}\}\setminus\{\text{a single nondegenerate edge}\}}\subset\pr{\square^{2}}^{\sharp}.
\]

\end{enumerate}

Lifting properties against (ii) and (iii) are clear. 
For $n > 1$, the map $p^{\natural}$ has the right lifting property for the maps in (i) by \cite[Lem.~4.14]{doherty-kapulkin-lindsey-sattler}. 
It thus remains to show that $p^{\natural}$
has the right lifting property for the maps $\{0\}^{\sharp}\subset\pr{\square^{1}}^{\sharp}$
and $\{1\}^{\sharp}\subset\pr{\square^{1}}^{\sharp}$, but this is
immediate from the fact that $p$ has the right lifting property for
the endpoint inclusions $\square^{0}\hookrightarrow K$ \cite[Thm.~4.16]{doherty-kapulkin-lindsey-sattler}.
Thus we have proved (1)$\implies$(2).

Next, suppose that condition (3) is satisfied. Again, we must show
that $p^{\natural}$ has the right lifting property for the maps in
(i), (ii), and (iii). The only nontrivial part is to show that $p^{\natural}$
has the right lifting property for the inclusion $\{1\}^{\sharp}\subset\pr{\square^{1}}^{\sharp}$.
So suppose we are given a lifting problem
\[\begin{tikzcd}
	{\{1\}^\sharp} & {\mathcal{C}^\natural} \\
	{(\square ^1)^\sharp} & {\mathcal{D}^\natural.}
	\arrow["{c'}", from=1-1, to=1-2]
	\arrow[from=1-1, to=2-1]
	\arrow[from=1-2, to=2-2]
	\arrow[dotted, from=2-1, to=1-2]
	\arrow["f"', from=2-1, to=2-2]
\end{tikzcd}\]Let us depict $f$ as $d\to d'=p\pr{c'}$. We wish to lift $f$ to
an equivalence $c\to c'$. To this end, use condition (3) to lift
an inverse $f^{-1} \from d'\to d$ of $f$ to an equivalence $\widetilde{f^{-1}} \from c'\to c$
in $\mcal C$. Using \cite[Lem.~4.14]{doherty-kapulkin-lindsey-sattler}, we can fill the
left-hand $2$-cube below in such a way that it lifts the right-hand
$2$-cube: 
\[\begin{tikzcd}
	c & {c'} & d & {d'} \\
	c & c & d & d.
	\arrow[dotted, from=1-1, to=1-2]
	\arrow[equal, from=1-1, to=2-1]
	\arrow["{\widetilde{f^{-1}}}", from=1-2, to=2-2]
	\arrow["f", from=1-3, to=1-4]
	\arrow[equal, from=1-3, to=2-3]
	\arrow["{f^{-1}}", from=1-4, to=2-4]
	\arrow[equal, from=2-1, to=2-2]
	\arrow[equal, from=2-3, to=2-4]
\end{tikzcd}\]The map $c\to c'$ is the desired lift of $f$. Thus we have proved
(3)$\implies$(2). The proof of (4)$\implies$(2) is similar, and
we are done.
\end{proof}

In simplicial sets, quasicategories are defined as those simplicial sets with the right lifting property against inner horn inclusions; if a simplicial set has the right lifting property with respect to \emph{all} horn inclusions then it is a Kan complex.
Joyal's special horn filling lemma \cite[Thm.~1.3]{joyal:qcat-kan} implies that this condition is equivalent to being a quasicategory where all morphisms are equivalences \cite[Cor.~1.4]{joyal:qcat-kan}.

The analogous statements are true for cubical quasicategories as well.
\begin{definition}
	A \emph{Kan fibration} is a map with the right lifting property against all open box inclusions.
	A \emph{Kan complex} is a cubical set such that the unique map $X \to \cube{0}$ is a Kan fibration.
\end{definition}
\begin{proposition}[{\cite[Lem.~4.14]{doherty-kapulkin-lindsey-sattler}}] \label{prop:special-horn-lemma}
	Let $f \from \mathcal{X} \to \mathcal{Y}$ be an inner fibration between cubical quasicategories.
	Given a commutative square
	\[ \begin{tikzcd}
		\dfobox \ar[r] \ar[d,hook] & \mathcal{X} \ar[d, "f"] \\
		\cube{n} \ar[r] & \mathcal{Y}
	\end{tikzcd} \]
	if the top morphism sends the critical edge to an equivalence in $\mathcal{X}$ then this square admits a lift. \qed
\end{proposition}
\begin{corollary} \label{cor:qcat-is-kan}
	Let $\mathcal{X}$ be a cubical quasicategory.
	Then, $\mathcal{X}$ is a Kan complex if and only if every edge in $\mathcal{X}$ is an equivalence. \qed
\end{corollary}
Cubical Kan complexes are the fibrant objects in the \emph{Grothendieck} model structure on cubical sets, which models the homotopy theory of spaces.
\begin{theorem}[{Cisinski, cf.~\cite[Thm.~1.34]{doherty-kapulkin-lindsey-sattler}}]
	The category of cubical sets $\cSet$ admits a model structure, called the \emph{Grothendieck model structure}, whose:
	\begin{itemize}
		\item cofibrations are the monomorphisms;
		\item fibrations are the Kan fibrations;
		\item weak equivalences are the maps $f \from X \to Y$ such that, for any Kan complex $Z$, the induced map $f^* \from [Y, Z] \to [X, Z]$ is a bijection. \qed
	\end{itemize}
\end{theorem}
We use the Grothendieck model structure to give a strengthening of \cref{ho-equiv-is-K-equiv}.
Here, we write $E[1]$ for the (cubical) nerve of the contractible category on two objects.
\begin{proposition} \label{all-equivs-agree}
	Let $f$ be a 1-cube in a cubical quasicategory $\mathcal{X}$.
	The following are equivalent:
	\begin{enumerate}
		\item $f$ is an equivalence;
		\item the map $f \from \cube{1} \to \mathcal{X}$ factors through the inclusion of the middle edge $\cube{1} \to K$;
		\item the map $f \from \cube{1} \to \mathcal{X}$ factors through either inclusion $\cube{1} \to E[1]$.
	\end{enumerate}
\end{proposition}
\begin{proof}
	The equivalence of (1) and (2) is \cref{ho-equiv-is-K-equiv}.
	It is clear that (3) implies (1), so it remains to show that (1) implies (3). 

	Let $\operatorname{Core}(\mathcal{X})$ denote the maximal cubical subset of $\mathcal{X}$ spanned by 1-cubes which are equivalences.
	One verifies that $\operatorname{Core}(\mathcal{X})$ is again a quasicategory, and that every 1-cube is an equivalence.
	By \cref{cor:qcat-is-kan}, $\operatorname{Core}(\mathcal{X})$ is a Kan complex, hence fibrant in the Grothendieck model structure on cubical sets.
	If (1) holds then $f$ factors through $\operatorname{Core}(\mathcal{X})$, so it suffices to show that the inclusion $\cube{1}\ito E[1]$ is a trivial cofibration in the Grothendieck model structure. But this is clear, as both $\cube{1}$ and $E[1]$ are weakly equivalent to $\square^0$.  
\end{proof}

Both the triangulation and marked triangulation functors induce Quillen equivalences wih the Joyal model structure on $\sSet$ and the marked model structure on $\msSet$, respectively.
\begin{theorem}[{\cite[Thm.~6.1 \& Thm.~6.26]{doherty-kapulkin-lindsey-sattler}}] \label{triangulation-quillen-equiv}
	The triangulation adjunction
	\[ \begin{tikzcd}
		\cSet \ar[r, bend left, "T"{name=Upper}] & \sSet \ar[l, bend left, "U"{name=Lower}] \ar[from=Lower, to=Upper, phantom, "\perp"]
	\end{tikzcd} \]
	is a Quillen equivalence between both:
	\begin{enumerate}
		\item the cubical Joyal model structure and the Joyal model structure; as well as
		\item the Grothendieck model structure and the Kan--Quillen model structure. \qed
	\end{enumerate}
\end{theorem}
\begin{theorem} \label{prop:T^+U^+}
The adjunction
\[ \begin{tikzcd}
	\mcSet \ar[r, bend left, "T^+"{name=Upper}] & \msSet \ar[l, bend left, "U^+"{name=Lower}] \ar[from=Upper, to=Lower, phantom, "\perp"]
\end{tikzcd} \]
is a Quillen equivalence.
\end{theorem}
\begin{proof}
\cref{prop:catfib} and \cite[Prop.~7.15]{JT07}
show that the adjunction is a Quillen adjunction. 
Applying 2-out-of-3 on the diagram:
\[ \begin{tikzcd}
	\cSet \ar[r, "T"] \ar[d, "\minm{(-)}"'] & \sSet \ar[d, "\minm{(-)}"] \\
	\mcSet \ar[r, "T^+"] & \msSet
\end{tikzcd} \]
the claim then follows from \cref{marked-cset-quillen-equiv,triangulation-quillen-equiv}, and \cite[Prop.~3.1.5.3]{lurie:htt}.
\end{proof}

In the remaining sections, we will frequently make use of the induced Quillen equivalences on slice categories.
We record these adjunctions here, and reference them implicitly throughout the remainder of the paper.
\begin{corollary} \leavevmode
	\begin{enumerate} 
		\item For a simplicial set $K$, the functor $U \from \sSet \to cSet$ induces a Quillen adjunction 
		\[ \begin{tikzcd}
			\cSet_{/ UK} \ar[r, bend left, ""{name=Upper}] & \sSet_{/ K} \ar[l, bend left, "U"{name=Lower}] \ar[from=Upper, to=Lower, phantom, "\perp"]
		\end{tikzcd} \]
		between both the Grothendieck and Kan--Quillen model structures, as well as the cubical Joyal and Joyal model structures.
		The left adjoint sends $X \to UK$ to the composite $TX \to TUK \xrightarrow{\eps_K} K$.
		Moreover, if $K$ is fibrant then this is a Quillen equivalence.
	\item For a marked simplicial set $\overline{K}$, the functor $U \from \msSet \to \mcSet$ induces a Quillen adjunction  
		\[ \begin{tikzcd}
			\mcSet_{/ U^+ \overline{K}} \ar[r, bend left, ""{name=Upper}] & \sSet_{/ \overline{K}} \ar[l, bend left, "U^+"{name=Lower}] \ar[from=Upper, to=Lower, phantom, "\perp"]
		\end{tikzcd} \]
		whose left adjoint sends $\overline{X} \to U^+ \overline{K}$ to the composite $T^+ \overline{X} \to T^+U^+ \overline{K} \xrightarrow{\eps_{\overline{K}}} \overline{K}$.
		Moreover, if $K$ is fibrant then this is a Quillen equivalence.
	\end{enumerate}
\end{corollary}

In addition to the triangulation adjunction between simplicial sets and cubical sets, there exists an adjunction $Q \adj \int$ going ``in the other direction,'' that was first defined in \cite{kapulkin-lindsey-wong}.
More precisely, there is a monad on $\cSet$ \cite[Def.~5.1]{doherty-kapulkin-lindsey-sattler} whose underlying endofunctor $C \from \cSet \to \cSet$ is defined by the pushout
\[ CX \ := \ \colim \left( \begin{tikzcd}
	{X} \ar[r] \ar[d, "{\face{}{1,1} \gprod \id_{X}}", swap] & {\cube{0}} \\
	{\cube{1} \gprod X}
\end{tikzcd} \right) \]
To give some intuition for the construction: the object $\cube{1} \gprod X$ forms a cylinder object on $X$ in the Grothendieck model structure; the pushout $CX$ identifies one end of the cylinder with a point, thus $CX$ models a cone over $X$.

Instantiating at $X = \varnothing$, this monad gives rise to an augmented cosimplicial object $[n] \mapsto C^{n+1} \varnothing$, which we restrict to a cosimplicial object $[n] \mapsto C^{n+1} \varnothing = C^n \cube{0}$.
We denote this cosimplicial object by $Q[-] \from \Delta \to \cSet$ and observe that, since $\cube{1} \gprod -$ sends cubes to cubes, each object $Q[n]$ admits a map $\cube{n} \to Q[n]$.
This map is an epimorphism, hence $Q[n]$ may be regarded as a quotient of $\cube{n}$.
Via this quotient \cite[Prop.~6.3]{doherty-kapulkin-lindsey-sattler}, the simplicial face operator $\face{}{i}$ may be seen as a cubical face inclusion (either $\face{}{n,1}$ when $i = 0$, or $\face{}{n-i+1, 0}$ for all other $i$), and the simplicial degeneracy operator $\degen{}{i}$ arises as either a cubical degeneracy ($\degen{}{n}$ if $i = 0$) or a negative connection ($\conn{}{n-i, 0}$ for all other $i$).

We use this cosimplicial object to define a pair of adjoint functors
\[ \begin{tikzcd}
	\sSet \ar[r, bend left, "Q"{name=Upper}] & \cSet \ar[l, bend left, "\int"{name=Lower}] \ar[from=Lower, to=Upper, phantom, "\perp"]
\end{tikzcd} \]
where $Q \from \sSet \to \cSet$ is the extension by colimits and $\int \from \cSet \to \sSet$ is defined by
\[ ({\textstyle \int \! X})_n \ := \cSet \big( Q[n], X \big) . \]

As $Q[1] \cong \cube{1}$, we may identify 1-simplices of a simplicial set $K$ with 1-cubes of $QK$, and similarly 1-cubes of a cubical set $X$ may be identified with 1-simplices of $\int \! X$ (more precisely, there is a natural isomorphism between their underlying quivers).
We use this fact to ``upgrade'' to an adjunction
\[ \begin{tikzcd}
	\msSet \ar[r, bend left, "Q^+"{name=Upper}] & \mcSet \ar[l, bend left, "\int^+"{name=Lower}] \ar[from=Lower, to=Upper, phantom, "\perp"]
\end{tikzcd} \]
between marked simplicial sets and marked cubical sets.
The left adjoint is defined by $Q^+(K, S) = (QK, S)$ and the right adjoint is defined by $\int^+ \! (X, S) = (\int \! X, S)$.

\begin{proposition} \label{Q-int-equiv-marked-Joyal}
	The adjunctions
	
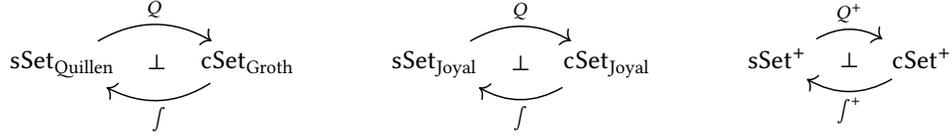
\[ \begin{tikzcd}
		\sSet_{\mathrm{Quillen}} \ar[r, bend left, "Q"{name=UpperT}] & \cSet_{\mathrm{Groth}} \ar[l, bend left, "\int"{name=LowerT}] \ar[from=LowerT, to=UpperT, phantom, "\perp"] &[1ex] \sSet_{\mathrm{Joyal}} \ar[r, bend left, "Q"{name=UpperT}] & \cSet_{\mathrm{Joyal}} \ar[l, bend left, "\int"{name=LowerT}] \ar[from=LowerT, to=UpperT, phantom, "\perp"] &[1ex] \msSet \ar[r, bend left, "Q^+"{name=Upper}] & \mcSet \ar[l, bend left, "\int^+"{name=Lower}] \ar[from=Lower, to=Upper, phantom, "\perp"]
	\end{tikzcd} \]
	are Quillen equivalences.
\end{proposition}
\begin{proof}\let\qed\relax 
	For the unmarked adjunctions, this is \cite[Thm.~6.23 \& Prop.~6.25]{doherty-kapulkin-lindsey-sattler}.
	For the marked adjunction, this follows from 2-out-of-3 on the diagram
	\[ \begin{tikzcd}
		\msSet \ar[r, "Q^+"] \ar[d, "U"'] & \mcSet \ar[d, "U"] \\
		\sSet \ar[r, "Q"] & \sSet
	\end{tikzcd} \tag*{\qedsymbol} \]
\end{proof}
\begin{proposition}[{cf.\ \cite[Thm.~6.10]{doherty-kapulkin-lindsey-sattler}}] \label{Q-is-full-faithful}
	The functors $Q$ and $Q^+$ are full and faithful. \qed
\end{proposition}

\section{\label{sec:ccfib_lfib}Cubical Cocartesian fibrations and left fibrations}

In this section, we introduce cocartesian fibrations of cubical sets
and the closely related notion of left fibrations of cubical sets.

\subsection{Cocartesian fibrations of cubical sets}

In this subsection, we define cocartesian fibrations of cubical sets.
Roughly speaking, they are maps of cubical sets whose fibers are cubical
quasicategories equipped with an action of the base. The goal of this
section is to make this precise, starting from a combinatorial definition
of cocartesian fibrations.
\begin{definition}
\label{def:cc}Let $p \from X\to Y$ be a map of cubical sets. An edge $e \from \square^{1}\to X$
of $X$ is called \emph{$p$-cocartesian} if every solid commutative
diagram of the form
\begin{equation}\label{d:cocart}
\begin{tikzcd}
	{\square^1} & {\sqcap^n_{i,1}} & X \\
	& {\square ^n} & Y
	\arrow["{e_{i,1}}"', from=1-1, to=1-2]
	\arrow["e", curve={height=-18pt}, from=1-1, to=1-3]
	\arrow[from=1-2, to=1-3]
	\arrow[from=1-2, to=2-2]
	\arrow["p", from=1-3, to=2-3]
	\arrow[dotted, from=2-2, to=1-3]
	\arrow[from=2-2, to=2-3]
\end{tikzcd}
\end{equation}admits a filler, where $n\geq2$, $1\leq i\leq n$, and $e_{i,1}$
denotes the critical edge for the $\pr{i,1}$-face. We say that $p$
is a \emph{cocartesian fibration} if it is an inner fibration and, for each vertex $x\in X$ and each edge $f \from p\pr x\to y$ in $Y$, there is a $p$-cocartesian edge $x\to x'$ lying over $f$.

A cocartesian fibration of cubical sets $p \from X\to Y$ is called a \emph{left
fibration} if its fibers are cubical Kan complexes. (By \cref{prop:shadow} below, this is equivalent to the condition that
every edge of $X$ be $p$-cocartesian.)
\end{definition}

Let $p \from X\to Y$ be a cocartesian fibration of cubical quasicategories.
Since $p$ is an inner fibration, its fiber $X_{y}=X\times_{Y}\{y\}$
is a cubical quasicategory for every $y\in Y$. We defined cocartesian
fibrations so that the fibers $\{X_{y}\}_{y\in Y}$ behave ``functorially''
in terms of the edges of $Y$. More precisely, for each edge $f \from y_{0}\to y_{1}$
in $Y$, we have a functor $f_{!} \from X_{y_{0}}\to X_{y_{1}}$ of cubical
quasicategories, and the association $f\mapsto f_{!}$ is coherently
functorial in an appropriate sense. To construct the functor $f_{!}$,
we need the following lemma.
\begin{lemma}
\label{lem:cocart_natural_transformation}Let $p \from X\to Y$ be a cocartesian
fibration of cubical sets, let $i \from A\to B$ and $i' \from A'\to B'$ be monomorphisms
of cubical sets, and suppose we are given a commutative diagram
\[\begin{tikzcd}
	{(A\otimes \square^n\otimes B')\cup (B\otimes \sqcap^{n}_{i,1}\otimes B')\cup (B\otimes \square^n\otimes A')} & X \\
	{B\otimes \square ^n\otimes B'} & Y,
	\arrow["f", from=1-1, to=1-2]
	\arrow[from=1-1, to=2-1]
	\arrow["p", from=1-2, to=2-2]
	\arrow[dotted, from=2-1, to=1-2]
	\arrow[from=2-1, to=2-2]
\end{tikzcd}\]where $n\geq1$ and $1\leq i\leq n$. Assume that for each pair of
vertices $\pr{x,y}\in\pr{A\times B'}\cup\pr{B\times A'}$, the map
$f\vert\{x\}\otimes\square^{n}\otimes\{y\}$ carries the critical
edge $e_{i,1}$ to a $p$-cocartesian edge. Then there is a
filler such that, for each $\pr{b,b'}\in B\times B'$, the restriction
$\{b\}\otimes\square^{n}\otimes\{b'\}$ carries the cocartesian edge
$e_{i,1}$ to a $p$-cocartesian morphism. Moreover, the full cubical
subset of $\Fun_{\pr{A\otimes\square^{1}\otimes B'}\cup\pr{B\otimes\{0\}\otimes B'}\cup\pr{B\otimes\square^{1}\otimes A'}//Y}\pr{B\otimes\square^{1}\otimes B',X}$
spanned by such fillers is a contractible cubical Kan complex.
\end{lemma}

\begin{proof}
For the first assertion, working cube by cube, we may assume that
$i$ and $j$ are one of the boundary inclusions of the standard cubes.
In this case, the left vertical arrow has the form $\sqcap_{i,1}^{n}\to\square^{n}$
for some $n\geq1$ and $1\leq i\leq n,$ so the claim is immediate
from the definition of $p$-cocartesian morphisms and cocartesian
fibrations. For the second assertion, we observe that the cubical
subset in question has the extension property for the boundary inclusions,
which follows from the first assertion.
\end{proof}
\begin{definition}
\label{def:cocart_transport}Let $p \from X\to Y$ be a cocartesian fibration
of cubical sets, and let $f \from y_{0}\to y_{1}$ be an edge of $Y$. We
define the \emph{cocartesian transport} functor $f_{!} \from X_{y_{0}}\to X_{y_{1}}$
as follows: By \cref{lem:cocart_natural_transformation}, we
can find a dotted filler
\[\begin{tikzcd}
	{X_{y_0}\otimes \{0\}} && X \\
	{X_{y_0}\otimes \square ^1} & {\square ^1} & Y
	\arrow[hook, from=1-1, to=1-3]
	\arrow[from=1-1, to=2-1]
	\arrow["p", from=1-3, to=2-3]
	\arrow["h"{description}, dotted, from=2-1, to=1-3]
	\arrow[from=2-1, to=2-2]
	\arrow["f"', from=2-2, to=2-3]
\end{tikzcd}\]such that, for each $x\in X_{y_{0}}$, the edge $h\vert\{x\}\otimes\square^{1}$
is $p$-cocartesian. The functor $f_{!}$ is obtained by restricting
the filler to $X_{y_{0}}\otimes\{1\}$. Note that $f_{!}$ is well-defined
up to natural equivalence by \cref{lem:cocart_natural_transformation}.
\end{definition}

Given a cocartesian fibration, every edge in the base induces a functor
between the fibers over its source and the target, via cocartesian
transport. The definition of cocartesian fibration grants us something
more; if $p \from X \to Y$ is a cocartesian fibration of cubical sets, then
every $2$-cube in $Y$ depicted in the left hand cube below gives
rise to a diagram of cubical quasicategories on the right, commuting
up to (canonical) natural equivalence:
\[\begin{tikzcd}
	{y_{00}} & {y_{10}} & {X_{y_{00}}} & {X_{y_{10}}} \\
	{y_{01}} & {y_{11}} & {X_{y_{01}}} & {X_{y_{11}}.}
	\arrow[from=1-1, to=1-2]
	\arrow[from=1-1, to=2-1]
	\arrow[from=1-2, to=2-2]
	\arrow[from=1-3, to=1-4]
	\arrow[from=1-3, to=2-3]
	\arrow[from=1-4, to=2-4]
	\arrow[from=2-1, to=2-2]
	\arrow[from=2-3, to=2-4]
\end{tikzcd}\]

Unraveling the definitions, this fact essentially rests upon the following
(half of) ``three out of four'' property of cocartesian edges. 
\begin{proposition}
\label{prop:four_three}Let $p \from X\to Y$ be a cocartesian fibration
of cubical sets, and consider a $2$-cube $\sigma$ in $X$ depicted
as
\[\begin{tikzcd}
	{x_{00}} & {x_{10}} \\
	{x_{01}} & {x_{11}}
	\arrow["a", from=1-1, to=1-2]
	\arrow["b"', from=1-1, to=2-1]
	\arrow["d", from=1-2, to=2-2]
	\arrow["c"', from=2-1, to=2-2]
\end{tikzcd}\]Suppose that $a$ and $b$ are $p$-cocartesian. Then $c$ is $p$-cocartesian
if and only if $d$ is $p$-cocartesian.
\end{proposition}

The proof of \cref{prop:four_three} is a bit involved 
and takes up the rest of this subsection.
\begin{definition}
Let $p \from X\to Y$ be a cocartesian fibration of cubical sets, and let
$e \from x_{0}\to x_{1}$ be an edge of $X$. The \emph{shadow} of $e$ is
an edge filling the $2$-cube on the left
\[\begin{tikzcd}
	{x_0} & {x_1'} & {p(x_0)} & {p(x_1)} \\
	{x_1} & {x_1} & {p(x_1)} & {p(x_1)}
	\arrow["{e'}", from=1-1, to=1-2]
	\arrow["e"', from=1-1, to=2-1]
	\arrow[dotted, from=1-2, to=2-2]
	\arrow["{p(e)}", from=1-3, to=1-4]
	\arrow["{p(e)}"', from=1-3, to=2-3]
	\arrow[equal, from=1-4, to=2-4]
	\arrow[equal, from=2-1, to=2-2]
	\arrow[equal, from=2-3, to=2-4]
\end{tikzcd}\]Here the left hand cube lies over the $2$-cube in $Y$ on the right,
obtained from $p(e)$ by applying the negative connection $\conn{2}{1,0}$, and where $e'$
is $p$-cocartesian. By \cref{lem:cocart_natural_transformation},
the shadow of $e$ is well-defined up to equivalence in the cubical
quasicategory $\Fun\pr{\square^{1},X_{p\pr{x_{1}}}}$ 
\end{definition}

\begin{proposition}
\label{prop:shadow}Let $p \from X\to Y$ be a cocartesian fibration of
cubical sets. An edge $e \from x_{0}\to x_{1}$ of $X$ is $p$-cocartesian
if and only if its shadow is an equivalence of the cubical quasicategory
$X_{p\pr{x_{1}}}$.
\end{proposition}

\begin{proof}
If $e$ is $p$-cocartesian, then the degenerate edge $x_1\degen{}{1} \in X_1$ is a shadow of $e$, which is an equivalence.

Conversely, suppose the shadow of $f$ is an equivalence. 
For $n\geq2$ and $1\leq i\leq n$, we must show that every lifting problem of the form
\[\begin{tikzcd}
	{\square^1} & {\sqcap^n_{i,1}} & X \\
	& {\square ^n} & Y
	\arrow["{e_{i,1}}"', from=1-1, to=1-2]
	\arrow["e", curve={height=-18pt}, from=1-1, to=1-3]
	\arrow["f"', from=1-2, to=1-3]
	\arrow[from=1-2, to=2-2]
	\arrow["p", from=1-3, to=2-3]
	\arrow[dotted, from=2-2, to=1-3]
	\arrow["g"', from=2-2, to=2-3]
\end{tikzcd}\]
admits a solution. For simplicity, we will assume that $i=0$; we
will indicate where changes are necessary for the general case.

For each integer $k\in\mbb Z$, let $[k,k+1]$ denote a copy of $\square^{1}$,
depicted as $k\to k+1$. We then set $[0,n]=[0,1]\push\limits_{\{1\}}\cdots\push\limits_{\{n-1\}}[n-1,n]$.
We define a map $h \from [0,n]\otimes\square^{n}\to\square^{n}$ as follows:
for each $0\leq i<n$, the restriction $h\vert[k,k+1]\otimes\square^{n}$
is given by
\begin{align*}
[k,k+1]\otimes\square^{n} & \cong[k,k+1]\otimes\square^{k}\otimes\square^{n-k}\\
 & \to[k,k+1]\otimes\{1^{k}\}\otimes\square^{n-k}\\
 & \cong[k,k+1]\otimes\square^{1}\otimes\square^{n-k-1}\\
 & \to\square^{1}\otimes\square^{n-k-1}\\
 & \cong\{1^{k}\}\otimes\square^{n-k}\\
 & \hookrightarrow\square^{k}\otimes\square^{n-k}=\square^{n}.
\end{align*}
where the map $[k,k+1]\otimes\square^{1}\to\square^{n-k-1}$ is given
by the negative connection. We observe that $h_{0}=\id$ and $h_{n}$
is the constant map at $1^{n}\in\square^{n}$, where we wrote $h_{k}=h\vert\{k\}\otimes\square^{n}$.
(If $i>0$, we replace the domain of $[0,n]\otimes\square^{n}$ by
\begin{align*}
\pr{\square^{a}\otimes\pr{[0,1]\otimes\square^{1}}\otimes\square^{b}} & \push\limits_{\square^{a}\otimes\{1\}\otimes\square^{1}\otimes\square^{b}}\pr{\pr{[1,a+1]\otimes\square^{a}}\otimes\square^{1}\otimes\square^{b}}\\
\push\limits_{\{a+1\}\otimes\square^{a}\otimes\square^{1}\otimes\square^{b}} & \pr{\square^{a}\otimes\square^{1}\otimes\pr{[a+1,n]\otimes\square^{b}}},
\end{align*}
where we set $a=i-1$ and $b=n-i$.)

Now let $V$ denote the set of vertices of $\sqcap_{i,1}^{n}$. For
each $v\in V$, we choose a map $\alpha^{\pr v} \from [0,n]\otimes\{v\}\to X$
carrying each edge to a $p$-cocartesian edge, and which lies over
$h\vert[0,n]\otimes\{v\}$. We choose these maps so that $\alpha^{\pr v}\vert[i,i+1]\otimes\{v\}$
is degenerate whenever it lies over a degenerate edge. Using \cref{lem:cocart_natural_transformation} iteratively, we can find
a map $\alpha \from [0,n]\otimes\sqcap_{i,1}^{n}\to X$ rendering the diagram
\[\begin{tikzcd}
	{(\{0\}\otimes \sqcap^n_{i,1})\push\limits_{\{0\}\otimes \sqcap^n_{i,1}}([0,n]\otimes V))} & X \\
	{[0,n]\otimes \sqcap^n_{i,1}} & Y
	\arrow["{(f,\alpha^{(v)})}", from=1-1, to=1-2]
	\arrow[from=1-1, to=2-1]
	\arrow["p", from=1-2, to=2-2]
	\arrow["\alpha"{description}, from=2-1, to=1-2]
	\arrow["h"', from=2-1, to=2-2]
\end{tikzcd}\]commutative. 

By hypothesis, the edge $\alpha_{n}\pr{e_{i,1}}$ is an equivalence
in the cubical quasicategory $X_{g\pr{1^{n}}}$. Thus we can extend
$\alpha_{n}$ to a map $\beta \from \square^{n}\to X$. We then get the
following commutative diagram:
\[\begin{tikzcd}
	{([0,n]\otimes \sqcap^n_{i,1})\push\limits_{\{n\}\otimes \square ^n}(\{n\}\otimes \square ^n)} & X \\
	{[0,n]\otimes \square^n} & Y.
	\arrow["{(\alpha,\beta)}", from=1-1, to=1-2]
	\arrow[from=1-1, to=2-1]
	\arrow["p", from=1-2, to=2-2]
	\arrow[dotted, from=2-1, to=1-2]
	\arrow["h"', from=2-1, to=2-2]
\end{tikzcd}\]The left hand inclusion is a composition of pushouts of inclusions
of the form $\sqcap_{0,0}^{a}\to\square^{n}$, and the images of the
critical edges of $\sqcap_{0,0}^{a}$ are all degenerate in $X$.
(This is because they lie over the degenerate edge $h\pr{1^{n}}\to h\pr{1^{n}}$,
since the final vertex of a nondegenerate cube of $\square^{n}$ that
is missing in $\sqcap_{i,1}^{n}$ is $1^{n}$.). Since $p$ is an
inner fibration, this means that the above diagram has a dotted filler.
Its restriction to $\{0\}\otimes\square^{n}$ gives the solution to
the original lifting problem, and the proof is complete.
\end{proof}
\begin{lemma}
\label{lem:four_three}Let $p \from X\to Y$ be a map of cubical sets, and
suppose we are given a $3$-cube $\sigma$ in $X$ depicted as
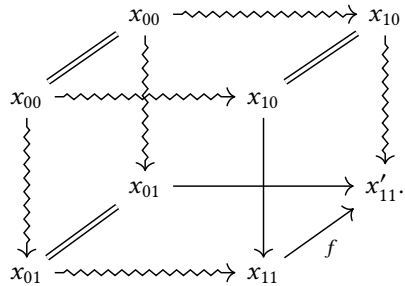
\[\begin{tikzcd}
	& {x_{00}} && {x_{10}} \\
	{x_{00}} && {x_{10}} \\
	& {x_{01}} && {x'_{11}.} \\
	{x_{01}} && {x_{11}}
	\arrow[squiggly, from=1-2, to=1-4]
	\arrow[squiggly, from=1-2, to=3-2]
	\arrow[squiggly, from=1-4, to=3-4]
	\arrow[equal, from=2-1, to=1-2]
	\arrow[squiggly, from=2-1, to=2-3]
	\arrow[squiggly, from=2-1, to=4-1]
	\arrow[equal, from=2-3, to=1-4]
	\arrow[from=2-3, to=4-3]
	\arrow[from=3-2, to=3-4]
	\arrow[equal, from=4-1, to=3-2]
	\arrow[squiggly, from=4-1, to=4-3]
	\arrow["f"', from=4-3, to=3-4]
\end{tikzcd}\]Suppose that all the squiggly arrows $\rightsquigarrow$ are $p$-cocartesian,
and that $p\sigma$ factors as $\square^{3}\cong\square^{1}\otimes\square^{2}\to\square^{0}\otimes\square^{2}\cong\square^{2}$
(the first coordinate corresponds to the slanted arrows). Then $f$
is an equivalence in the cubical quasicategory $X_{p\pr{x_{11}}}$. 
\end{lemma}

\begin{proof}
We will show that $f$ has a left inverse, say $g$. A dual argument
will prove that $g$ also admits a left inverse, and then we may conclude
by the two out of six property of equivalences.

Let $C_{00}\subset\square^{2}$ denote the union of the two edges
containing the vertex $00\in\square^{2}$. Using the fact that the
squiggly arrows are $p$-cocartesian, we can construct a diagram $\tau \from \pr{\partial\square^{2}\otimes\square^{2}}\cup\pr{\square^{2}\otimes C_{00}}\to X$
depicted as
\[\begin{tikzcd}[sep = small]
	\bullet & \bullet & {} & \bullet & \bullet \\
	\bullet & \bullet & {} & \bullet & \bullet \\
	{} & {} && {} & {} \\
	\bullet & \bullet & {} & \bullet & \bullet \\
	\bullet & \bullet & {} & \bullet & \bullet
	\arrow[squiggly, from=1-1, to=1-2]
	\arrow[squiggly, from=1-1, to=2-1]
	\arrow[from=1-2, to=2-2]
	\arrow["\rightarrow"{description}, draw=none, from=1-3, to=2-3]
	\arrow[squiggly, from=1-4, to=1-5]
	\arrow[squiggly, from=1-4, to=2-4]
	\arrow[squiggly, from=1-5, to=2-5]
	\arrow[squiggly, from=2-1, to=2-2]
	\arrow[from=2-4, to=2-5]
	\arrow["\parallel"{description}, draw=none, from=3-1, to=3-2]
	\arrow["\downarrow"{description}, draw=none, from=3-4, to=3-5]
	\arrow[squiggly, from=4-1, to=4-2]
	\arrow[squiggly, from=4-1, to=5-1]
	\arrow[from=4-2, to=5-2]
	\arrow["\parallel"{marking, allow upside down}, draw=none, from=4-3, to=5-3]
	\arrow[squiggly, from=4-4, to=4-5]
	\arrow[squiggly, from=4-4, to=5-4]
	\arrow[from=4-5, to=5-5]
	\arrow[squiggly, from=5-1, to=5-2]
	\arrow[squiggly, from=5-4, to=5-5]
\end{tikzcd}\]which makes the diagram 
\[\begin{tikzcd}
	{(\partial\square^{2}\otimes\square^{2})\cup (\square ^2\otimes C_{00})} && X \\
	{\square ^2\otimes \square ^2} & {\square ^0\otimes \square ^2} & Y
	\arrow["\tau", from=1-1, to=1-3]
	\arrow[from=1-1, to=2-1]
	\arrow["p", from=1-3, to=2-3]
	\arrow["{\widetilde{\tau}}"{description}, dotted, from=2-1, to=1-3]
	\arrow[from=2-1, to=2-2]
	\arrow["{p\sigma}"', from=2-2, to=2-3]
\end{tikzcd}\]commutative, and such that $\tau\vert\square^{2}\otimes C_{00}$ factors
through $C_{00}$. We claim that this diagram admits a filler
$\widetilde{\tau}$. The restriction $\widetilde{\tau}\vert\square^{2}\otimes\{11\}$
witnesses the fact that $f$ has a left inverse.

The left hand inclusion is the composite of a pushout of the inclusion
$i \from \pr{\partial\square^{2}\to\square^{2}}\pp\pr{C_{00}\to\sqcap_{1,1}^{2}}$
and the inclusion $j \from \pr{\partial\square^{2}\to\square^{2}}\pp\pr{\sqcap_{1,1}^{2}\to\square^{2}}$.
The map $i$ is a pushout of the inclusion 
\[
i':\pr{\partial\square^{2}\to\square^{2}}\pp\pr{\{0\}\otimes\{1\}\to\square^{1}\otimes\{1\}},
\]
which we can identify with the inclusion $\sqcap_{3,1}^{3}\to\square^{3}$.
The image of the critical edge of $\sqcap_{3,1}^{3}$ is the $p$-cocartesian
edge $x_{01}\to x_{11}$, so we can fill this $3$-cube. The inclusion
$j$ can be identified with $\sqcap_{3,1}^{4}\to\square^{4}$. The
image of the critical edge is the $p$-cocartesian edge $x_{00}\to x_{10}$,
so we can fill this $4$-cube. 
\end{proof}
We now arrive at the proof of \cref{prop:four_three}.
\begin{proof}
[Proof of \cref{prop:four_three}]We will show that $d$
is $p$-cocartesian, assuming that $c$ is $p$-cocartesian. The other
half of the corollary follows by a dual argument.

Define $h \from \square^{1}\otimes\square^{2}\otimes\square^{1}\to\square^{1}\otimes\square^{1}$
as the composite
\[
\square^{1}\otimes\square^{2}\otimes\square^{1}\cong\pr{\square^{1}\otimes\square^{1}}\otimes\pr{\square^{1}\otimes\square^{1}}\xrightarrow{\gamma\otimes\gamma}\square^{1}\otimes\square^{1},
\]
where $\gamma$ denotes the negative connection. Let $C_{1},C_{2}\subset\square^{2}$
denote the union of two $1$-cubes containing the vertices $10$ and
$01$, respectively. (Thus $C_{1}$ looks like $00\to10\to11$, and
$C_{2}$ looks like $00\to01\to11$.) We can then depict $h\vert\square^{1}\otimes C_{i}\otimes\square^{1}$
as follows: 
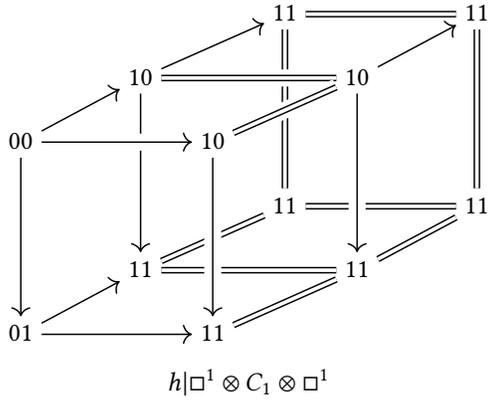
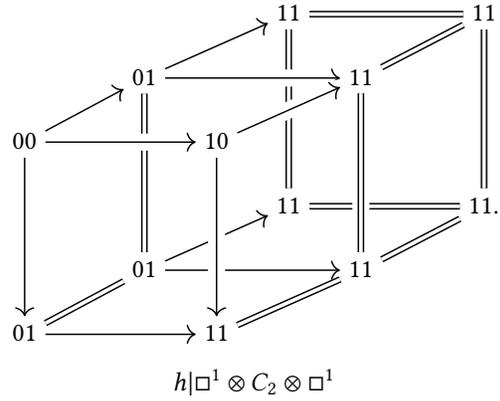
\[ \quad \begin{tikzcd}[sep = 1.2em, cramped]
	&&&& 11 &&& 11 &&&&&& 11 &&& 11 \\
	&& 10 &&& 10 &&&&&& 01 &&& 11 \\
	00 &&& 10 &&&&&& 00 &&& 10 \\
	&&&& 11 &&& 11 &&&&&& 11 &&& {11.} \\
	&& 11 &&& 11 &&&&&& 01 &&& 11 \\
	01 &&& 11 &&&&&& 01 &&& 11 \\
	{} & {} & {} & {} & {} & {} & {} & {} & {} & {} & {} & {} & {} & {} & {} & {} & {} & {}
	\arrow[equal, from=4-5, to=4-8]
	\arrow[equal, from=4-14, to=4-17]
	\arrow[equal, from=5-3, to=4-5]
	\arrow[equal, from=5-3, to=5-6]
	\arrow[equal, from=5-6, to=4-8]
	\arrow[from=5-12, to=4-14]
	\arrow[from=5-12, to=5-15]
	\arrow[equal, from=1-5, to=1-8]
	\arrow[equal, from=1-5, to=4-5]
	\arrow[equal, from=1-8, to=4-8]
	\arrow[equal, from=1-14, to=1-17]
	\arrow[equal, from=1-14, to=4-14]
	\arrow[equal, from=1-17, to=4-17]
	\arrow[from=2-3, to=1-5]
	\arrow[equal, crossing over, from=2-3, to=2-6]
	\arrow[from=2-3, to=5-3]
	\arrow[from=2-6, to=1-8]
	\arrow[crossing over, from=2-6, to=5-6]
	\arrow[from=2-12, to=1-14]
	\arrow[crossing over, from=2-12, to=2-15]
	\arrow[equal, from=2-12, to=5-12]
	\arrow[equal, from=2-15, to=1-17]
	\arrow[equal, crossing over, from=2-15, to=5-15]
	\arrow[from=3-1, to=2-3]
	\arrow[crossing over, from=3-1, to=3-4]
	\arrow[from=3-1, to=6-1]
	\arrow[equal, crossing over, from=3-4, to=2-6]
	\arrow[crossing over, from=3-4, to=6-4]
	\arrow[from=3-10, to=2-12]
	\arrow[crossing over, from=3-10, to=3-13]
	\arrow[from=3-10, to=6-10]
	\arrow[crossing over, from=3-13, to=2-15]
	\arrow[crossing over, from=3-13, to=6-13]
	\arrow[equal, from=5-15, to=4-17]
	\arrow[from=6-1, to=5-3]
	\arrow[from=6-1, to=6-4]
	\arrow[equal, from=6-4, to=5-6]
	\arrow[equal, from=6-10, to=5-12]
	\arrow[from=6-10, to=6-13]
	\arrow[equal, from=6-13, to=5-15]
	\arrow["{h\vert\square ^1\otimes C_1\otimes \square^1}", phantom, from=7-1, to=7-8]
	\arrow["{h\vert\square ^1\otimes C_2\otimes \square^1}", phantom, from=7-10, to=7-17]
\end{tikzcd}\]
We extend $\sigma \from \square^{1}\otimes\square^{1}\cong\square^{1}\otimes\{00\}\otimes\square^{1}\to X$
to a map $\widetilde{\sigma}_{i} \from \square^{1}\otimes C_{i}\otimes\square^{1}\to X$
over $Y$ depicted as
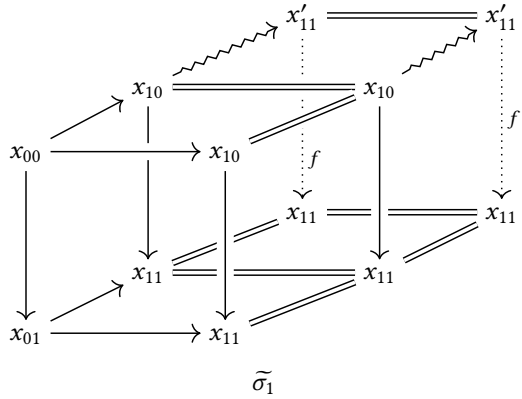
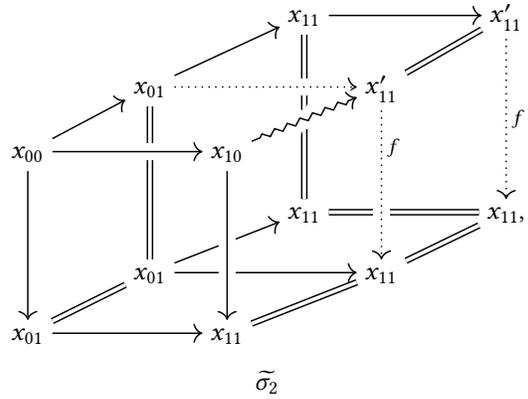
\[ \text{\hspace{1em}} \begin{tikzcd}[sep = 1.1em, cramped]
	&&&& {x'_{11}} &&& {x'_{11}} &&&&&& {x_{11}} &&& {x'_{11}} \\
	&& {x_{10}} &&& {x_{10}} &&&&&& {x_{01}} &&& {x'_{11}} \\
	{x_{00}} &&& {x_{10}} &&&&&& {x_{00}} &&& {x_{10}} \\
	&&&& {x_{11}} &&& {x_{11}} &&&&&& {x_{11}} &&& {x_{11},} \\
	&& {x_{11}} &&& {x_{11}} &&&&&& {x_{01}} &&& {x_{11}} \\
	{x_{01}} &&& {x_{11}} &&&&&& {x_{01}} &&& {x_{11}} \\
	{} & {} & {} & {} & {} & {} & {} & {} & {} & {} & {} & {} & {} & {} & {} & {} & {} & {}
	\arrow[equal, from=1-14, to=4-14]
	\arrow[equal, from=2-12, to=5-12]
	\arrow[equal, from=4-5, to=4-8]
	\arrow[equal, from=4-14, to=4-17]
	\arrow[equal, from=5-3, to=4-5]
	\arrow[equal, from=5-3, to=5-6]
	\arrow[equal, from=5-6, to=4-8]
	\arrow[equal, from=5-15, to=4-17]
	\arrow[from=5-12, to=4-14]
	\arrow[from=5-12, to=5-15]
	\arrow[equal, from=1-5, to=1-8]
	\arrow["f" near end, dotted, from=1-5, to=4-5]
	\arrow["f", dotted, from=1-8, to=4-8]
	\arrow[from=1-14, to=1-17]
	\arrow["f", dotted, from=1-17, to=4-17]
	\arrow[squiggly, from=2-3, to=1-5]
	\arrow[equal, crossing over, from=2-3, to=2-6]
	\arrow[from=2-3, to=5-3]
	\arrow[squiggly, from=2-6, to=1-8]
	\arrow[crossing over, from=2-6, to=5-6]
	\arrow[from=2-12, to=1-14]
	\arrow[crossing over, dotted, from=2-12, to=2-15]
	\arrow[equal, from=2-15, to=1-17]
	\arrow["f" near start, crossing over, dotted, from=2-15, to=5-15]
	\arrow[from=3-1, to=2-3]
	\arrow[crossing over, from=3-1, to=3-4]
	\arrow[from=3-1, to=6-1]
	\arrow[equal, crossing over, from=3-4, to=2-6]
	\arrow[crossing over, from=3-4, to=6-4]
	\arrow[from=3-10, to=2-12]
	\arrow[crossing over, from=3-10, to=3-13]
	\arrow[from=3-10, to=6-10]
	\arrow[squiggly, crossing over, from=3-13, to=2-15]
	\arrow[crossing over, from=3-13, to=6-13]
	\arrow[from=6-1, to=5-3]
	\arrow[from=6-1, to=6-4]
	\arrow[equal, from=6-4, to=5-6]
	\arrow[equal, from=6-10, to=5-12]
	\arrow[from=6-10, to=6-13]
	\arrow[equal, from=6-13, to=5-15]
	\arrow["{\widetilde{\sigma_1}}", phantom, from=7-1, to=7-8]
	\arrow["{\widetilde{\sigma_2}}", phantom, from=7-10, to=7-17]
\end{tikzcd}\]
which we construct using the following rules:
\begin{enumerate}
\item Start with the front face $\sigma$.
\item The 1-cube $x_{10} \to x'_{11}$ (depicted as the squiggly arrow in the right diagram) is a $p$-cocartesian edge lying over
$p\sigma\pr{10\to11}$. All the other slanted arrows are $p$-cocartesian
edges that are already given (i.e., one of $a,b,c$ or the identity
morphisms).
\item We then fill the front cube by using the defining property of cocartesian
edges. (First fill the top, bottom, and sides, and then fill the back.)
We always use degeneracies and connections when these faces can be
filled by them.
\item Use the same procedure for the back cube. 
\end{enumerate}

We then end up with two $2$-cubes 
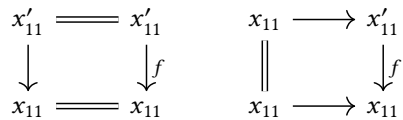
\[\begin{tikzcd}
	{x'_{11}} & {x'_{11}} & {x_{11}} & {x'_{11}} \\
	{x_{11}} & {x_{11}} & {x_{11}} & {x_{11}}
	\arrow[equal, from=1-1, to=1-2]
	\arrow[from=1-1, to=2-1]
	\arrow["f", from=1-2, to=2-2]
	\arrow[from=1-3, to=1-4]
	\arrow[equal, from=1-3, to=2-3]
	\arrow["f", from=1-4, to=2-4]
	\arrow[equal, from=2-1, to=2-2]
	\arrow[from=2-3, to=2-4]
\end{tikzcd}\]in the cubical quasicategory $X_{p\sigma\pr{11}}$. We claim that
these $2$-cubes fit into a larger $3$-cube in $X_{p\sigma\pr{11}}$,
depicted as
\[ \text{\hspace{1em}} \begin{tikzcd}
	& {x'_{11}} && {x'_{11}} \\
	{x_{11}} && {x'_{11}} \\
	& {x_{11}} && {x_{11}} \\
	{x_{11}} && {x_{11}}
	\arrow[equal, from=1-2, to=1-4]
	\arrow[from=1-2, to=3-2]
	\arrow["f", from=1-4, to=3-4]
	\arrow["\simeq", from=2-1, to=1-2]
	\arrow[from=2-1, to=2-3]
	\arrow[equal, from=2-1, to=4-1]
	\arrow["\simeq", from=2-3, to=1-4]
	\arrow[equal, from=3-2, to=3-4]
	\arrow["f"{pos=0.2}, from=2-3, to=4-3]
	\arrow["\simeq", from=4-1, to=3-2]
	\arrow[from=4-1, to=4-3]
	\arrow["\simeq", from=4-3, to=3-4]
\end{tikzcd}\]in which all slanted arrows are equivalences. It will then follow
that the map $f$, which is the shadow of $d$, is an equivalence,
and then \cref{prop:shadow} will show that $d$ is $p$-cocartesian.

For each $i\in\{1,2\}$, let $\tau_{i} \from \square^{1}\otimes\sqcap_{i,1}^{2}\otimes\square^{1}\to X$
denote the amalgamation of the restrictions of $\widetilde{\sigma}_{1}$
and $\widetilde{\sigma}_{2}$. Set $C=\sqcap_{1,1}^{2} \mathbin{\cap} \sqcap_{2,1}^{2}$.
Using the maps $\tau_{i}$ and the projection $\square^{1}\otimes C\to C$,
we obtain a map $\tau$ depicted as the top horizontal arrow of the
diagram
\[\begin{tikzcd}[cramped, column sep = 0em, row sep = 5em]
	{\square^1 \otimes \big( (\{0\}\otimes \sqcap^2_{1,1})\cup (\square ^1\otimes C)\cup (\{1\}\otimes \sqcap^{2}_{2,1}) \big) \otimes \square ^1} & {} &[1.4em] {} &[1em] X \\
	{\square ^1\otimes (\square ^1\otimes \square ^2)\otimes \square ^1} & {\square ^1\otimes (\square ^0\otimes \square ^2)\otimes \square ^1} & \cube{1} \otimes \cube{1} & Y
	\arrow["\tau", from=1-1, to=1-4]
	\arrow[from=1-1, to=2-1]
	\arrow["p", from=1-4, to=2-4]
	\arrow["{\widetilde{\tau}}"{description}, dotted, from=2-1, to=1-4]
	\arrow[from=2-1, to=2-2]
	\arrow["h"', from=2-2, to=2-3]
	\ar[from=2-3, to=2-4]
\end{tikzcd}\]The inclusion $\pr{\{0\}\otimes\sqcap_{1,1}^{2}}\cup\pr{\square^{1}\otimes C}\cup\pr{\{1\}\otimes\sqcap_{2,1}^{2}}\to\square^{1}\otimes\square^{2}$
is a composition of pushouts of inclusions of left open boxes. Moreover,
for each $a,b\in\{0,1\}$, the restriction $\tau\vert\{a\}\otimes\pr{\{0\}\otimes\sqcap_{1,1}^{2}}\cup\pr{\square^{1}\otimes C}\cup\pr{\{1\}\otimes\sqcap_{2,1}^{2}}\otimes\{b\}$
carries \textit{every} edge to a $p$-cocartesian edge (the slanted edges in $\tilde{\sigma}_i$ are $p$-cocartesian). Hence, by \cref{lem:cocart_natural_transformation}, we can find a dotted
filler $\widetilde{\tau}$. According to \cref{lem:four_three},
for each $a,b\in\{0,1\}$, the edge $\widetilde{\tau}\vert\{a\}\otimes\square^{1}\otimes\{11\}\otimes\{b\}$
is an equivalence. So the $3$-cube $\widetilde{\tau}\vert\square^{1}\otimes\square^{1}\otimes\{11\}\otimes\square^{1}$
has the desired properties, and the proof is complete.
\end{proof}

\subsection{Stability properties of cocartesian edges}

In the previous subsection, we introduced the notion of cocartesian
fibrations of cubical sets. In this subsection, we record some basic
stability properties of cocartesian edges.
\begin{proposition}
\label{prop:eq_cc}Let $p \from \mcal C\to\mcal D$ be an inner fibration
of cubical quasicategories, and let $f \from x\to y$ be a morphism of $\mcal C$.
The following conditions are equivalent:
\begin{enumerate}
\item The map $f$ is an equivalence of $\mcal C$.
\item The map $f$ is $p$-cocartesian, and the map $p\pr f$ is an equivalence
of $\mcal D$.
\end{enumerate}
\end{proposition}

\begin{proof}
The implication (1)$\implies$(2) is established in \cite[Lem.~4.14]{doherty-kapulkin-lindsey-sattler}.
For the reverse implication, suppose that condition (2) is satisfied.
We must show that $f$ is an equivalence. To this end, consider the
following union of four $2$-cubes
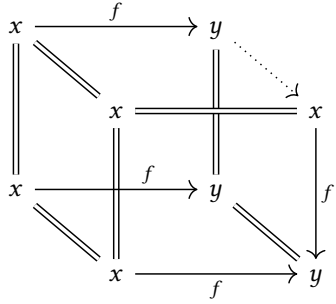
\[\begin{tikzcd}
	x && y \\
	& x && x \\
	x && y \\
	& x && y
	\arrow["f", from=1-1, to=1-3]
	\arrow[equal, from=1-1, to=2-2]
	\arrow[equal, from=1-1, to=3-1]
	\arrow[dotted, from=1-3, to=2-4]
	\arrow[equal, from=1-3, to=3-3]
	\arrow[equal, from=2-2, to=2-4]
	\arrow[equal, from=2-2, to=4-2]
	\arrow["f", from=2-4, to=4-4]
	\arrow["f"{pos=0.7}, from=3-1, to=3-3]
	\arrow[equal, from=3-1, to=4-2]
	\arrow[equal, from=3-3, to=4-4]
	\arrow["f"', from=4-2, to=4-4]
\end{tikzcd}\]corresponding to the bottom, left, front, and back faces, which are
filled by using either degeneracies or connections. Since $\mcal Y$
is a cubical quasicategory and $p\pr f$ is an equivalence, the top
face can be filled in $\mcal Y$. Since $f$ is $p$-cocartesian,
we can lift this $2$-cube to a $2$-cube in $\mcal X$. The same
argument, using \cite[Lem.~4.14]{doherty-kapulkin-lindsey-sattler}, shows that the entire
$3$-cube can be filled in $\mcal Y$, and then it can be lifted to
$\mcal X$. The top and the right faces witnesses the fact that $f$
is an equivalence, and we are done.
\end{proof}
\begin{corollary}
\label{cor:cc_catfib}Every cocartesian fibration of cubical quasicategories
is a categorical fibration.
\end{corollary}

\begin{proof}
This follows from \cref{prop:eq_cc,prop:catfib}.
\end{proof}
\begin{proposition}
\label{prop:cc_equiv}Consider the following commutative diagram of
cubical quasicategories: 
\[\begin{tikzcd}
	{\mathcal{C}} & {\mathcal{C}'} \\
	{\mathcal{D}} & {\mathcal{D}'.}
	\arrow["f", from=1-1, to=1-2]
	\arrow["\simeq"', from=1-1, to=1-2]
	\arrow["p"', from=1-1, to=2-1]
	\arrow["{p'}", from=1-2, to=2-2]
	\arrow["g"', from=2-1, to=2-2]
	\arrow["\simeq", from=2-1, to=2-2]
\end{tikzcd}\]Suppose that $f$ and $g$ are categorical equivalences and that $p$
and $p'$ are categorical fibrations. 
\begin{enumerate}
\item A morphism of $\mcal C$ is $p$-cocartesian if and only if its image
in $\mcal C'$ is $p'$-cocartesian.
\item The functor $p$ is a cocartesian fibration if and only if $p'$ is
a cocartesian fibration.
\end{enumerate}
\end{proposition}

\begin{proof}
We start with (1). Let $\alpha$ be a morphism of $\mcal C$. By definition,
$\alpha$ is $p$-cocartesian if and only if, for any $n\geq2$, $1\leq i\leq n$,
and $\varepsilon\in\{0,1\}$ and any commutative diagram 
\[\begin{tikzcd}
	{\square ^1} & {\sqcap^{n}_{i,1}} & {\mathcal{C}} \\
	& {\square^n} & {\mathcal{D},}
	\arrow["{e_{i,1}}"', from=1-1, to=1-2]
	\arrow["\alpha", curve={height=-12pt}, from=1-1, to=1-3]
	\arrow[from=1-2, to=1-3]
	\arrow["{\iota_{i,1}}"', from=1-2, to=2-2]
	\arrow["p", from=1-3, to=2-3]
	\arrow[dotted, from=2-2, to=1-3]
	\arrow[from=2-2, to=2-3]
\end{tikzcd}\]there is a dotted filler rendering the diagram commutative.
In other words, $\alpha$ is $p$-cocartesian if and only if, for every $n\geq 2$, every extension problem in the category $\Fun([1],\CS_{\square^1/})$ having the form
\[\begin{tikzcd}
	{\iota_{i,1}} & p \\
	{\mathrm{id}_{\square^n}}
	\arrow[from=1-1, to=1-2]
	\arrow[hook, from=1-1, to=2-1]
	\arrow[dotted, from=2-1, to=1-2]
\end{tikzcd}\]
admits a solution.
We now observe that the left vertical arrow is an injective cofibration of injectively cofibrant objects, and that $p$ is injectively fibrant in $\Fun([1],\CS_{\square^1/})$. 
Therefore, by \cite[Prop.~A.2.3.1]{lurie:htt}, the extension problem is equivalent to the one in the homotopy category of $\Fun([1],\CS_{\square^1/})$.
But in the homotopy category, the extension problem is equivalent to the one where $p$ is replaced by $p'$.
Again by \cite[Prop.~A.2.3.1]{lurie:htt}, the latter extension problem is equivalent to asking if the image of $\alpha'$ is $p'$-cocartesian.
This proves (1).

For (2), assume that $p'$ is a cocartesian
fibration. To show that $p$ is a cocartesian fibration, let $c_{0}\in\mcal C$
be an arbitrary object, and let $e \from p\pr{c_{0}}\to d_{1}$ be a morphism
of $\mcal D$. We must find a $p$-cocartesian lift of $e$. For this,
find a $p'$-cocartesian morphism $\widetilde{g\pr e} \from f\pr{c_{0}}\to c'_{1}$
of $\mcal C'$ lying over $g\pr e$. Since the functor $\mcal C_{d_{1}}\to\mcal C'_{g\pr{d_{1}}}$
is an equivalence of cubical quasicategories, we can find an object
$c_{1}\in\mcal C$ and an equivalence $c'_{1}\xrightarrow{\simeq}f\pr{c_{1}}$
in $\mcal C'_{g\pr{d_{1}}}$. Since $p'$ is an inner fibration, we
can find a $2$-cube in $\mcal C'$ depicted on the left, which lies
over the degenerate $2$-cube depicted on the right:
\[\begin{tikzcd}
	{f(c_0)} & {c_1'} & {p'f(c_0)} & {g(d_1)} \\
	{f(c_0)} & {f(c_1)} & {p'f(c_0)} & {g(d_1).}
	\arrow["{\widetilde{g(e)}}", from=1-1, to=1-2]
	\arrow[equal, from=1-1, to=2-1]
	\arrow["\simeq", from=1-2, to=2-2]
	\arrow["{g(e)}", from=1-3, to=1-4]
	\arrow[equal, from=1-3, to=2-3]
	\arrow[equal, from=1-4, to=2-4]
	\arrow["\alpha"', from=2-1, to=2-2]
	\arrow["{g(e)}"', from=2-3, to=2-4]
\end{tikzcd}\]
Since the functor $\mcal C\to\mcal {D}\times_{\mcal{D} '}\mcal{C}'$ is a weak categorical equivalence of categorical fibrations over $\mcal{D}$, it has an inverse equivalence over $\mcal{D}$. This implies that $\alpha$ is homotopic to the image of a morphism
$\widetilde{e} \from c_{0}\to c_{1}$ in $\mcal C$ lying over $e$. By
Proposition \cref{prop:four_three,prop:eq_cc}, the map
$\alpha$ and $f\pr{\widetilde{e}}$ are $p'$-cocartesian. It follows
from part (1) that $\widetilde{e}$ is $p$-cocartesian. We have thus
obtained a $p$-cocartesian edge lying over $e$, hence $p$ is a cocartesian fibration.

This proves the reverse implication. 
For the forward implication, we repeat the same argument using the homotopy inverse of $\pr{f,g} \from p\to p'$,
which exists because it is a weak equivalence of fibrant-cofibrant
objects in $\Fun\pr{[1],\CS}$.
\end{proof}
\begin{proposition}
\label{prop:cc_morph_stable_under_equiv}Let $p \from \mcal C\to\mcal D$
be a categorical fibration. The set of $p$-cocartesian morphisms
is stable under equivalences in $\Fun\pr{\square^{1},\mcal C}$.
\end{proposition}

\begin{proof}
Let $f_{0},f_{1}$ be morphisms of $\mcal C$, and suppose there is
a $2$-cube $\alpha \from \square^{1}\otimes\square^{1}\to\mcal C$ depicted
as
\[\begin{tikzcd}
	\bullet & \bullet \\
	\bullet & {\bullet.}
	\arrow["\simeq", from=1-1, to=1-2]
	\arrow["{f_0}"', from=1-1, to=2-1]
	\arrow["{f_1}", from=1-2, to=2-2]
	\arrow["\simeq"', from=2-1, to=2-2]
\end{tikzcd}\]In other words, $\alpha\vert\{\varepsilon\}\otimes\square^{1}=f_{\varepsilon}$,
and $\alpha\vert\square^{1}\otimes\{\varepsilon\}$ is an equivalence,
for each $\varepsilon\in\{0,1\}$. We will show that if $f_{0}$ is
$p$-cocartesian, so is $f_{1}$. 

Let $n\geq2,1\leq i\leq n$, and consider a solid commutative diagram
\[\begin{tikzcd}
	{\square ^1} & {\sqcap^{n}_{i,1}} & {\mathcal{C}} \\
	& {\square^n} & {\mathcal{D}.}
	\arrow["{e_{i,1}}"', from=1-1, to=1-2]
	\arrow["{f_1}", curve={height=-12pt}, from=1-1, to=1-3]
	\arrow[from=1-2, to=1-3]
	\arrow[from=1-2, to=2-2]
	\arrow["p", from=1-3, to=2-3]
	\arrow[dotted, from=2-2, to=1-3]
	\arrow[from=2-2, to=2-3]
\end{tikzcd}\]We need to show that the diagram admits a filler. The map $p \from \mcal C^{\natural}\to\mcal D^{\natural}$
is a fibration of marked cubical quasicategories (\cref{prop:catfib}),
so we can expand this diagram as
\[\begin{tikzcd}
	{(\square ^1)^\sharp\otimes (\square ^1)^\flat} & {(\square ^1)^\sharp\otimes (\sqcap^{n}_{i,1})^\flat} & {\mathcal{C}^\natural} \\
	& {(\square ^1)^\sharp\otimes (\square^n)^\flat} & {\mathcal{D}^\natural.}
	\arrow["{\id \otimes e_{i,1}}"', from=1-1, to=1-2]
	\arrow["\alpha", curve={height=-12pt}, from=1-1, to=1-3]
	\arrow[from=1-2, to=1-3]
	\arrow["\iota"', from=1-2, to=2-2]
	\arrow["p", from=1-3, to=2-3]
	\arrow[dotted, from=2-2, to=1-3]
	\arrow[from=2-2, to=2-3]
\end{tikzcd}\]
It will suffice to construct a filler of this diagram. 
Using \cite[Prop.~A.2.3.1]{lurie:htt} in the model category $\CS^+_{/\mathcal{D}^\natural}$, we deduce that this is equivalent to finding a filler of the similar diagram in which $\pr{\square^{1}}^{\sharp}$
is replaced by $\{0\}^{\sharp}$ (using the inclusion $\{0\}\hookrightarrow\square^{1}$). 
The existence of such a filler follows from the fact that $f_{0}$ is
$p$-cocartesian.
\end{proof}
\begin{proposition}
\label{prop:cc_cancellation}Let $p \from X\to Y$ and $q \from Y\to Z$ be inner
fibrations of cubical sets, and let $f$ be an edge of $X$. Suppose
that the following conditions are satisfied:
\begin{enumerate}
\item The map $p$ is a categorical fibration.
\item The edge $f$ is $\pr{q\circ p}$-cocartesian.
\item The edge $p\pr f$ is $q$-cocartesian.
\end{enumerate}
Then $f$ is $p$-cocartesian.
\end{proposition}

\begin{proof}
Let $n\geq2$ and $1\leq i\leq n$, and suppose we are given a solid
commutative diagram 
\[\begin{tikzcd}
	{\square ^1} & {\sqcap^{n}_{i,1}} & X \\
	& {\square^n} & {Y.}
	\arrow["{e_{i,1}}"', from=1-1, to=1-2]
	\arrow["f", curve={height=-18pt}, from=1-1, to=1-3]
	\arrow[from=1-2, to=1-3]
	\arrow[from=1-2, to=2-2]
	\arrow["p", from=1-3, to=2-3]
	\arrow[dotted, from=2-2, to=1-3]
	\arrow[from=2-2, to=2-3]
\end{tikzcd}\]We wish to show that there is a dotted filler rendering the diagram
commutative. In other words, we wish to show that the cubical set
$\Fun_{\sqcap_{i,1}^{n}//Y}\pr{\square^{n},X}$ is nonempty. We will
prove more strongly that it is contractible. 

The cubical set in question is a fiber of the map
\[
\pi \from \Fun_{\sqcap_{i,1}^{n}//Z}\pr{\square^{n},X}\to\Fun_{\sqcap_{i,1}^{n}//Z}\pr{\square^{n},Y}.
\]
Using condition (1), we find that $\pi$ is a categorical fibration.
Also, by hypotheses (2) and (3), the domain and the codomain of $\pi$
are contractible cubical Kan complexes (\cref{lem:cocart_natural_transformation}).
Thus $\pi$ is an acyclic fibration of contractible cubical Kan complexes,
and hence its fibers are all contractible. The proof is now complete.
\end{proof}

\section{\label{sec:model}Model structures for left fibrations and cocartesian
fibrations}

In this section, we construct model structures for left fibrations
and cocartesian fibrations over a fixed cubical set $B$, called the
\textit{covariant model structure} and the \textit{cocartesian model
structure}. The fibrant-cofibrant objects of the covariant model structure
are the left fibrations $X\to B$, and a map $X\to Y$ of left fibrations
over $B$ is a covariant equivalence if and only if the induced
map $X_{b}\to Y_{b}$ of fibers is a homotopy equivalence of Kan complexes
for each $b\in B$. The cocartesian model structure admits a similar
description, using cocartesian fibrations in place of left fibrations.
We will mainly be concerned with the construction, and carry out the
analysis of this model structure in later sections.

\subsection{\label{subsec:ngu}Review of Nguyen's machinery}

Our construction of the covariant and cocartesian model structures
relies on machinery by Nguyen, which we now recall.
\begin{definition}
\cite[$\S$2]{Ngu19} Let $\msf C$ be a locally presentable category
equipped with a cofibrantly generated weak factorization system $\pr{\mscr L,\mscr R}$. 
\begin{itemize}
\item A \emph{functorial cylinder object} on $\msf C$ is an endofunctor
$I \from \msf C\to\msf C$ equipped with natural transformations $\id_{\msf C}\sqcup\id_{\msf C}\xrightarrow{\pr{\iota_{0},\iota_{1}}}I\xrightarrow{\sigma}\id_{\msf C}$,
such that $\sigma\partial_{0}=\sigma\partial_{1}=\id_{\id_{\msf C}}$
and the components of $\pr{\iota_{0},\iota_{1}} \from \id_{\msf C}\sqcup\id_{\msf C}\to I$
belong to $\mscr L$. We typically denote the action of $I$ on an
object $X\in\msf C$ by 
\[
I\pr X=I\otimes X.
\]
We also write $\partial I\otimes X=X\sqcup X$, $X=\{0\}\otimes X=\{1\}\otimes X$,
and write $\{\varepsilon\}\otimes X\to I\otimes X$ for the component
of $\iota_{\varepsilon}$.
\item A functorial cylinder object $\pr{I,\iota_{0},\iota_{1},\sigma}$
on $\msf C$ is called \emph{exact} if it satisfies the following
conditions:
\begin{itemize}
\item The functor $I$ commutes with small colimits.
\item For every morphism $\pr{i \from K\to L}\in\mscr L$, the map
\[
\partial I \pp  i \from \pr{\partial I\otimes L}\push\limits_{\partial I\otimes K}\pr{I\otimes K}\to I\otimes L
\]
belongs to $\mscr L$.
\item For every morphism $\pr{i \from K\to L}\in\mscr L$ and every $\varepsilon\in\{0,1\}$,
the morphism
\[
\iota_{\varepsilon} \pp  i \from \pr{\{\varepsilon\}\otimes L}\push\limits_{\{\varepsilon\}\otimes K}\pr{I\otimes K}\to I\otimes L
\]
belongs to $\mscr L$.
\end{itemize}
\end{itemize}
\end{definition}
Geometrically, the morphism $\iota_{\eps}  \pp  i$ is the inclusion $\operatorname{Cyl}(i) \to \operatorname{Cyl}(L)$ of the ``mapping cylinder'' on $i$ into the cylinder on $L$. 
The morphism $\bd I  \pp  i$ is the inclusion of the ``double-mapping cylinder'' $\operatorname{Cyl}(i, i) \to \operatorname{Cyl}(L)$ of $i$ with itself into the cylinder on $L$. 

\begin{definition}
\cite[$\S$ 2]{Ngu19}\label{def:lan_i} Let $\msf C$ be a locally
presentable category equipped with a cofibrantly generated weak factorization
system $\pr{\mscr L,\mscr R}$ and an exact cylinder $\pr{I,\iota_{0},\iota_{1},\sigma}$.
A class of morphisms $\mrm{LAn}\subset\mscr L$ is called a class
of \emph{left $I$-anodyne extensions }if the following conditions
are satisfied:
\begin{itemize}
\item There is a small set of morphisms $\Lambda\subset\mscr L$ generating
$\mrm{LAn}$ as a saturated class of morphisms.
\item For any $\pr{i \from K\to L}\in\mscr L$, the map $\iota_{0} \pp i$
belongs to $\mrm{LAn}$.
\item For any $\pr{i \from K\to L}\in\mrm{LAn}$, the map $\partial I \pp i$
belongs to $\mrm{LAn}$.
\end{itemize}

Maps having the right lifting property for left $I$-anodyne extensions
are called \emph{left $I$-fibrations}.
\end{definition}

\begin{theorem}
\cite[Thm.~2.17]{Ngu19}\label{thm:ngu19} Let $\msf C$ be a locally
presentable category equipped with a cofibrantly generated weak factorization
system $\pr{\mscr L,\mscr R}$, an exact cylinder $\pr{I,\iota_{0},\iota_{1},\sigma}$,
and a class $\mrm{LAn}$ of left $I$-anodyne extensions. 
Suppose that, for all $X \in \msf{C}$, the unique map $\varnothing \to X$ from the initial object is in $\mscr{L}$.
Then, there is a unique model structure on $\msf C$ such that:
\begin{enumerate}
\item cofibrations are the maps in $\mscr L$; and
\item if $X,Y\in\msf C$ are fibrant, then a map $f \from X \to Y$ is a fibration
if and only if it is a left $I$-fibration.
\end{enumerate}
\end{theorem}

We conclude this section with a particularly nice class of weak equivalences
in the model structure of \cref{thm:ngu19}.
\begin{definition}
\cite[Def.~2.25]{Ngu19} Let $\msf C$ be a locally presentable
category equipped with a cofibrantly generated weak factorization
system $\pr{\mscr L,\mscr R}$ and an exact cylinder $\pr{I,\iota_{0},\iota_{1},\sigma}$.
\begin{enumerate}
	\item A map $f \from X\to Y$ in $\msf C$ is called a \emph{left $I$-deformation
	retract} if there are maps $r \from Y\to X$ and $h \from I\otimes Y\to Y$ satisfying equalities
	\[ \begin{array}{r r@{}l}
		\restr{h}{\{0\}\otimes Y} = fr & \restr{h}{\{1\}\otimes Y} = \id_{Y} & {} \\
		rf=\id_{X} & \restr{h}{I\otimes X} = f \sigma_{X}.
	\end{array} \]
\item 
	A map $f \from X\to Y$ in $\msf C$ is called a \emph{right $I$-deformation
	retract} if there are maps $r \from Y\to X$ and $h \from I\otimes Y\to Y$ satisfying equalities
	\[ \begin{array}{r r@{}l}
		\restr{h}{\{ 0 \} \otimes Y} = \id_{Y} & \restr{h}{\{ 1 \} \otimes Y} = fr & {} \\
		rf = \id_{X} & \restr{h}{I\otimes X} = f \sigma_{X} & .
	\end{array} \]
\end{enumerate}
\end{definition}

\begin{proposition}
\cite[Lem.~3.15, Proposition 3.16]{Ngu19}\label{prop:retracts}
In the situation of \cref{thm:ngu19}, left $I$-deformation
retracts are left $I$-anodyne, and right $I$-deformation retracts
are stable under base change along left $I$-fibrations. \qed
\end{proposition}

\subsection{The covariant model structure}

We now apply \cref{thm:ngu19} to construct a model structure
for left fibrations over a fixed base.
\begin{definition}
A map of cubical sets is called a \emph{left anodyne extension}
if it has the left lifting property for left fibrations. Equivalently,
the class of left anodyne extensions is generated as a saturated class
of morphisms by the\emph{ left open boxes} $\{\sqcap_{i,1}^{n}\subset\square^{n}\}_{1\leq i\leq n}$
and inner open boxes.
\end{definition}

\begin{proposition}
\label{prop:lan_pp}Let $i,j$ be monomorphisms of cubical sets. If
$i$ or $j$ is left anodyne, so is their pushout-product $i \pp j$.
\end{proposition}
\begin{proof}
	It suffices to consider the case where one of $i, j$ is a boundary inclusion $\bd \cube{m} \ito \cube{m}$ and the other is a left open box inclusion $\obox{n}{i, 1} \ito \cube{n}$.
	In this case, the pushout product is either $\obox{m+n}{i, 1} \ito \cube{m+n}$ or $\obox{m+n}{m+i, 1} \ito \cube{m+n}$ (\cite[Lem.~1.26]{doherty-kapulkin-lindsey-sattler}), both of which are left anodyne by definition.
\end{proof}
\begin{proposition}
\label{prop:cov}Let $B$ be a cubical set. There is a model structure
on $\CS_{/B}$ with the following properties:
\begin{enumerate}
\item cofibrations are the monomorphisms; and
\item a map of fibrant objects is a fibration if and only if it is a left
fibration.
\end{enumerate}
\end{proposition}

\begin{proof}
We apply \cref{thm:ngu19} to the factorization system whose
left class is the class of monomorphisms, the exact cylinder $\cube{1} \otimes -$,
and the class of left anodyne extensions. The only nontrivial part
is that the class of left anodyne extensions determines a class of
left $\pr{\square^{1}\otimes-}$-anodyne extensions in the sense of
\cref{def:lan_i}, but this follows from \cref{prop:lan_pp}.
\end{proof}
We call the model structure of \cref{prop:cov} the \emph{covariant
model structure} on $\CS_{/B}$, and denote it by $\pr{\CS_{/B}}_{\cov}$.
Its weak equivalences are called \emph{covariant equivalences} over
$B$.

We record some basic facts about the covariant model structure; some of these results will be strengthened later (\cref{thm:cat_inv,thm:cc_Joy}).

\begin{proposition}
\label{rem:JT07_cov}Let $B$ be a cubical set, and let $F \from \pr{\CS_{/B}}_{\cov}\to\mbf M$
be a left adjoint, where $\mbf M$ is a model category and $F$ preserves
cofibrations. Then $F$ is left Quillen if and only if it carries
left anodynes to weak equivalences.
\end{proposition}
\begin{proof}
	This follows from \cite[Prop.~7.15]{JT07}.
\end{proof}

\begin{proposition}
\label{rem:cc_adj-1}
Every map $f \from B\to B'$
of marked cubical sets induces a Quillen adjunction
\[
f_{!} \from \pr{\CS_{/B}}_{\cov}\bigadj\pr{\CS_{/B'}}_{\cov} \ : \! f^{*}.
\]
\end{proposition}
\begin{proof}
	Follows from \cref{rem:JT07_cov}.
\end{proof}

\begin{proposition}
	For a cubical Kan complex $B$, the covariant model structure on $\CS_{/B}$
	coincides with the slice model structure induced by the Kan--Quillen
	model structure on $\CS$.
\end{proposition}
\begin{proof}
	By \cite[Lem.~4.14]{doherty-kapulkin-lindsey-sattler}, every left fibration $E\to B$ is a Kan fibration, hence these model structures share the same classes of cofibrations and fibrant objects.
\end{proof}

\begin{proposition} \label{rem:cat_inv}
	Every categorical equivalence
	$\mcal B\to\mcal B'$ of cubical quasicategories induces a left Quillen
	equivalence $\pr{\CS_{/\mcal B}}_{\cov}\xrightarrow{\simeq}\pr{\CS_{/\mcal B'}}_{\cov}$.
\end{proposition}
\begin{proof}
	By \cref{prop:catfib}, the covariant model structure $\pr{\CS_{/\mcal B}}_{\cov}$
is a left Bousfield localization of the Joyal model structure. Consequently,
the homotopy category of the former embeds fully faithfully into that
of the latter. 
	This then follows from \cref{prop:cc_equiv}.
\end{proof}

\subsection{The cocartesian model structure}

We next construct a model structure for cocartesian fibrations. The
construction is a little more involved than the case of left fibrations.
\begin{definition}
\label{def:mlan}We define the class of \emph{marked left anodyne
extensions} to be the saturation of the following morphisms:

\begin{enumerate}[label=(ML\arabic*)]

\item The inclusion $\pr{\sqcap_{i,1}^{n}}^{!}\subset\pr{\square_{i,1}^{n}}^{!}$
for $n\geq2$ and $1\leq i\leq n$, where $\pr{\sqcap_{i,1}^{n}}^{!}$ and $\pr{\square_{i,1}^{n}}^{!}$
denote the marking at the critical edge for the $\pr{i,1}$-face.

\item The inclusion $\{0\}^{\sharp}\subset\pr{\square^{1}}^{\sharp}$.

\item The inclusion $\pr{\hat{\sqcap}_{i,\varepsilon}^{n}}^{\flat}\subset\pr{\hat{\square}_{i,\varepsilon}^{n}}^{\flat}$
for $n\geq2$, $1\leq i\leq n$, and $\varepsilon\in\{0,1\}$.

\item The inclusion $E[1]^{\flat}\subset E[1]^{\sharp}$, where $E[1]$ denotes
the cubical nerve of the contractible groupoid with two objects $0$ and $1$.

\item The inclusion $\pr{\sqcap_{i,1}^{2}}^{\sharp}\cup\pr{\square^{2}}^{\flat}\subset\pr{\square^{2}}^{\sharp}$
for $i=1,2$.

\end{enumerate}

We call a map of marked cubical sets a \emph{marked left fibration}
if it has the right lifting property for marked left anodyne extensions.
\end{definition}

We have the following recognition result for marked left fibrations.
\begin{proposition}
\label{prop:mlfib}A map $p \from \pr{X,S}\to\pr{Y,T}$ of marked cubical sets
is a marked left fibration if and only if it satisfies the following
conditions:
\begin{enumerate}[label=(\arabic*)]
\item The map $p \from X\to Y$ is an inner fibration.
\item For every vertex $x\in X$ and every edge $e \from p\pr x\to y'$ in $T$,
there is an edge $x\to x'$ in $S$ lying over $e$.
\item An edge of $X$ belongs to $S$ if and only if it is $p$-cocartesian
and lies over an element of $T$.
\end{enumerate}
\end{proposition}

\begin{proof}
Suppose first that $p$ is a marked left fibration. We must show that
$p$ satisfies conditions (1) through (3). Conditions (1) and (2)
follow from the right lifting property for the maps in (ML3) and (ML2),
respectively. The ``only if'' part of condition (3) follows from
the right lifting property for the maps in (ML1). For the ``if''
part of (2), let $f \from x_{0}\to x_{1}$ be a $p$-cocartesian edge of
$X$ such that $pf$ belongs to $T$. We must show that $f$ belongs
to $S$. For this, we use lifting against (ML3) to find a marked edge
$f' \from x_{0}\to x'_{1}$ lying over $pf$. Being a marked edge, the
edge $f'$ is $p$-cocartesian. Therefore, we can find a $2$-cube
depicted on the left, lying over the degenerate $2$-cube depicted
on the right: 
\[\begin{tikzcd}
	{x_0} & {x_1'} & {y_0} & {y_1} \\
	{x_0} & {x_1} & {y_0} & {y_1.}
	\arrow["{f'}", from=1-1, to=1-2]
	\arrow[equal, from=1-1, to=2-1]
	\arrow["g", from=1-2, to=2-2, dotted]
	\arrow["{pf}", from=1-3, to=1-4]
	\arrow[equal, from=1-3, to=2-3]
	\arrow[equal, from=1-4, to=2-4]
	\arrow["f"', from=2-1, to=2-2]
	\arrow["{pf}"', from=2-3, to=2-4]
\end{tikzcd}\]
Let $X_{pf}$ denote the pullback of $p$ along the 1-cube $pf \from \cube{1} \to Y$, which is a cubical quasicategory since $\cube{1}$ is. 
Write $\restr{p}{pf} \from X_{pf} \to \cube{1}$ for the pullback projection.
Since the left 2-cube is contained in $X_{pf}$, applying \cref{prop:four_three} yields that $g$ is cocartesian with respect to $\restr{p}{pf}$,
hence it is an equivalence of $X_{pf}$ by \cref{prop:eq_cc}. 
Thus, the lifting property for (ML4) implies that $g$ is marked.
The lifting property for (ML5) now shows that $f$ is marked, as desired.

For the converse, suppose that $p$ satisfies conditions (1) through
(3). We must show that $p$ is a marked left fibration. The lifting
property for the maps in (ML1), (ML2), and (ML3) is immediate from
conditions (1), (2), and (3). For (ML4), we may assume by base change
that $Y$ is equal to $E[1]$ (so that, in particular, it is a cubical quasicategory).
In this case, the claim follows from \cref{prop:eq_cc}.
Likewise, the right lifting property for (ML5) follows from \cref{prop:four_three} 
\end{proof}
We also have the following stability result for marked left anodynes.
\begin{proposition}
\label{prop:mlan_pp}Let $i$ and $j$ be monomorphisms of marked
cubical sets. If $i$ or $j$ is marked left anodyne, then so is $i \pp j$.
\end{proposition}

\begin{proof}
We will assume that $j$ is marked left anodyne; the proof for the
case where $i$ is marked left anodyne can be treated similarly. Since
the class of marked left anodynes is saturated, we may assume that
$i$ is either the boundary inclusion $i_{n} \from \pr{\partial\square^{n}}^{\flat}\to\pr{\square^{n}}^{\flat}$
for $n\geq1$ with minimal marking (note that $i_{0} \pp -$
is isomorphic to the identity), or the map $u \from \pr{\square^{1}}^{\flat}\to\pr{\square^{1}}^{\sharp}$.
We may also assume that $j$ lies in one of the classes (ML1) through
(ML5) of \cref{def:mlan}. Note that if $j$ induces an
isomorphism between the underlying cubical sets, then $i_{n} \pp j$
is an isomorphism; also, if $j$ is bijective on vertices, then $u \pp j$
is an isomorphism. This leaves us with the following cases:
\begin{itemize}
\item The map $j$ is in (ML1) or (ML2), and $i=i_{n}$ for some $n\geq1$.
In this case, $i_{n} \pp j$ is a pushout of a map in (ML1).
\item The map $j$ is in (ML2), and $i=u$. In this case, the map $u \pp j$
is in (ML5).
\item The map $j$ is in (ML3), and $i=i_{n}$ for some $n\ge1$. In this
case, $i_{n} \pp j$ is a pushout of a map in (ML3).
\end{itemize}
The proof is now complete.
\end{proof}
With these preparations, we can now construct the model structure
for cocartesian fibrations. 
\begin{proposition}
\label{prop:cc}Let $\overline{B}$ be a marked cubical set. There is a model
structure on $\CS_{/\overline{B}}^{+}$ with the following properties:
\begin{enumerate}
\item Cofibrations are the monomorphisms.
\item A map of fibrant objects is a fibration if and only if it is a marked
left fibration.
\end{enumerate}
\end{proposition}

\begin{proof}
We apply \cref{thm:ngu19} to the class of marked left anodyne
extensions and the cylinder $\pr{\square^{1}}^{\sharp}\otimes-$.
The only nontrivial part is that the class of marked left anodyne
extensions is a class of left $(\pr{\square^{1}}^{\sharp}\otimes-)$-extensions,
but this is proved in \cref{prop:mlan_pp}.
\end{proof}
We call the model structure of \cref{prop:cc} the \emph{cocartesian
model structure} and denote it by $\pr{\CS_{/\overline{B}}^{+}}_{\cc}$.
Its weak equivalences are called \emph{cocartesian equivalences}
over $\overline{B}$. We are primarily interested in the model structure
when $\overline{B}=B^{\sharp}$ is maximally marked. In this case,
\cref{prop:mlfib} tells us that the fibrant-cofibrant
objects are of the form $X^{\natural}\to B^{\sharp}$, where $X\to B$
is a cocartesian fibration and $X^{\natural}$ indicates marking the
cocartesian edges. 

We record some elementary facts about the cocartesian model structure below. 
\begin{proposition}
\label{rem:JT07_cc}
Let $\overline{B}$ be a marked cubical set, and
let 
\[ F \from \pr{\CS_{/\overline{B}}^{+}}_{\cc}\to\mbf M \] 
be a left adjoint, where $\mbf M$ is a model category and $F$ preserves cofibrations.
Then $F$ is is left Quillen if and only if it carries marked left
anodynes to weak equivalences. 
\end{proposition}
\begin{proof}
	This follows from \cite[Prop.~7.15]{JT07}.
\end{proof}

\begin{proposition}
\label{rem:cc_adj}Every map $f \from \overline{B}\to\overline{B}'$ of
marked cubical sets induces a Quillen adjunction
\[
	\begin{tikzcd}
		\mcSet_{/ \overline{B}} \ar[r, bend left, "f_!"{name=Upper}] & \mcSet_{/ \overline{B}'} \ar[l, bend left, "f^*"{name=Lower}] \ar[from=Upper, to=Lower, phantom, "\perp"]
	\end{tikzcd}
\]
between the cocartesian model structures.
\end{proposition}
\begin{proof}
	This follows from \cref{rem:JT07_cc}.
\end{proof}

\begin{proposition}
\label{rem:cc_slice}Let $\mcal C$ be a cubical quasicategory, and
let $\mcal C^{\natural}$ denote the marked cubical set obtained from
$\mcal C$ by marking its equivalences. The cocartesian model structure
on $\CS_{/\mcal C^{\natural}}^{+}$ coincides with the slice model
structure for marked cubical quasicategories. 
\end{proposition}
\begin{proof}
	It suffices to show that these
	model structures have the same class of fibrant objects (and cofibrations,
	which is obvious). This follows from \cref{prop:catfib,prop:eq_cc,prop:mlfib}. 
\end{proof}
In particular, taking
$\mcal C=\square^{0}$, we find that the cocartesian model structure
on $\CS_{/\square^{0}}^{+}\cong\CS^{+}$ recovers the Joyal model structure
for marked cubical sets.

We conclude this subsection with an analog of \cref{rem:cat_inv}
for the cocartesian model structure.
\begin{proposition}
\label{prop:cat_inv_cc}Let $\mcal B$ be a cubical quasicategory.
\begin{enumerate}
\item The forgetful functor
\[
\pr{\CS_{/\mcal B^{\sharp}}^{+}}_{\cc}\to\pr{\CS_{/\mcal B}}_{\Joyal}
\]
is right Quillen.
\item Given fibrant objects $\mcal X^{\natural},\mcal Y^{\natural}\in\pr{\CS_{/\mcal B}^{+}}_{\cc}$,
the map
\[
\ho\pr{\pr{\CS_{/\mcal B^{\sharp}}^{+}}_{\cc}}\pr{\mcal X^{\natural},\mcal Y^{\natural}}\to\ho\pr{\pr{\CS_{/\mcal B}}_{\Joyal}}\pr{\mcal X,\mcal Y}
\]
is injective, and its image consists of those elements represented
by maps $\mcal X\to\mcal Y$ preserving cocartesian edges over $\mcal B$.
\item Any categorical equivalence $\mcal B\to\mcal B'$ induces a left Quillen
equivalence $\pr{\CS_{/\mcal B^{\sharp}}^{+}}_{\cc}\xrightarrow{\simeq}\pr{\CS_{/\mcal B^{\p\sharp}}^{+}}_{\cc}$. 
\end{enumerate}
\end{proposition}


\begin{proof}
For part (1), we factor the functor as 
\[
\pr{\CS_{/\mcal B^{\sharp}}^{+}}_{\cc} \xrightarrow{ \natm{\mcal{B}} \times_{\maxm{\mcal{B}}} - } \pr{\CS_{/\mcal B^{\natural}}^{+}}_{\cc}\to\pr{\CS_{/\mcal B}}_{\Joyal}.
\]
The first functor is right Quillen by \cref{rem:cc_adj}, and
the second one is so by \cref{rem:cc_slice}. Thus we have proved
(1).

For part (2), we may replace $\pr{\CS_{/\mcal B}}_{\Joyal}$ by $\pr{\CS_{/\mcal B^{\natural}}^{+}}_{\cc}$.
By \cref{prop:eq_cc}, the pullback functor $\natm{\mathcal{B}} \times_{\maxm{\mathcal{B}}} -$ sends $\natm{\mathcal{X}}$ and $\natm{\mathcal{Y}}$ to their markings at equivalences (\cref{rem:cc_slice}).
Using the explicit description of marked edges in the marked geometric product, a homotopy of the form $\maxm{(\cube{1})} \otimes (\mathcal{X}, \textrm{equiv's}) \to (\mathcal{Y}, \textrm{equiv's})$ between marked maps $\natm{\mathcal{X}} \to \natm{\mathcal{Y}}$ always ascends to a homotopy $\maxm{(\cube{1})} \otimes \natm{\mathcal{X}} \to \natm{\mathcal{Y}}$.

 Part (3) follows from part (2) and \cref{prop:cc_equiv}.
\end{proof}
We will refine item (3) of \cref{prop:cat_inv_cc} later
as \cref{thm:cat_inv_cc}.

\subsection{Equivalences of fibrant objects}

So far, we have constructed model structures whose fibrant-cofibrant
objects are left fibration and cocartesian fibrations over a fixed
base $B$. At this point, however, it is not very clear what homotopy
theory is modeled by this model structure. We will answer this question
in many different ways, and much of this paper is devoted to addressing
this question (see \cref{prop:TUcomparison,thm:str,thm:str_marked,thm:cat_inv,thm:cat_inv_cc}).
As a first step, we give the following characterization
of weak equivalences of the covariant and cocartesian model structures:
\begin{proposition}
\label{prop:fiberwise_eq_cov}Let $B$ be a cubical set, and let $f \from X\to Y$
be a map of left fibrations over $B$. The following conditions are
equivalent:
\begin{enumerate}
\item The map $f$ is a covariant equivalence.
\item For each $b\in B$, the map $f_{b} \from X_{b}\to Y_{b}$ between the fibers
over $b$ is a homotopy equivalence of cubical Kan complexes.
\end{enumerate}
\end{proposition}

\begin{proposition}
\label{prop:fiberwise_eq_cc}Let $B$ be a cubical set, and let $f \from X^{\natural}\to Y^{\natural}$
be a map of fibrant objects of $\CS_{/B^{\sharp}}^{+}$. The following
conditions are equivalent:
\begin{enumerate}
\item The map $f$ is a cocartesian equivalence.
\item For every vertex $b\in B$, the map $f_{b} \from X_{b}\to Y_{b}$
is a categorical equivalence.
\end{enumerate}
\end{proposition}

The proofs of these propositions are very similar, so we will only
prove \cref{prop:fiberwise_eq_cc}.
\begin{proof}
The implication (1)$\implies$(2) follows from \cref{rem:cc_adj,rem:cc_slice}. 
For the converse, suppose that condition (2) is satisfied. 
We will construct a pair of maps $g \from Y^{\natural}\to X^{\natural}$
and $h \from \pr{\square^{1}}^{\sharp}\otimes X^{\natural}\to X^{\natural}$
over $B$ satisfying $h_{0}=\id_{X^{\natural}}$ and $h_{1}=gf$.
This will show that $f$ has a left homotopy inverse; by repeating
the same argument, we find that $g$ also has a left homotopy inverse.
The two out of six properties of cocartesian equivalences then shows
that $f$ is a cocartesian equivalence.

Let $B\pr n\subset B$ denote the $n$-skeleton of $B$. Set $X\pr n^{\natural}=X^{\natural}\times_{B^{\sharp}}B\pr n^{\sharp}$,
and define $Y\pr n^{\natural}$ similarly. We will construct families
of maps $\{g\pr n \from Y\pr n^{\natural}\to X\pr n^{\natural}\}_{n\geq0}$
and $\{h\pr n \from \pr{\square^{1}}^{\sharp}\otimes X\pr n^{\natural}\to X\pr n^{\natural}\}_{n\geq0}$
such that $h\pr n_{0}=\id$ and $h\pr n_{1}=g\pr n\circ f\vert X\pr n$
for every $n\geq0$, and such that $g\pr n$ and $h\pr n$ extends
$g\pr{n-1}$ and $h\pr{n-1}$ for $n\geq1$. The amalgamation of these
maps will then give us the desired $g$ and $h$.

The construction is inductive. For the base case $n=0$, the cubical
set $B\pr 0$ is discrete, so we can simply use condition (2). For
the inductive step, suppose we have constructed $h\pr n$. We must
find a filler in the diagram
\[\begin{tikzcd}
	{X(n+1)^\natural\push\limits_{\{0\}^\sharp \otimes X(n)^\natural}((\square^1)^\sharp \otimes X(n)^\natural)\push\limits_{\{1\}^\sharp \otimes X(n)^\natural}Y(n)^\natural} && {X(n+1)^\natural} \\
	{((\square^1)^\sharp\otimes X(n+1)^\natural)\push\limits_{\{1\}^\sharp \otimes X(n+1)^\natural}Y(n+1)^\natural}
	\arrow["{(\mathrm{id},h(n),g(n))}", from=1-1, to=1-3]
	\arrow["i"', from=1-1, to=2-1]
	\arrow[dotted, from=2-1, to=1-3]
\end{tikzcd}\]in $\CS_{/B\pr{n+1}^{\sharp}}^{+}$. Since $i$ is a monomorphism and $\natm{X(n+1)}$ is fibrant over $\maxm{B(n+1)}$,
it suffices to show that $i$ is a cocartesian equivalence over $B\pr{n+1}^{\sharp}$.
Since $X\pr n^{\natural}\to Y\pr n^{\natural}$ and $X\pr{n+1}^{\natural}\to Y\pr{n+1}^{\natural}$
are cocartesian equivalences over $B\pr{n+1}^{\sharp}$ (by \cref{rem:cc_adj}), we find that the maps $\phi$ and $\psi$ below
are cocartesian equivalences: 
\[\begin{tikzcd}
	{X(n+1)^\natural\push\limits_{\{0\}^\sharp \otimes X(n)^\natural}((\square^1)^\sharp \otimes X(n)^\natural)} & {X(n+1)^\natural\push\limits_{\{0\}^\sharp \otimes X(n)^\natural}((\square^1)^\sharp \otimes X(n)^\natural)\push\limits_{\{1\}^\sharp \otimes X(n)^\natural}Y(n)^\natural} \\
	{(\square^1)^\sharp\otimes X(n+1)^\natural} & {((\square^1)^\sharp\otimes X(n+1)^\natural)\push\limits_{\{1\}^\sharp \otimes X(n+1)^\natural}Y(n+1)^\natural.}
	\arrow["\phi", from=1-1, to=1-2]
	\arrow["j", from=1-1, to=2-1]
	\arrow["i", from=1-2, to=2-2]
	\arrow["\psi"', from=2-1, to=2-2]
\end{tikzcd}\]Consequently, it suffices to show that $j$ is a cocartesian equivalence.
But this is even marked left anodyne by \cref{prop:mlan_pp},
so we are done.
\end{proof}
\begin{remark}
There is an alternative, perhaps more conceptual, proof of \cref{prop:fiberwise_eq_cov}. 
Since we do not make use of this proof, we will
only give a sketch. 

The hard part is (2)$\implies$(1), so we will focus on this direction.
Suppose that condition (2) is satisfied. For each $A\in\CS_{/B}$,
set
\[
\Map_{B}\pr{A,X}=\Hom_{\CS_{/B}}\pr{\square^{\bullet}\otimes A,X}\in\CS.
\]
We can show that $\Map_{B}\pr{-,X} \from \pr{\CS_{/B}}_{\cov}^{\op}\to\CS_{\KQ}$
is right Quillen (in fact, this is a model of the derived mapping
space \cite{arakawa-carranza-kapulkin:derived-mapping-space}). 
By definition, $\pi_{0}\pr{\Map_{B}\pr{A,X}}$
can be identified with homotopy classes of maps $A\to X$ in the
covariant model structure. Thus, it suffices to show that the map
$\Map_{B}\pr{A,f}$ is a homotopy equivalence. By the usual inductive
argument, we may reduce to the case where $A=\square^{n}$. Since
the inclusion $\square^{0}\cong\{0^{n}\}\to\square^{n}$ is left anodyne,
we can further reduce this to the case where $n=0$, in which case
the claim follows from (2). 

In a sense, this proof is a manifestation of the Yoneda lemma, because
under straightening--unstraightening, the objects $\square^{0}\to B$
correspond to representable presheaves. 
In particular, more care is required to adopt this argument to the cocartesian case.
\end{remark}

\section{\label{sec:triangulation}Comparison with quasicategorical counterparts}

In this section, we compare left fibrations and cocartesian fibrations
of cubical quasicategories with their quasicategorical counterparts under triangulation and its right adjoint.
We will see that they are equivalent at a ``point-set'' level (\cref{prop:cc_edges_vs_cc_edges}) and also at a ``global'' level
(\cref{prop:TUcomparison}).

We begin with a point-set level comparison.
\begin{proposition}
\label{prop:cc_edges_vs_cc_edges}Let $p \from X\to Y$ be an inner fibration
of simplicial sets, and let $f$ be an edge of $X$. 
\begin{enumerate}
\item If $f$ is $p$-cocartesian, then it is $U\pr p$-cocartesian.
\item If $p$ is a categorical fibration of quasicategories and $f$ is
$U\pr p$-cocartesian, then $f$ is $p$-cocartesian.
\end{enumerate}
\end{proposition}

\begin{proof}
We first prove (1). Suppose that $f$ is $p$-cocartesian. We must
show that $f$ is $U\pr p$-cocartesian. In other words, we must solve
a lifting problem of the form
\[\begin{tikzcd}
	{\square^1} & {\sqcap^n_{i,1}} & {U(X)} \\
	& {\square^n} & {U(Y),}
	\arrow["u"', from=1-1, to=1-2]
	\arrow["f", curve={height=-12pt}, from=1-1, to=1-3]
	\arrow[from=1-2, to=1-3]
	\arrow[from=1-2, to=2-2]
	\arrow["p", from=1-3, to=2-3]
	\arrow[dotted, from=2-2, to=1-3]
	\arrow[from=2-2, to=2-3]
\end{tikzcd}\]where $n\geq2$, $1\leq i\leq n$, and $u$ is the critical edge of
the face $\partial_{i,1}$. Using the adjunction $T \adj U$ and commutativity of the categorical product, we
are reduced to solving a lifting problem of the form
\[\begin{tikzcd}
	{} & {} & {} \\[-2em]
	{\Delta^1} & { \big( \bd (\simp{1})^{n-1} \times \Delta^1 \big) \cup \big( (\Delta^1)^{n-1}\times\{0\} \big) } \ar[u, draw=none, "f"{name=Top, description, very near end, yshift=1.4em}] & X \\
	& {(\Delta^1)^n} & Y,
	\arrow["u"', from=2-1, to=2-2]
	\arrow[from=2-1, to=2-3, rounded corners, 
			to path={ 
				|- (Top.south)
				-| (\tikztotarget)
			}
	]
	\arrow["g", from=2-2, to=2-3]
	\arrow["\iota"', from=2-2, to=3-2]
	\arrow["p", from=2-3, to=3-3]
	\arrow[dotted, from=3-2, to=2-3]
	\arrow[from=3-2, to=3-3]
\end{tikzcd}\]
where $\bd (\simp{1})^{n-1} \subset \pr{\Delta^{1}}^{n-1}$ is the simplicial subset consisting
of those simplices $f \from \Delta^{k}\to\pr{\Delta^{1}}^{n-1}$ one of
whose coordinate projections is a constant map (that is, the triangulation $T(\bd \cube{n-1})$ of the boundary of the $(n-1)$-cube),
and $u$ classifies the edge $0^{n-1}0\to0^{n-1}1$.

Choose a filtration 
\[
\bd (\simp{1})^{n-1} = S\pr 0\subset S\pr 1\subset\cdots\subset S\pr N=\pr{\Delta^{1}}^{n-1}
\]
such that each $S\pr a$ is obtained from $S\pr{a-1}$ by adjoining
a nondegenerate simplex $\sigma_{a}$, and such that $\dim\sigma_{a}$
is a non-decreasing function of $a$. We will construct a family of
maps $\{g_{a} \from S\pr a\times\Delta^{1}\to X\}_{0\leq a\leq N}$ over
$Y$, such that each $g_{a}$ agrees with $g$ on $\bd(\simp{1})^{n-1} \times \simp{1}$.
The map $g_{N}$ will then give the desired filler.

The construction is inductive. To begin, we define $g_{0}$
to be the restriction of $g$. 
For the inductive step, let $a\geq1$,
and suppose we have defined $g_{a-1}$. We consider the commutative
diagram
\[\begin{tikzcd}
	{(S(0)\times \Delta^1)\cup (S(a-1)\times \{0\})} & {(S(0)\times \Delta^1)\cup S(a)\times \{0\}} & X \\
	{S(a-1)\times \Delta^1} & {(S(a-1)\times \Delta^1)\cup (S(a)\times \{0\})} \\
	& {S(a)\times \Delta^1} & Y.
	\arrow[from=1-1, to=1-2]
	\arrow[from=1-1, to=2-1]
	\arrow["g", from=1-2, to=1-3]
	\arrow[from=1-2, to=2-2]
	\arrow[from=1-3, to=3-3]
	\arrow["{g_a}"{description}, from=2-1, to=1-3]
	\arrow[from=2-1, to=2-2]
	\arrow["{g'_{a}}"{description}, dotted, from=2-2, to=1-3]
	\arrow["j", from=2-2, to=3-2]
	\arrow[from=3-2, to=3-3]
\end{tikzcd}\]Since the left hand square is a pushout, there is a dotted arrow $g'_{a}$
rendering the diagram commutative. Thus, we are reduced to extending
$g'_{a}$ along $j$. For this, consider the commutative diagram
\[\begin{tikzcd}
	{(\partial\Delta^{\dim\sigma_a}\times \Delta^1)\cup (\Delta^{\dim\sigma_a}\times \{0\})} & {(S(a-1)\times \Delta^1)\cup (S(a)\times \{0\})} \\
	{\Delta^{\dim\sigma_a}\times \Delta^1} & {S(a)\times \Delta^1.}
	\arrow[from=1-1, to=1-2]
	\arrow[from=1-1, to=2-1]
	\arrow["j", from=1-2, to=2-2]
	\arrow["{\sigma_a\times \mathrm{id}}"', from=2-1, to=2-2]
\end{tikzcd}\]Note that $\sigma_{a}\pr 0$ is the initial vertex of $\pr{\Delta^{1}}^{n-1}$;
in particular, the edge $\pr{\sigma_{a}\pr 0,0}\to\pr{\sigma_{a}\pr 0,1}$
is exactly the edge $u$. As explained in the second and the third
paragraph of the proof of \cite[Prop.~2.1.2.6]{lurie:htt}, we can
write the left hand inclusion as a finite composition of pushouts
of inner horn inclusions and a single left horn inclusion, which maps
the initial edge to $\pr{0,0}\to\pr{0,1}$. Therefore, our assumption
on $p$ and $f$ guarantees that $g'_{a}$ extends to a map $g_{a} \from S\pr a\times\Delta^{1}\to X$
over $Y$, completing the induction.

Next, we prove (2). Suppose that $p$ is a categorical fibration of
quasicategories and that $f$ is $U\pr p$-cocartesian. We must show
that $f$ is $p$-cocartesian. To this end, let $\Fun'\pr{\Delta^{1},X}\subset\Fun\pr{\Delta^{1},X}$
denote the full simplicial subset spanned by the maps $\Delta^{1}\to X$
corresponding to $U\pr p$-cocartesian edges. The map
\[
\pi \from \Fun'\pr{\Delta^{1},X}\to\Fun\pr{\{0\},X}\times_{\Fun\pr{\{0\},Y}}\Fun\pr{\Delta^{1},Y}
\]
is a categorical fibration of quasicategories (since $p$ is one and
$U\pr p$-cocartesian edges are stable under equivalences by \cref{prop:cc_morph_stable_under_equiv}), and the map $U\pr{\pi}$
is a trivial fibration by \cref{lem:cocart_natural_transformation}.
It follows that $\pi$ is also a trivial fibration. In concrete terms,
this means the following:
\begin{itemize}
\item [($\ast$)]Given a monomorphism $A\to B$ of simplicial sets and a
commutative diagram 
\[\begin{tikzcd}
	{(B\times \Delta ^1)\cup (A\times \{0\})} & X \\
	{B\times \Delta^1} & Y,
	\arrow["{\phi }", from=1-1, to=1-2]
	\arrow[from=1-1, to=2-1]
	\arrow["p", from=1-2, to=2-2]
	\arrow[dotted, from=2-1, to=1-2]
	\arrow[from=2-1, to=2-2]
\end{tikzcd}\]if $\phi\vert\{b\}\times\Delta^{1}$ is a $U\pr p$-cocartesian
edge for every $b\in B$, then there is a dotted arrow rendering
the diagram commutative.
\end{itemize}

Now we show that $f$ is $p$-cocartesian. We must show that for every
$n\ge2$, every lifting problem of the form
\[\begin{tikzcd}
	{\Delta^{\{0,1\}}} & {\Lambda^n_0} & X \\
	& {\Delta^n} & Y
	\arrow[from=1-1, to=1-2]
	\arrow["f", curve={height=-18pt}, from=1-1, to=1-3]
	\arrow[from=1-2, to=1-3]
	\arrow["i"', from=1-2, to=2-2]
	\arrow["p", from=1-3, to=2-3]
	\arrow[dotted, from=2-2, to=1-3]
	\arrow[from=2-2, to=2-3]
\end{tikzcd}\]admits a solution. For this, we observe that the maps
\[
\Delta^{n}\cong\Delta^{n}\times\{1\}\to\Delta^{n}\times\Delta^{1}\xrightarrow{r}\Delta^{n}
\]
exhibits $i$ as a retract of $\pr{\Lambda_{0}^{n}\times\Delta^{1}}\cup\pr{\Delta^{n}\times\{0\}}\to\Delta^{n}\times\Delta^{1}$,
where $r$ is defined by
\[
r\pr{k,i}=\begin{cases}
0 & \text{if }\pr{k,i}=\pr{1,0},\\
k & \text{otherwise}.
\end{cases}
\]
For each vertex $v\in\Lambda_{0}^{n}$, the map $\{v\}\times\Delta^{1}\xrightarrow{r}\Lambda_{0}^{n}$
is either degenerate or equal to the edge $0\to1$. Consequently,
their images in $X$ are $p$-cocartesian.
So a filler exists by ($\ast$), and we are done.
\end{proof}

\begin{remark}
We should remark that in the proof above we used the fact that degenerate edges of $X$ are $p$-cocartesian because $X$ and $Y$ are quasicategories, which follows from Joyal's lifting theorem \cite[Thm.~1.3]{joyal:qcat-kan}.
In general, a degenerate edge of the total simplicial set of an inner fibration of simplicial sets may fail to be cocartesian.
A counterexample is given in \cite[Rem.~2.1]{rezk_cart} and attributed to Campbell \cite{Cam20}.
\end{remark}

\begin{corollary}
\label{cor:ccfib_qcat_cqcat}Let $p \from \mcal C\to\mcal D$ be a categorical
fibration of quasicategories.
\begin{enumerate}
\item The map $p$ is a left fibration if and only if $U\pr p$ is a left
fibration.
\item The map $p$ is a cocartesian fibration if and only if $U\pr p$ is
a cocartesian fibration. \qed
\end{enumerate}
\end{corollary}

We next turn to a global comparison.
\begin{theorem}
\label{prop:TUcomparison}For every (simplicial) quasicategory $\mcal C$, the
functors
\begin{align*}
U & :\pr{\SS_{/\mcal C}}_{\cov}\to\pr{\CS_{/U\pr{\mcal C}}}_{\cov},\\
U^{+} & :\pr{\SS_{/\mcal C^{\sharp}}^{+}}_{\cc}\to\pr{\CS_{/U\pr{\mcal C}^{\sharp}}^{+}}_{\cc},
\end{align*}
are right Quillen equivalences.
\end{theorem}

For the proof of \cref{prop:TUcomparison},
we need the following lemma.
\begin{lemma}
\label{lem:cc_fibrations}Let $\mcal Z$ be a cubical quasicategory,
and let 
\[\begin{tikzcd}
	{\mathcal{X}} && {\mathcal{Y}} \\
	& {\mathcal{Z}}
	\arrow["f", from=1-1, to=1-3]
	\arrow["p"', from=1-1, to=2-2]
	\arrow["q", from=1-3, to=2-2]
\end{tikzcd}\]be a commutative diagram of cubical quasicategories. Suppose that
$p$ and $q$ are cocartesian fibrations, and that $f$ carries $p$-cocartesian
edges to $q$-cocartesian edges. The following conditions are equivalent:
\begin{enumerate}
\item The map $f^{\natural} \from \mcal X^{\natural}\to\mcal Y^{\natural}$ is
a fibration of $\pr{\CS_{/\mcal Z^{\sharp}}^{+}}_{\cc}$.
\item The map $f \from \mcal X\to\mcal Y$ is a categorical fibration.
\end{enumerate}
\end{lemma}

\begin{proof}
The implication (1)$\implies$(2) follows from \cref{prop:cat_inv_cc}.
For the converse, suppose that condition (2) is satisfied. We must
show that $f^{\natural}$ is a fibration in the cocartesian model
structure, i.e., that it is a marked left fibration. According to
\cref{prop:mlfib}, we must prove the following:

\begin{enumerate}[label=(\alph*)]

\item A morphism of $\mcal X$ is $p$-cocartesian if and only if
it is $f$-cocartesian and its image in $\mcal Y$ is $q$-cocartesian.

\item For every object $x_{0}\in\mcal X$ and every $q$-cocartesian
morphism $e \from y_{0}=f\pr{x_{0}}\to y_{1}$ in $\mcal Y$, there is a
$p$-cocartesian morphism $\widetilde{e} \from x_{0}\to x_{1}$ lying over
$e$.

\end{enumerate}

Assertion (a) follows from \cref{prop:cc_cancellation}.
For (b), choose a $p$-cocartesian morphism $\widetilde{e}':x_{0}\to x_{1}'$
lying over $q\pr e$. By part (a), the map $f\pr{\widetilde{e}'}$
is $q$-cocartesian. Thus, by \cref{prop:shadow}, there
is a $2$-cube in $\mcal Y$ depicted on the right 
\[\begin{tikzcd}
	{x_0} & {x'_1} & {y_0} & {y_1} \\
	{x_1} & {x_1} & {y_1} & {y_1,}
	\arrow["{\widetilde{e}'}", from=1-1, to=1-2]
	\arrow["{\widetilde{e}}"', dotted, from=1-1, to=2-1]
	\arrow["{\widetilde{g}}", from=1-2, to=2-2]
	\arrow["\simeq"', from=1-2, to=2-2]
	\arrow["{f(\widetilde{e}')}", from=1-3, to=1-4]
	\arrow["e"', from=1-3, to=2-3]
	\arrow["g", from=1-4, to=2-4]
	\arrow["\simeq"', from=1-4, to=2-4]
	\arrow[equal, from=2-1, to=2-2]
	\arrow[equal, from=2-3, to=2-4]
\end{tikzcd}\]where $g$ is an equivalence. Since $f$ is a categorical fibration,
we can lift $g$ to an equivalence $x_{1}'\to x_{1}$ for some $x_1 \in X_0$. 
We can then
lift the right hand $2$-cube to a $2$-cube depicted on the left,
using \cite[Lem.~4.14]{doherty-kapulkin-lindsey-sattler}. The resulting edge $\widetilde{e} \from x_{0}\to x_{1}$
is $f$-cocartesian by \cref{prop:cc_morph_stable_under_equiv}
and (a), so this proves (b).
\end{proof}

The analogous statement for left fibrations is also true. 
\begin{corollary} \label{lem:left-fibrations}
	Let $\mathcal{Z}$ be a cubical quasicategory, and let 
	\[\begin{tikzcd}
		{\mathcal{X}} && {\mathcal{Y}} \\
		& {\mathcal{Z}}
		\arrow["f", from=1-1, to=1-3]
		\arrow["p"', from=1-1, to=2-2]
		\arrow["q", from=1-3, to=2-2]
	\end{tikzcd}\]
	be a commutative diagram of cubical quasicategories, where $p$ and $q$ are left fibrations.
	Then, $f$ is a left fibration if and only if it is a categorical fibration.
\end{corollary}
\begin{proof}
	By \cite[Prop.~3.3.16]{Hirschhorn}, it suffices to show the covariant model structure is a left Bousfield localization of the cubical Joyal model structure (over $\mathcal{Z}$). 
	This is clear, as left fibrations of cubical quasicategories are categorical fibrations (\cref{cor:cc_catfib}). 
\end{proof}
\begin{remark} \label{rem:fib-between-fibs-simplicial-analog} 
	The proofs given for \cref{lem:cc_fibrations,lem:left-fibrations} directly translate to proofs in the simplicial setting.
\end{remark}
\begin{proof}
[Proof of \cref{prop:TUcomparison}]We will focus on the
case of cocartesian model structures; the covariant case can be proved
similarly.

We first show that $U^{+}$ is right Quillen and that its total right
derived functor is fully faithful. 
Consider the diagram
\[\begin{tikzcd}
	{\mathsf{sSet}^+_{/\mathcal{C}^\sharp}} & {\mathsf{cSet}^+_{/U(\mathcal{C})^\sharp}} \\
	{\mathsf{sSet}_{/\mathcal{C}}} & {\mathsf{cSet}_{/U(\mathcal{C})}}
	\arrow["{U^+}", from=1-1, to=1-2]
	\arrow[from=1-1, to=2-1]
	\arrow[from=1-2, to=2-2]
	\arrow["U"', from=2-1, to=2-2]
\end{tikzcd}\]
where the categories in the top row carry the cocartesian
model structures and those in the bottom row carry the Joyal model
structures.
We claim that every functor in this diagram is right Quillen:
the bottom functor is right Quillen by \cite[Prop.~4.28]{doherty-kapulkin-lindsey-sattler}, and the top functor is right Quillen by \cref{cor:ccfib_qcat_cqcat} and the fact that the left adjoint preserves monomorphisms.
The right functor is right Quillen by \cref{prop:cat_inv_cc}, and the left functor is also right Quillen (this can be seen, for instance, using \cref{rem:fib-between-fibs-simplicial-analog}).  

Thus, given a pair of fibrant objects $\mcal X^{\natural},\mcal Y^{\natural}\in\SS_{/\mcal C^{\sharp}}^{+}$,
we obtain the following commutative diagram of homotopy classes of
maps in the respective model categories:
\[\begin{tikzcd}
	{[\mathcal{X}^\natural,\mathcal{Y}^\natural]} & {[U^+(\mathcal{X}^\natural),U^+(\mathcal{Y}^\natural)]} \\
	{[\mathcal{X},\mathcal{Y}]} & {[U(\mathcal{X}),U(\mathcal{Y})].}
	\arrow["\phi", from=1-1, to=1-2]
	\arrow["\alpha"', from=1-1, to=2-1]
	\arrow["\beta", from=1-2, to=2-2]
	\arrow["\psi"', from=2-1, to=2-2]
\end{tikzcd}\]We wish to show that $\phi$ is a bijection. 
Using \cref{prop:cat_inv_cc} and its simplicial analog, we find that:
\begin{itemize}
\item The set $[\mcal X,\mcal Y]$ is the set of equivalence classes of functors
$\mcal X\to\mcal Y$ over $\mcal C$, with equivalence relation given
by natural transformations over $\mcal C$ whose components are equivalences.
\item The map $\alpha$ is the inclusion of the subset of those elements
represented by functors $\mcal X\to\mcal Y$ preserving cocartesian
edges.
\item We have a similar description for the map $\beta$.
\end{itemize}
Since $\psi$ is bijective, we are therefore reduced to showing that
a map $f \from \mcal X\to\mcal Y$ over $\mcal C$ preserves cocartesian
edges if and only if $U\pr f$ has the same property.
This is a consequence of \cref{prop:cc_edges_vs_cc_edges}.

We complete the proof by showing that the functor $\mbb RU^{+} \from \ho\pr{\SS_{/\mcal C^{\sharp}}^{+}}\to\ho\pr{\CS_{/U\pr{\mcal C}^{\sharp}}^{+}}$
is essentially surjective. Suppose we are given a cocartesian fibration
$\mcal X\to U\pr{\mcal C}$. Using the fact that the adjunction $T\dashv U$
is a Quillen equivalence for the Joyal model structures, we can find
a categorical fibration $p \from \mcal D\to\mcal C$ and a categorical equivalence
$f \from \mcal X\xrightarrow{\simeq}U\pr{\mcal D}$ rendering the diagram
\[\begin{tikzcd}
	{\mathcal{X}} && {U(\mathcal{D})} \\
	& {U(\mathcal{C})}
	\arrow["f", from=1-1, to=1-3]
	\arrow["\simeq"', from=1-1, to=1-3]
	\arrow[from=1-1, to=2-2]
	\arrow[from=1-3, to=2-2]
\end{tikzcd}\]commutative. Using \cref{prop:cc_equiv}, we find that
$U\pr p$ is a cocartesian fibration, and that $f$ preserves cocartesian
morphisms over $U\pr{\mcal C}$. 
Moreover, any homotopy inverse $g$ to $f$ also preserves cocartesian morphisms,
and the homotopies $fg\simeq\id$ and $gf\simeq\id$ lift
to homotopies in the cocartesian model structure over $U\pr{\mcal C}^{\sharp}$
by \cref{prop:cc_morph_stable_under_equiv}. 
Hence $\mcal X^{\natural}$
lies in the essential image of $\mbb RU^{+}$, and we are done.
\end{proof}

The proof that the alternative slice adjunction is a Quillen equivalence will be given in \cref{Q-equiv-cc}, where we are able to prove it without the assumption that the base is a quasicategory.
We also have an analog of \cref{prop:TUcomparison} for the $(Q^+ \adj \int^+)$ adjunction.
\begin{theorem} \label{thm:int-equiv-cc-qcat}
	For a cubical quasicategory $\mathcal{B}$, the functors
	\begin{align*}
		{\textstyle \int} &\from \Big( \cSet_{ / \mathcal{B}} \Big)_{\mathrm{cov}} \to \Big( \sSet_{ / {(\int \! \mathcal{B})}} \Big)_{\mathrm{cov}} \\
		{\textstyle \int^+} &\from \Big( \mcSet_{ / \maxm{\mathcal{B}}} \Big)_{\mathrm{cc}} \to \Big( \msSet_{ / \maxm{(\int \! \mathcal{B} )}} \Big)_{\mathrm{cc}}
	\end{align*}
	are right Quillen equivalences.
\end{theorem}
\begin{proof}
	We prove the statement for cocartesian model structures, the proof for covariant model structures is analogous.
	
	Firstly, we show $\int^+$ is a right Quillen functor.
	Its left adjoint preserves cofibrations (monomorphisms), so it remains to show $\int^+$ preserves fibrations between fibrant objects.
	By \cref{lem:cc_fibrations} and its simplicial analog (cf.\ \cref{rem:fib-between-fibs-simplicial-analog}), we need only show $\int$ preserves categorical fibrations and cocartesian fibrations over $\mathcal{B}$. 
	Preservation of categorical fibrations holds since $\int$ is right Quillen with respect to the cubical and simplicial Joyal model structures (\cite[Prop.~6.11]{doherty-kapulkin-lindsey-sattler}).

	For the claim on cocartesian fibrations, suppose we are given a cocartesian fibration $p\from \mathcal{E}\to \mathcal{B}.$  
	We must show that $\int \! p$ is a cocartesian fibration.
	Since triangulation is a Quillen equivalence (\cref{triangulation-quillen-equiv}), by \cref{cor:ccfib_qcat_cqcat} and the simplicial analog of \cref{prop:cc_equiv} (cf. \cite[Tag 023N]{kerodon}), we may assume that $p=U(p')$ for some cocartesian fibration $p'\from\mathcal{E}'\to \mathcal{B}'$ of simplicial quasicategories.
	Thus, we are reduced to showing that $\int \circ \ U$ preserves cocartesian fibrations between quasicategories.
	According to \cite[Prop.~6.20]{doherty-kapulkin-lindsey-sattler} (and \cite[Cor.~1.4.4.(b)]{hovey}), there is a natural transformation between right adjoints $\Psi \from \id_{\sSet} \Rightarrow \int \circ \ U$ whose components at quasicategories are categorical equivalences. 
	So the claim follows from another application of the simplicial analog of \cref{prop:cc_equiv}.
	To prove this adjunction is a Quillen equivalence, we may assume that $\mathcal{B}=U\mathcal{C}$ for some quasicategory $\mathcal{C}$ (cf. part (3) of \cref{prop:cat_inv_cc} and \cite[Prop. 3.3.1.1]{lurie:htt}). 
	Now $\Psi$ ascends to a natural transformation $\id_{\msSet} \Rightarrow \int^+ \circ \ U^+$ on marked simplicial sets.
	With this, the diagram
	\[ \begin{tikzcd}
		\msSet_{/ \maxm{\mathcal{C}}} \ar[r, equal] \ar[d, "U^+"'] & \msSet_{/ \maxm{\mathcal{C}}} \\
		\mcSet_{ / \maxm{U \mathcal{C}}} \ar[r, "\int^+"] & \msSet_{/ \maxm{(\int \! U \mathcal{C})}} \ar[u, "\Psi^*"']
	\end{tikzcd} \]
	commutes up to a natural transformation, which is defined on an object $\overline{X} \to \maxm{\mathcal{C}}$ as the map into the pullback
	\[ \begin{tikzcd}
		\overline{X} \ar[rr, "\Psi_X", "\sim"'] \ar[dd] \ar[rd, dotted] &[-1.4em] {} & \int^+ \! U^+ \overline{X} \ar[dd, ""] \\[-1.4em]
		{} & \bullet \ar[ur] \ar[dl] \ar[rd, phantom, "\pbtick" very near start] & {} \\
		\maxm{\mathcal{C}} \ar[rr, "\Psi_{\mathcal{C}}", "\sim"'] & {} & \maxm{(\int \! U \mathcal{C})}
	\end{tikzcd} \]
	induced by the naturality square for $\Psi$.
	The functors $U$ and $\int$ preserve categorical fibrations, hence if the underlying map $X \to B$ is a categorical fibration (which implies both $X$ and $\mathcal{B}$ are quasicategories) then the map into the pullback is a categorical equivalence by 2-out-of-3.
	It follows that if $\overline{X} \to \maxm{B}$ is a cocartesian fibration then the map into the pullback is a cocartesian equivalence over $\maxm{\mathcal{C}}$ \cite[Cor.~3.3.1.5]{lurie:htt}.
	
	As the above diagram of functors commutes up to a cocartesian equivalence and both $U^+, \Psi^*$ are right Quillen equivalences (by \cref{prop:TUcomparison} and \cite[Prop.~3.3.1.1]{lurie:htt}, respectively), it follows that $\int^+$ is a right Quillen equivalence by 2-out-of-3.
\end{proof}
\begin{remark} \label{rem:int-is-not-quillen-equiv-always}
	\Cref{thm:int-equiv-cc-qcat} generally fails if the base $\mathcal{B}$ is not a cubical quasicategory.
	For an example, consider the case $\mathcal{B} = \cube{2}$.
	By \cref{prop:TUcomparison}, the covariant (or cocartesian) model structure on $\cSet_{/ \maxm{(\cube{2})}}$ models the $\infty$-category of functors from $T \cube{2} = [1]^2$ to the $\infty$-category of spaces (respectively, quasicategories), i.e.\ the category of commutative squares in $\catname{Gpd}_{\infty}$ (respectively, $\Cat_{\infty}$).
	On the other hand, $\cube{2}$ admits no non-degenerate maps from $Q[n]$ for $n \geq 2$, from which we deduce that $\int \! \cube{2}$ is the boundary of a square.
	Thus, the cocartesian model structure on $\sSet_{/ \maxm{(\int \! \cube{2})}}$ models \emph{non-commutative} squares in $\catname{Gpd}_{\infty}$ (respectively, $\Cat_{\infty}$).
	
	From this example, we see the problem stems from the fact that $\int \from \cSet \to \sSet$ fails to preserve the Joyal weak equivalence type of non-fibrant objects (in fact, it even fails to preserve the homotopy type).
\end{remark}

\section{\label{sec:straightening}Straightening--unstraightening over ordinary
categories}

In \cite[Thm.~2.2.1.2]{lurie:htt}, Lurie established the \emph{straightening--unstraightening
equivalence}, which asserts the following: Given a quasicategory $\mcal Q$,
the underlying $\infty$-category of $\pr{\SS_{/\mcal Q}}_{\cov}$
is equivalent to the $\infty$-category $\Fun\pr{\mcal Q,\mcal S}$,
where $\mcal S\subset\curlyCat_{\infty}$ denotes the full subcategory
of $\infty$-groupoids. Since every cubical quasicategory $\mcal D$
is equivalent to one of the form $U\pr{\mcal D_{\Delta}}$ for some
quasicategory $\mcal D_{\Delta}$, together with \cref{prop:TUcomparison}
(and \cref{rem:cat_inv}), we obtain a description of the underlying
$\infty$-category of $\pr{\CS_{/\mcal D}}_{\cov}$ in terms of pre(co)sheaves
over $\mcal D_{\Delta}$.

While this description is useful for abstract reasoning, it is of
limited use if we want to do explicit calculations. Indeed, there
is no canonical choice of the quasicategory $\mcal D_{\Delta}$ (although
it exists as a fibrant replacement of $T\pr{\mcal D}$). Moreover,
the straightening--unstraightening equivalence over quasicategories
is also a complicated black box. (See \cite[$\S$ 2.2.1]{lurie:htt}.) In
combination, these gives us an uncontrollable description of straightening
of left fibrations over $\mcal D$. 

The goal of this section is to offer a partial solution to this situation.
More specifically, we will construct an explicit Quillen equivalence
between the covariant model structure over $\mcal D$ and a model
category of cubical presheaves, assuming that $\mcal D$ is the nerve
of an ordinary category (\cref{thm:str}). We will also cover
the case of cocartesian fibrations (\cref{thm:str_marked}). 

In the setting of simplicial sets, the contents of this section are
covered in \cite[$\S$3.2.5]{lurie:htt}.
Readers familiar with straightening--unstraightening in this context might expect a cubical version of this over an arbitrary base, as given in \cite[$\S$2.2.1]{lurie:htt}.
At this level of generality, one encounters rather unruly combinatorics in trying to define the straightening and unstraightening constructions.
This is to be expected, as straightening is inherently a cubical--simplicial construction \cite{kapulkin-voevodsky}.
While it is possible to give a purely simplicial construction by triangulating all the occurrences of cubes, the same cannot be said for a purely cubical construction.
It is for this reason that we only treat the case where the base cubical set is the nerve of a 1-category.
On a more practical level, the version of straightening for $\cat{B} = N \cat{C}$ is perfectly sufficient for our intended application to the cubical Bousfield--Kan formula.

\subsection{\label{subsec:str}Unmarked case}

In this subsection, we consider straightening of left fibrations over
ordinary categories.
\begin{definition}
For each $n\geq0$, we write $\Side\pr{\square^{n}}$ for the poset
of faces of $\square^{n}$ (i.e., maps $\square^{k}\to\square^{n}$
that can be written as an iterated composition of face maps), ordered
by inclusion. Equivalently, it is the full subcategory of $\square_{/\square^{n}}$
spanned by the nondegenerate cubes of $\square^{n}$. 

We will identify each element of $\Side\pr{\square^{n}}$ with its set $S$ of vertices, writing $\square_{S}$ for the corresponding face of $S$.
We additionally write $\cube{S}$ for the image of $S$ under the functor $\Side (\cube{n}) \to \cSet$ which projects a non-degenerate cube $\cube{k} \to \cube{n}$ to its domain $\cube{k} \in \cSet$ (note the difference between writing $S$ in the subscript vs.\ in the superscript).

Every map $u \from \square^{n}\to\square^{m}$
in $\square$ determines a map $\Side\pr{\square^{n}}\to\Side\pr{\square^{m}}$
given by $S\mapsto u\pr S$, where we identify $u$ with a set map
$[1]^{n}\to[1]^{m}$. These maps make $\Side\pr{\square^{\bullet}}$
into a cocubical set.
\end{definition}

\begin{definition}
We define the \emph{last vertex} functor $\LV \from \square_{/\square^{n}}\to[1]^{n}$
by assigning, to each object $\square^{k}\to\square^{n}$, the image
of the vertex $1^{k}\in\square^{k}$.
\end{definition}

\begin{definition}
\label{def:int}Let $\mcal C$ be a small category. We define a functor
\[
\int = \int_{\mcal C} \from \Fun\pr{\mcal C,\CS}\to\CS_{/N\pr{\mcal C}}
\] 
as follows: An $n$-cube of $\int F$ is a pair of an $n$-cube $\sigma \from [1]^{n}\to\mcal C$
of $N\pr{\mcal C}$ and a natural transformation
\[
\{\square^{S}\to F\circ\sigma\pr{\LV\pr{\square_{S}}}\}_{S\in\Side\pr{\square^{n}}}.
\]
If $u \from \square^{n}\to\square^{m}$ is a morphism in $\square$, then
the cubical structure map $u^{*} \from \pr{\int F}_{n}\to\pr{\int F}_{m}$
is induced by the restriction $\square^{S}\to\square^{u\pr S}$ of
$u$.

We also define the \emph{rectification functor }
\[
\Rect \from \CS_{/N\pr{\mcal C}}\to\Fun\pr{\mcal C,\CS}
\]
by the formula $\Rect\pr X=X\times_{N\pr{\mcal C}}N\pr{\mcal C_{/\bullet}}$.
\end{definition}

\begin{remark}
In the situation of \cref{def:int}, we can describe the
cubes of $\int F$ of lower dimensions as follows: 
\begin{enumerate}
\item A vertex of $\int F$ consists of the pairs $\pr{c,x}$, where $c \in\mcal C$
and $x\in F\pr C$.
\item An edge $\pr{c,x}\to\pr{d,y}$ is a pair of a morphism $f \from c\to d$
in $\mcal C$ and an edge $Ff (x) \to y$ in $F\pr d$.
\item A $2$-cube of $\int F$ whose boundary is depicted on the left below
is equivalent to a $2$-cube in $F\pr{C_{11}}$ depicted on the right:
\[\begin{tikzcd}[column sep = 3.7em, row sep = 3.5em]
	{(c_{00},x_{00})} & {(c_{10},x_{10})} & {F(kf)(x_{00})} & {F(k)(x_{10})} \\
	{(c_{01},x_{01})} & {(c_{11},x_{11})} & {F(h)(x_{01})} & {x_{11}.}
	\arrow["{(f, \alpha)}", from=1-1, to=1-2]
	\arrow["{(g, \beta)}"', from=1-1, to=2-1]
	\arrow["{(h, \gamma)}"', from=2-1, to=2-2]
	\arrow["{(k, \delta)}", from=1-2, to=2-2]
	\arrow["{F(k)(\alpha)}", from=1-3, to=1-4]
	\arrow["{F(h)(\beta)}"', from=1-3, to=2-3]
	\arrow["\delta", from=1-4, to=2-4]
	\arrow["\gamma"', from=2-3, to=2-4]
\end{tikzcd}\]
\end{enumerate}
This should convince the reader that $\int$ is a generalization of
the ordinary Grothendieck construction. In fact, given a small category
$\mcal C$ and a functor $G \from \mcal C\to\msf{Cat}$, there is a natural
isomorphism of cubical sets
\[
N\pr{\int G}\cong\int\pr{N\circ G},
\]
where $\int G\to\mcal C$ denotes the Grothendieck construction of
$G$.
\end{remark}

\begin{proposition}
\label{prop:radj_L}Let $\mcal C$ be a small category. The functor
$\Rect \from \CS_{/N\pr{\mcal C}}\to\Fun\pr{\mcal C,\CS}$ is left adjoint
to the functor $\int \from \Fun\pr{\mcal C,\CS}\to\CS_{/N\pr{\mcal C}}$.
\end{proposition}

For the proof of \cref{prop:radj_L}, we need the following
elementary observation.
Recall from \cite[Def.~2.1]{bousfield_factorizationsystem} that an \emph{orthogonal factorization system} $(\mathcal{X}_-, \mathcal{X}_+)$ on a category $\mathcal{X}$ consists of two subcategories, a \emph{left class} $\mathcal{X}_-$ and a \emph{right class} $\mathcal{X}_+$, such that:
\begin{enumerate}
	\item Every morphism $f \in \mathcal{X}$ factors as $f = f_+ f_-$, where $f_- \in \mathcal{X}_-$ and $f_+ \in \mathcal{X}_+$, and
	\item A morphism $f \from A\to B$ of $\mathcal{X}$ belongs to $\mathcal{X}_-$ if and only if it has the unique left lifting property for the maps in $\mathcal{X}_+$; that is, for every morphism $g \from X\to Y$ in $\mathcal{X}_+$ and every commutative diagram 
		\[\begin{tikzcd}
	A & X \\
	B & Y,
	\arrow[from=1-1, to=1-2]
	\arrow["f"', from=1-1, to=2-1]
	\arrow["g", from=1-2, to=2-2]
	\arrow[dotted, from=2-1, to=1-2]
	\arrow[from=2-1, to=2-2]
		\end{tikzcd}\]
	there is a unique dotted arrow rendering the diagram commutative. 
\item	A morphism of $\mathcal{X}$ belongs to $\mathcal{X}_+$ if and only if it has the unique right lifting property for the maps in $\mathcal{X}_-$. 
\end{enumerate}
\begin{lemma}
\label{lem:OFS}Let $\mcal X$ be a category equipped with an orthogonal
factorization system $\pr{\mcal X_{-},\mcal X_{+}}$. For every object
$x \in\mcal X$, the inclusion
\[
\pr{\mcal X_{+}}_{/x }\to\mcal X_{/x}
\]
is a right adjoint. 
\end{lemma}

\begin{proof}
The left adjoint is given by factoring each map $f$ as $f=f_{+}f_{-}$ and then mapping it to $f_{+}$.
\end{proof}

\begin{example}
\label{exa:sides}
Fix a $n$-cube $[1]^n \to \mathcal{C}$ in a category $\mathcal{C}$, and an object $c \in \mathcal{C}$. 
The subcategory $\Side\pr{\square^{n}} \times_{\mathcal{C}} \pr{\mcal C_{/c}}\subset\square_{/\square^{n}}\times_{\mcal C}\pr{\mcal C_{/c}}$
determines the right class of an orthogonal factorization system. 
The left class consists of those maps whose images in $\square$ are a
composite of degeneracies and connections, and whose images in $\mcal C$
are isomorphisms.
Consequently, the inclusion $\Side\pr{\square^{n}}\times_{\mcal C}\mcal C_{/c}\subset\square_{/\square^{n}}\times_{\mcal C}\pr{\mcal C_{/c}}$
is a right adjoint, and hence is homotopy final.
\end{example}

\begin{proof}
[Proof of \cref{prop:radj_L}]Since $\CS_{/N\pr{\mcal C}}\simeq\Set^{\int_{\square}N\pr{\mcal C}}$
is a presheaf category, it suffices to show that there is a bijection
\[
\Hom_{\Fun\pr{\mcal C,\CS}}\pr{\Rect\pr{\square^{n}},F}\cong\Hom_{\CS_{/N\pr{\mcal C}}}\pr{\square^{n},\int F}
\]
which is natural in $\pr{\sigma \from \square^{n}\to N\pr{\mcal C}}\in\CS_{/N\pr{\mcal C}}$
and $F\in\Fun\pr{\mcal C,\CS}$. 

By definition, the right hand side can be identified with the set
of natural transformations depicted as
\[\begin{tikzcd}[row sep = 1.7em, column sep = 4.2em]
	{\operatorname{Side}(\square^n)} && {\mathsf{cSet}} \\
	\boxcat_{/ \cube{n}} & {} & {} \\
	{[1]^n} \\
	{\mathcal{C}}
	\arrow[from=1-1, to=2-1, hook, "i"']
	\arrow[from=2-1, to=3-1, "\LV"']
	\arrow[""{name=0, anchor=center, inner sep=0}, "{\square^\bullet}", from=1-1, to=1-3]
	\arrow["\sigma"', from=3-1, to=4-1]
	\arrow[""{name=1, anchor=center, inner sep=0}, "F"', from=4-1, to=1-3]
	\arrow[Rightarrow, from=0, to=1, shorten=1.5em]
\end{tikzcd}\]
Therefore, it suffices
to show that the left Kan extension $\Lan_{\sigma\circ\LV' \circ i}\square^{\bullet}$
is (naturally) isomorphic to $\Rect\pr{\square^{n}}$.

The relevant Kan extension is given by the formula
\[
\Lan_{\sigma\circ\LV \circ i}\square^{\bullet}(c) =\colim\limits_{\pr{\square_{S}\to\square^{n},\,\sigma\pr{\LV\pr S}\to c}\in\Side\pr{\square^{n}}\times_{\mcal C}\pr{\mcal C_{/c}}}\square_{S}.
\]
By \cref{exa:sides}, the colimit can be taken over the larger
category $\square_{/\square^{n}}\times_{\mcal C}\mcal C_{/c}\cong\square_{/ \left( \square^{n}\times_{N\pr{\mcal C}}N\pr{\mcal C_{/c}} \right)}$.
It follows from the Yoneda lemma that
\[
	\colim\limits_{\cube{k} \in \square_{/ \left( \square^{n}\times_{N\pr{\mcal C}}N\pr{\mcal C_{/c}} \right)}} \cube{k}
	\cong \square^{n}\times_{N\pr{\mcal C}}N\pr{\mcal C_{/c}}
	= \Rect\pr{\square^{n}}\pr c,
\]
as required.
\end{proof}
We give an explicit description the unit and counit of the $(\Rect \adj \int)$ adjunction for later use.
Starting with $\eta$, fix $X\in\CS_{/N\pr{\mcal C}}$, and let $x\in X_{n}$
be an $n$-cube lying over $\sigma \in N \mathcal{C}_n$.
For each $S\in\Side\pr{\square^{n}}$, 
the restriction
\[ \restr{x}{S} \from \cube{S} \to X \to N \mathcal{C}  \]
sends the last vertex of $\cube{S}$ to $\sigma(\LV(S))$ by assumption, hence it induces a map into the pullback
\[ x_S \from \cube{S} \to X \times_{N \mathcal{C}} N (\mathcal{C}_{/ \sigma(\LV (S))}). \]
These maps assemble into a natural transformation $\{x_{S}\}_{S\in\Side\pr{\square^{n}}}$, which determines a cube $\eta\pr x\in\pr{\int\Rect\pr F}_{n}$ lying over
$\sigma$. 

To define $\varepsilon$, let $F\in\Fun\pr{\mcal C,\CS}$ and $c \in\mcal C$.
We must define a map
\[
\varepsilon_{F,C} \from \pr{\int F}\times_{N\pr{\mcal C}}N\pr{\mcal C_{/c}}\to F\pr c.
\]
A typical $n$-cube on the left hand side consists of the following
data:
\begin{itemize}
\item A functor $\sigma \from [1]^{n}\to\mcal C$.
\item A morphism $\sigma\pr{1,\dots,1}\to c$ in $\mathcal{C}$.
\item A natural transformation $\{\square_{S}\to F\circ\sigma\pr{\LV\pr S}\}_{S\in\Side\pr{\square^{n}}.}$
\end{itemize}
We declare that the image of such an $n$-cube is given by the composite
\[
\square^{n}\to F \left( \sigma\pr{1,\dots,1} \right) \to F\pr c.
\]

We now come to the main result of this subsection.
\begin{theorem} \label{thm:str}
Let $\mcal C$ be a small category. The adjunction
\[ \begin{tikzcd}
	\big( \cSet_{/ N(\cat{C})} \big)_{\mathrm{cov}} \ar[r, bend left, "\Rect"{name=Upper}] & \Fun(\cat{C}, \cSet)_{\mathrm{proj}} \ar[l, bend left, "\int"{name=Lower}] \ar[from=Upper, to=Lower, phantom, "\perp"]
\end{tikzcd} \]
is a Quillen equivalence.
\end{theorem}

\begin{proof}
We first check that the adjunction is a Quillen adjunction. To show
that $\int$ preserves trivial fibrations, suppose we are given a
projective trivial fibration $p \from F\to G$ and a solid commutative diagram
\[\begin{tikzcd}
	{\partial\square ^n} & {\int F} \\
	{\square ^n} & {\int G}
	\arrow[from=1-1, to=1-2]
	\arrow[from=1-1, to=2-1]
	\arrow["{\int p}", from=1-2, to=2-2]
	\arrow[dotted, from=2-1, to=1-2]
	\arrow[from=2-1, to=2-2]
\end{tikzcd}\]in $\CS_{/N\pr{\mcal C}}$. We wish to show that there is a dotted
filler rendering the diagram commutative. Unwinding the definitions,
this lifting problem is equivalent to 
\[\begin{tikzcd}
	{\partial\square ^n} & {F(c)} \\
	{\square ^n} & {G(c),}
	\arrow[from=1-1, to=1-2]
	\arrow[from=1-1, to=2-1]
	\arrow["{p_c}", from=1-2, to=2-2]
	\arrow[dotted, from=2-1, to=1-2]
	\arrow[from=2-1, to=2-2]
\end{tikzcd}\]where $c\in\mcal C$ denotes the image of the terminal vertex of $\square^{n}$.
Since $p_{c}$ is a trivial fibration, there is a dotted filler rendering
diagram commutative, as desired. A similar argument shows that $\int$
preserves fibrations of fibrant objects (since Kan fibrations are
left fibrations), so by \cite[Prop.~7.15]{JT07}, the functor
$\int$ is right Quillen.

Next, we show that the adjunction $\Rect\dashv\int$ is a Quillen
equivalence. Let $N_{\Delta}\pr{\mcal C}$ be the nerve of $\mcal C$,
which is a simplicial set. We will exploit Lurie's \emph{relative
nerve} functor
\[
\int:\Fun\pr{\mcal C,\SS}_{\proj}\to\pr{\SS_{/N_{\Delta}\pr{\mcal C}}}_{\cov}
\]
whose left adjoint, denoted by $\sRect$, is given by $\sRect\pr K=K\times_{N_{\Delta}\pr{\mcal C}}N_{\Delta}\pr{\mcal C_{/\bullet}}$.
Recall that the relative nerve functor is a right Quillen equivalence
\cite[Prop.~3.2.5.18]{lurie:htt}. 

Now consider the functors
\[\begin{tikzcd}
	{\operatorname{Fun}(\mathcal{C},\mathsf{sSet})} & {\mathsf{sSet}_{/N_{\Delta}(\mathcal{C})}} \\
	{\operatorname{Fun}(\mathcal{C},\mathsf{cSet})} & {\mathsf{cSet}_{/N(\mathcal{C})}.}
	\arrow["\int", from=1-1, to=1-2]
	\arrow["{\operatorname{Fun}(\mathcal{C},U)}"', from=1-1, to=2-1]
	\arrow["U"', from=1-2, to=2-2]
	\arrow["\int"', from=2-1, to=2-2]
\end{tikzcd}\]
This square does \emph{not} commute on the nose, but there
is a preferred natural transformation $U\circ\int\to\int\circ\Fun\pr{\mcal C,U}$,
induced by the natural transformation 
\[
\theta \from T\pr{-\times_{N\pr{\mcal C}}N ({\mcal C_{/\bullet}}) }\to T\pr -\times_{N_{\Delta}\pr{\mcal C}}N_{\Delta} ({\mcal C_{/\bullet}})
\]
of left adjoints. Since all but the top arrow are known to be Quillen
equivalences (by \cref{prop:TUcomparison}), it suffices
to show that $\theta$ is a natural weak equivalence. By the usual
cube-by-cube argument, we are reduced to showing that the components
of $\theta$ at the objects of the form $\pr{\square^{n}\to N\pr{\mcal C}}\in\CS_{/N\pr{\mcal C}}$
are weak equivalences. Since the inclusion $\square^{0}\to\square^{n}$
is left anodyne (because it is a finite composite of left deformation
retracts $\square^{k-1}\cong\square^{k-1}\otimes\{0\}\to\square^{k}$),
we can further reduce this to the case where $n=0$. 
In this case, the relevant map is an isomorphism.
\end{proof}
\begin{remark}
It is possible to prove \cref{thm:str} without resorting to
the relative nerve functor. Since we do not need this fact, we will
only give a sketch. 

Consider the functor
\[ \begin{array}{r@{}c@{\ }c@{\ }l}
 j & \from \mcal C^{\op} & \to & \pr{\CS_{/N\pr{\mcal C}}}_{\cov,\infty} \\
 {} & C & \mapsto & \quad \! N\pr{\mcal C_{C/}}
\end{array} \]
where $\pr{\CS_{/N\pr{\mcal C}}}_{\cov,\infty}$ denotes the underlying
$\infty$-category of $\CS_{/N\pr{\mcal C}}$. Since the inclusion
$\{C\}\to N\pr{\mcal C_{C/}}$ is a left deformation retract, it is
left anodyne. 
The derived mapping space lemma \cite[Thm.~2.2]{arakawa-carranza-kapulkin:derived-mapping-space} asserts that the natural cubical enrichment of $\CS_{/ N\pr{\mcal C}}$ correctly computes the mapping space when the codomain is fibrant (cf.\ \cite[Thm.~3.8]{arakawa-carranza-kapulkin:derived-mapping-space}), and using this, we find that $j$ is fully faithful. 
Moreover, objects of the form $\{C\}\to N\pr{\mcal C}$
generate $\pr{\CS_{/N\pr{\mcal C}}}_{\cov,\infty}$ under small colimits,
and these objects are completely compact in the sense of \cite[Def.~5.1.6.2]{lurie:htt}.
It follows from \cite[Cor.~5.1.6.11]{lurie:htt} that $j$ exhibits
the target as a free cocompletion of $\mcal C^{\op}$. By a similar
reasoning, we can show that the composite
\[
\mcal C^{\op}\to\pr{\CS_{/N\pr{\mcal C}}}_{\cov,\infty}\xrightarrow{\Rect_{\infty}}\Fun\pr{\mcal C,\CS}_{\cov,\infty}
\]
also exhibits the target as a free cocompletion of $\mcal C^{\op}$, hence the functor $\Rect_{\infty}$ is an equivalence.
\end{remark}

\subsection{Marked case}

In this subsection, we develop the marked version of the contents
in \cref{thm:str}.
\begin{definition}
Let $\mcal C$ be a small category. We define a functor
\[
\int^{+} \from \Fun\pr{\mcal C,\CS^{+}}\to\CS_{/N\pr{\mcal C}^{\sharp}}^{+}
\]
as follows: Given a functor $F \from \mcal C\to\CS^{+}$, we set $\int^{+}F=\pr{\int F_{\flat},M_{F}}$,
where $F_{\flat} \from \mcal C\to\CS$ denotes the functor obtained from
$F$ by forgetting the markings. The set $M_{F}$ consists of the
edges $\pr{C,x}\to\pr{D,y}$ such that the edge $Ff\pr x\to y$ is
marked in $F\pr D$.

We also define the \emph{rectification functor }
\[
\Rect^{+} \from \CS_{/N\pr{\mcal C}^{\sharp}}^{+}\to\Fun\pr{\mcal C,\CS^{+}}
\]
by the formula $\Rect^{+} ({\overline{X}}) = \overline{X}\times_{N\pr{\mcal C}^{\sharp}}N ({\mcal C_{/\bullet}})^{\sharp}$.
Just as in the unmarked case (\cref{prop:radj_L}), the
functors $\int^{+}$ and $\Rect^{+}$ form an adjoint pair.
\end{definition}

Here is the main result of this section.
\begin{theorem} \label{thm:str_marked}
Let $\mcal C$ be a small category. The adjunction
\[ \begin{tikzcd}
	\Big( \mcSet_{/ \maxm{N(\cat{C})}} \Big)_{\mathrm{cc}} \ar[r, yshift=1.8ex, "\Rect^+"{name=Upper}] & \Fun(\cat{C}, \mcSet)_{\mathrm{proj}} \ar[l, yshift=-1.8ex, "\int^+"{name=Lower}] \ar[from=Upper, to=Lower, phantom, "\perp"]
\end{tikzcd} \]
is a Quillen equivalence for the cocartesian model structure on the
left and the projective model structure on the right.
\end{theorem}

The proof of \cref{thm:str_marked} relies on a lemma. Given
a collection $S$ of objects of $\mbf M$, let $\mbf M_{S}\subset\mbf M$
denote the smallest full subcategory satisfying the following conditions:
\begin{itemize}
\item Every object of $S$ belongs to $\mbf M_{S}$.
\item If an object $X\in\mbf M$ belongs to $\mbf M_{S}$, then every object
weakly equivalent to $X$ belongs to $\mbf M_{S}$.
\item If $F$ is a small diagram in $\mbf M$ that factors through $\mbf M_{S}$,
then $\hocolim F\in\mbf M_{S}$.
\end{itemize}
We say that $S$ \emph{generates $\mbf M$ under homotopy colimits}
if $\mbf M=\mbf M_{S}$.
\begin{lemma}
\label{lem:enr_presheaves}Let $\mcal C$ be a small category. The
maps of marked cubical sets $\overline{X}\to N\pr{\mcal C}^{\sharp}$ whose
underlying map $X\to N\pr{\mcal C}$ are constant generate $\pr{\CS_{/N\pr{\mcal C}^{\sharp}}^{+}}_{\cc}$
under homotopy colimits.
\end{lemma}

\begin{proof}
Let $\Rect^{+} \from \SS_{/N_{\Delta}\pr{\mcal C}^{\sharp}}^{+}\to\Fun\pr{\mcal C,\SS^{+}}$
denote the left adjoint of the relative nerve functor, defined in
\cite[$\S$ 3.2.1]{lurie:htt}. It is a left Quillen equivalence with respect
to the cocartesian model structure and the projective model structure
\cite[Prop.~3.2.5.18]{lurie:htt}. We also write $T^{+} \from \CS_{/N\pr{\mcal C}^{\sharp}}^{+}\to\SS_{/N_{\Delta}\pr{\mcal C}^{\sharp}}^{+}$
for the left adjoint of the right Quillen equivalence of \cref{prop:TUcomparison}. 
We then have the composite left Quillen equivalence
\[
\Phi:\CS_{/N\pr{\mcal C}^{\sharp}}^{+}\xrightarrow{T^{+}}\SS_{/N_{\Delta}\pr{\mcal C}^{\sharp}}^{+}\xrightarrow{\Rect^{+}}\Fun\pr{\mcal C,\SS^{+}}.
\]

If $\overline{X}\to N\pr{\mcal C}^{\sharp}$ is a map whose underlying
map is constant at an object $c \in\mcal C$, then its image
under $\Phi$ is given by $\mcal C\pr{c ,-}^{\sharp}\times T^{+} ({\overline{X}})$.
Consequently, it suffices to show that these objects generate $\Fun\pr{\mcal C,\SS^{+}}$
under homotopy colimits. 
Since the adjunction $(T^{+} \adj U^{+})$
is a Quillen equivalence, every marked simplicial set is weakly equivalent
to an object in the image of $T^{+}$. So we only need to show that
the objects $\{\mcal C\pr{c ,-}^{\sharp}\times\overline{Y}\}_{c \in\mcal C,\,\overline{Y}\in\SS^{+}}$
generate $\Fun\pr{\mcal C,\SS^{+}}$ under homotopy colimits.
This follows from the description of generating cofibrations
of the projective model structure \cite[Thm.~11.6.1]{Hirschhorn}.
\end{proof}
\begin{proof}
[Proof of \cref{thm:str_marked}]
Just as in the proof of \cref{thm:str}, we are reduced to showing that the map

\[
	\theta_{\overline{X}} \from T^+ \big( -\times_{N(\mathcal{C})^\sharp}N(\mathcal{C}_{/\bullet})^\sharp \big) \to T^+(-)\times_{N_{\Delta}(\mathcal{C})^\sharp}N_{\Delta}(\mathcal{C}_{/\bullet})^\sharp
\]

is a weak equivalence in $\Fun(\mathcal{C},\mathsf{sSet}^+)$, for each $\overline{X}\in \CS^+_{/N(\mathcal{C})^\sharp}$. By \cref{lem:enr_presheaves}, we only need to consider the case where the map $X\to N(\mathcal{C})$ factors through some map $\square ^0\to N(\mathcal{C})$. But in this case, the map $\theta_{\overline{X}}$ is an isomorphism, so we are done.
\end{proof}

\section{\label{sec:invariance}Invariance}

So far, our discussion has centered on left fibrations and cocartesian
fibrations over cubical quasicategories. However, we have yet to understand the homotopy theory presented by the covariant and cocartesian model structures over a \emph{non-fibrant} base.
Borrowing intuitions from the simplicial setting, we would expect that these model structures are equivalent if we vary the base by a weak categorical equivalence (\cite[Thm.~5.2.14]{cisinski:higher-categories} and \cite[Prop.~3.3.1.1]{lurie:htt}).
This section is devoted to proving that this indeed the case.
More precisely, we prove the following two theorems:
\begin{theorem}
\label{thm:cat_inv}For every weak categorical equivalence $f \from B\to B'$
of cubical sets, the induced adjunction
\[
f_{!} \from \pr{\CS_{/B}}_{\cov}\bigadj\pr{\CS_{/B'}}_{\cov} \from f^{*}
\]
is a Quillen equivalence.
\end{theorem}

\begin{theorem}
\label{thm:cat_inv_cc}For every weak categorical equivalence $f \from B\to B'$
of cubical sets, the induced adjunction
\[
f_{!} \from \pr{\CS_{/B^{\sharp}}^{+}}_{\cc}\bigadj\pr{\CS_{/B^{\p\sharp}}^{+}}_{\cc} \from f^{*}
\]
is a Quillen equivalence.
\end{theorem}
These results have a number of interesting corollaries that do not follow from the formal aspects of $\infty$-categories; we develop these in \cref{sec:applications}.

The main difficulty in proving these theorems is how little control we have over maps into non-fibrant objects in general. 
We will prove \cref{thm:cat_inv,thm:cat_inv_cc}
by ``induction'' on the base. To perform the inductive step, we
will need a good understanding of how left fibrations and cocartesian
fibrations behave with respect to glueing, which will be addressed
in \cref{subsec:descent}. After this preliminary, we will
prove \cref{thm:cat_inv} in \cref{subsec:cat_inv_unm,subsec:cat_inv_m}.

\subsection{\label{subsec:descent}Descent for left fibrations}

The goal of this subsection is to prove that left fibrations and cocartesian
fibrations enjoy a descent property. To state the main result of this
section, we introduce some terminology.
\begin{definition}
A functor of $\infty$-categories is \emph{essentially a
left fibration} (resp. \emph{essentially a cocartesian fibration})
if it is equivalent to the image of a left fibration (resp. cocartesian
fibration) of quasicategories under the localization functor $\msf{QCat}\to\curlyCat_{\infty}$. 
\end{definition}

\begin{remark}
By \cref{prop:cc_equiv,cor:ccfib_qcat_cqcat},
a functor of $\infty$-categories is essentially a left fibration
if and only if it is the image of a left fibration of cubical quasicategories
under the localization functor $\msf{QCat}_{\square}\to\curlyCat_{\infty}$.
Here, $\msf{QCat}_{\square}$ denotes the $\infty$-categorical localization of the ordinary category $\msf{QCat}_{\square}$ of (small) cubical quasicategories
at categorical equivalences.
\end{remark}

Here is the main result of this section.
\begin{theorem}
[Descent for left fibrations]\label{thm:lfib_descent}Let $K$ be
a small simplicial set, let $\overline{\alpha} \from \overline{p}\to\overline{q}$
be a natural transformation of diagrams $K^{\rcone}\to\curlyCat_{\infty}$.
Suppose that the following conditions are satisfied:
\begin{itemize}
\item The map $\overline{q}$ is a colimit diagram.
\item The restriction $\alpha=\restr{\overline{\alpha}}{K\times\Delta^{1}}$
is cartesian.
\item For each $v\in K$, the map $\alpha_{v} \from \overline{p}\pr v\to\overline{q}\pr v$
is essentially a left fibration (or essentially a cocartesian fibration, respectively).
\end{itemize}
Then the following conditions are equivalent:
\begin{enumerate}
\item The diagram $\overline{p}$ is a colimit diagram. 
\item The natural transformation $\overline{\alpha}$ is cartesian, and
the map $\overline{\alpha}_{\infty} \from \overline{p}\pr{\infty}\to\overline{q}\pr{\infty}$
is essentially a left fibration (respectively, essentially a cocartesian fibration).
\end{enumerate}
\end{theorem}

\begin{remark}
\cref{thm:lfib_descent} is certainly well-known to experts.
In fact, the version of \cref{thm:lfib_descent} for cocartesian
fibrations was already proved by a combinatorial argument in \cite[Prop.~1.4.7]{GC},
and we expect the techniques applied there to be applicable for left
fibrations. The proof we present below relies on a somewhat different
argument, which is perhaps more conceptual.
\end{remark}

We will deduce \cref{thm:lfib_descent} as a corollary of a
more general descent result. For this, the following terminology will
be useful:
\begin{definition}
Let $K$ be a simplicial set, let $\mcal C$ be an $\infty$-category
with pullbacks and $K$-indexed colimits, and let $S$ be a class
of morphisms of $\mcal C$ which is stable under pullback. 
We say that \emph{$K$-indexed colimits are $S$-universal} in $\mcal C$
if for every morphism $f \from X\to Y$ in $S$, the functor $f^{*} \from \mcal C_{/Y}\to\mcal C_{/X}$
preserves $K$-indexed colimits.
\end{definition}
Since $S$ is stable under pullback, the full subcategory $\Fun\pr{\Delta^{1},\mcal C}^{\pr S}\subset\Fun\pr{\Delta^{1},\mcal C}$
spanned by the elements of $S$ determines a cartesian fibration $\Fun\pr{\Delta^{1},\mcal C}^{\pr S}\to\Fun\pr{\{1\},\mcal C}\cong\mcal C$ (cf.\ the discussion in \cite[Notation~6.1.3.4]{lurie:htt}).
We write $\mcal C_{/\bullet}^{\pr S} \from \mcal C^{\op}\to\curlyCat_{\infty}$
for the associated functor.

\begin{proposition}
\label{prop:descent}Let $K$ be a simplicial set, let $\mcal C$
be an $\infty$-category with pullbacks and $K$-indexed colimits,
let $S$ be a class of morphisms of $\mcal C$ which is stable under
pullback, and let $\overline{q} \from K^{\rcone}\to\mcal C$ be a colimit
diagram. Suppose that $K$-indexed colimits are $S$-universal in
$\mcal C$. The following conditions are equivalent:
\begin{enumerate}
\item The composite $\pr{K^{\rcone}}^{\op}\xrightarrow{\overline{q}}\mcal C^{\op}\xrightarrow{\mcal C_{/\bullet}^{\pr S}}\curlyCat_{\infty}$
is a limit diagram.
\item The composite $\pr{K^{\rcone}}^{\op}\xrightarrow{\overline{q}}\mcal C^{\op}\xrightarrow{\mcal C_{/\bullet}^{\pr S}}\curlyCat_{\infty}\xrightarrow{\operatorname{Core}}\mathcal{S}$
is a limit diagram, where $\operatorname{Core}$ denotes the right adjoint to the inclusion.
\item For every natural transformation $\overline{\alpha} \from \overline{p}\to\overline{q}$
of colimit diagrams, if the restriction $\alpha=\restr{\overline{\alpha}}{K\times\Delta^{1}}$
is a cartesian natural transformation and $\overline{\alpha}_{v} \from \overline{p}\pr v\to\overline{q}\pr v$
belongs to $S$, then $\overline{\alpha}$ is a cartesian natural transformation and $\overline{\alpha}_{\infty} \from \overline{p}\pr{\infty}\to\overline{q}\pr{\infty}$
belongs to $S$.
\end{enumerate}
\end{proposition}

\begin{proof}
We will prove (1)$\iff$(3); this is essentially \cite[Lem.~6.1.3.5]{lurie:htt}, but we record the proof for the reader's convenience. 
The equivalence (2)$\iff$(3) can be proved similarly by considering the subcategory of $\Fun(\Delta^1,\mathcal{C})^{(S)}$ spanned by the cartesian squares of $\mathcal{C}$ instead of $\Fun(\Delta^1,\mathcal{C})^{(S)}$.

According to \cite[Prop.~3.3.3.1]{lurie:htt},
condition (1) is equivalent to the condition that the map
\[
\theta \from \Fun_{\mcal C}^{\mrm{cart}}\pr{K^{\rcone},\Fun\pr{\Delta^{1},\mcal C}^{\pr S}}\to\Fun_{\mcal C}^{\mrm{cart}}\pr{K,\Fun\pr{\Delta^{1},\mcal C}^{\pr S}}
\]
be a categorical equivalence. Here the left hand side denotes the
full subcategory of $\Fun_{\mcal C}\pr{K^{\rcone},\Fun\pr{\Delta^{1},\mcal C}^{\pr S}}$
spanned by the objects corresponding to cartesian natural transformations
$K^{\rcone}\times\Delta^{1}\to\mcal C$, and the right hand side is
defined similarly. 

Now set $q=\restr{\overline{q}}{K}$, and write $\Fun\pr{K^{\rcone},\mcal C}_{\mrm{cart},S}^{/\overline{q}}\subset\Fun\pr{K^{\rcone},\mcal C}^{/\overline{q}}$
for the full subcategory spanned by the cartesian natural transformations
$\overline{p}\to\overline{q}$ whose components belong to $S$.
Analogously, let $\Fun\pr{K,\mcal C}_{\mrm{cart},S}^{/q}\subset\Fun\pr{K,\mcal C}^{/q}$
be the full subcategory spanned by the cartesian natural transformations
$p\to q$ whose components belong to $S$. We can then identify $\theta$
with the restriction map
\[
\theta' \from \Fun\pr{K^{\rcone},\mcal C}_{\mrm{cart},S}^{/\overline{q}}\to\Fun\pr{K,\mcal C}_{\mrm{cart},S}^{/q}.
\]

Let $\mcal X\subset\Fun\pr{K^{\rcone},\mcal C}^{/\overline{q}}$
denote the full subcategory spanned by the natural transformations $\overline{\alpha} \from \overline{p}\to\overline{q}$
satisfying the following conditions:
\begin{itemize}
\item The diagram $\overline{p}$ is a colimit diagram.
\item The restriction $\alpha=\restr{\overline{\alpha}}{K\times\Delta^{1}}$
is a cartesian natural transformation whose components belong to $S$.
\end{itemize}
Since $K$-indexed colimits are $S$-universal, the map $\theta'$
factors as
\[
\Fun\pr{K^{\rcone},\mcal C}_{\mrm{cart},S}^{/\overline{q}}\xrightarrow{\iota}\mcal X\xrightarrow{\pi}\Fun\pr{K,\mcal C}_{\mrm{cart},S}^{/q}.
\]
Since the map $\pi$ is an equivalence \cite[Thm.~4.3.2.15]{lurie:htt},
condition (1) is satisfied if and only if $\iota$ is essentially
surjective. But this is equivalent to saying that condition (3) holds,
and we are done.
\end{proof}
\begin{corollary}
\label{cor:descent}Let $K$ be a simplicial set, let $\mcal C$ be
an $\infty$-category with pullbacks and $K$-indexed colimits, and
let $S$ be a class of morphisms of $\mcal C$ which is stable under
pullback. The following conditions are equivalent:
\begin{enumerate}
\item $K$-indexed colimits are $S$-universal in $\mcal C$, and the functor
$\mcal C_{/\bullet}^{\pr S}\from \mcal C^{\op}\to\curlyCat_{\infty}$
preserves $K$-indexed limits.
\item $K$-indexed colimits are $S$-universal in $\mcal C$, and the functor $\operatorname{Core}(\mcal C_{/\bullet}^{\pr S})\from \mcal C^{\op}\to\mathcal{S}$ preserves $K$-indexed limits.
\item Given a colimit diagram $\overline{q} \from K^{\rcone} \to \mcal C$ and a natural transformation $\overline{\alpha} \from \overline{p}\to\overline{q}$ whose restriction $\alpha=\restr{\overline{\alpha}}{K\times\Delta^{1}}$ is cartesian and whose components $\overline{\alpha}_{v} \from \overline{p}(v) \to \overline{q}(v)$ belong to $S$, 
the diagram $\overline{p}$ is a colimit if and only if $\overline{\alpha}$
is a cartesian natural transformation and $\overline{\alpha}_{\infty} \from \overline{p}(\infty) \to \overline{q}(\infty)$ belong to $S$. \qed
\end{enumerate}
\end{corollary}

We can now prove \cref{thm:lfib_descent}: 
\begin{proof}[Proof of \cref{thm:lfib_descent}]
We will prove the assertion for cocartesian fibrations; the proof for left fibrations is similar (and is easier, as we do not have to pass to the core in this case). 

Let $CC$ denote the class of functors of small quasicategories that
are essentially cocartesian fibrations. We will prove the theorem by showing
that $CC$ satisfies condition (1) of \cref{cor:descent}
(with $\curlyCat_{\infty}$ and $\mathcal{S}$ replaced by the $\infty$-categories $\hat{\curlyCat}_{\infty}$ and $\hat{\mathcal{S}}$
of large $\infty$-categories and large $\infty$-groupoids). We must check the following pair of
assertions:

\begin{enumerate}[label=(\alph*)]

\item Small colimits are $CC$-universal in $\curlyCat_{\infty}$.

\item The functor $\operatorname{Core}(\curlyCat_{/\bullet}^{(CC)}):\curlyCat_{\infty}^{\op}\to \hat{\mathcal{S}}$ preserves small limits.
\end{enumerate}

Part (a) follows from \cite[Cor.~B.3.15]{lurie:higher-algebra}, since cocartesian fibrations
are flat categorical fibrations \cite[Ex.~B.3.11]{lurie:higher-algebra}. For (b),
we observe that the functor $\operatorname{Core}( \curlyCat_{/\bullet}^{(CC)} )$  is equivalent
to the functor $\mcal C\mapsto\operatorname{Core} \big(\Fun ({\mcal C, \curlyCat_{\infty}}) \big)$ by straightening
(see \cite[Appendix A]{gepner-haugseng-nikolaus} for a detailed account of this). 
The latter functor manifestly preserves small limits, so we are done.
\end{proof}

\subsection{\label{subsec:cat_inv_unm}Proof of \cref{thm:cat_inv}}

We now turn to the proof of \cref{thm:cat_inv}.
\begin{definition}
\label{def:good_lfib}We will say that a cubical set $X$ is\emph{
good} if it satisfies the following equivalent (by \cref{rem:cat_inv})
conditions:
\begin{enumerate}
\item There is \textit{some} weak categorical equivalence $X\to\mcal X$,
where $\mcal X$ is a cubical quasicategory, such that the adjunction
\[
\CS_{/X}\bigadj\CS_{/\mcal X}
\]
is a Quillen equivalence.
\item For \textit{every} weak categorical equivalence $X\to\mcal X$, where
$\mcal X$ is a cubical quasicategory, the adjunction
\[
\CS_{/X}\bigadj\CS_{/\mcal X}
\]
is a Quillen equivalence.
\end{enumerate}
\end{definition}

\begin{lemma}
\label{lem:square_good}For every $n\geq0$, the cubical set $\square^{n}$
is good.
\end{lemma}

\begin{proof}
A fibrant replacement of $\square^{n}$ in the Joyal model structure
is given by the inclusion $i \from \square^{n}\to N\pr{[1]^{n}}$ (because
this is the derived unit of the adjunction $T \from \CS_{\Joyal}\bigadj\SS_{\Joyal} \from U$).
Our goal is to show that the associated adjunction 
\[
i_{!} \from \pr{\CS_{/\square^{n}}}_{\cov}\bigadj\pr{\CS_{/N\pr{[1]^{n}}}}_{\cov} \from i^{*}
\]
is a Quillen equivalence. Since $i$ is bijective on vertices, \cref{prop:fiberwise_eq_cov} shows that $\mbb Ri^{*}$ is conservative.
It will therefore suffice to show that the derived unit is an isomorphism.

Let $X\to\square^{n}$ be a left fibration. Choose a projective fibrant
replacement
\[
\Rect\pr{i_{!}\pr X}=i_{!}\pr X\times_{N\pr{[1]^{n}}}N\pr{[1]_{/\bullet}^{n}}\to F
\]
in $\Fun\pr{[1]^{n},\CS}$. By \cref{thm:str}, its adjoint
map $\alpha \from i_{!}\pr X\to\int F$ under the $(\Rect \adj \int)$ adjunction is a fibrant replacement of $i_{!}\pr X$
in the covariant model structure over $N\pr{[1]^{n}}$. We must show
that the map
\[
\beta:X\to i^{*}\pr{\int F}
\]
adjoint to $\alpha$ (under the $(i_! \adj i^*)$ adjunction) is a covariant equivalence over $\square^{n}$.
By \cref{prop:fiberwise_eq_cov}, it suffices to show that
$\beta$ induces a homotopy equivalence upon passing to the fibers.

The map $\beta$ factors as $X\xrightarrow{\beta'}i^{*}\pr{\int\Rect i_{!}\pr X}\xrightarrow{\beta''}i^{*}\int F$.
The map $\beta''$ is a fiberwise weak homotopy equivalence by construction,
so it suffices to show that $\beta'$ is a fiberwise weak homotopy
equivalence. To this end, let $v\in\square^{n}$ be an arbitrary vertex.
The map $\beta'$, upon passing to the fibers over $v$, gives the
inclusion
\[
\theta \from X_{v}\cong X\times_{N\pr{[1]^{n}}}\{v\}\to X\times_{N\pr{[1]^{n}}}N\pr{[1]_{/v}^{n}}.
\]
Set $\square_{/v}^{n}=N\pr{[1]_{/v}^{n}}\times_{N\pr{[1]^{n}}}\square^{n}$.
We can identify $\theta$ with the map
\[
\theta' \from X\times_{\square^{n}}\{v\}\to X\times_{\square^{n}}\square_{/v}^{n}.
\]
The cubical set $\square_{/v}^{n}$ is isomorphic to $\square^{k}$
for some $k$ (namely, the number of coordinates of $v$ equal to
$1$). It follows that the inclusion $\{v\}\to\square_{/v}^{n}$ is
a finite composite of right deformation retracts. Since $X\to\square^{n}$
is a left fibration, \cref{prop:retracts} shows that $\theta'$
is right anodyne, hence a weak homotopy equivalence. 
\end{proof}
The following lemma addresses the stability of good cubical sets under
colimits.
\begin{lemma}
\label{lem:good_colim_lfib}Let $\mcal I$ be a small category, and
let $B_{\bullet} \from \mcal I^{\rcone}\to\CS$ be a colimit diagram such
that, for each $i\in\mcal I$, the cubical set $B_{i}$ is good. Assume
that one of the following conditions is satisfied:
\begin{enumerate}
\item $\mcal I$ is a discrete category.
\item $\mcal I=\omega$ is the poset of non-negative integers.
\item $\mcal I$ is the free cospan $2\ot0\to1$, and the map $B_{0}\to B_{1}$
is a monomorphism.
\end{enumerate}
Then the colimit cubical set $B_{\infty}$ is good.
\end{lemma}

\begin{proof}
Choose a natural weak categorical equivalence $B_{\bullet}\to\mcal B_{\bullet}$,
where $\mcal B_{\bullet}$ takes values in the full subcategory of
cubical quasicategories. Let 
\[
F \from \CS_{/\mcal B_{\infty}}\to\CS_{/B_{\infty}}
\]
denote the base change along $B_{\infty}\to\mcal B_{\infty}$. Our
goal is to show that $F$ is a right Quillen equivalence for the covariant
model structure. We will prove this by showing that $\mbb RF$ is
fully faithful and essentially surjective.

First, we show that $\mbb RF$ is fully faithful. Let $\mcal E_{\infty}\to\mcal B_{\infty}$
be an arbitrary left fibration, and set $E_{\infty}=B_{\infty}\times_{\mcal B_{\infty}}\mcal E_{\infty}$.
We wish to show that the map $E_{\infty}\to\mcal E_{\infty}$ is a
covariant equivalence over $\mcal B_{\infty}$. For each $i\in\mcal I$,
set $E_{i}=B_{i}\times_{B_{\infty}}E_{\infty}$ and $\mcal E_{i}=\mcal B_{i}\times_{\mcal B_{\infty}}\mcal E_{\infty}$.
Using that pullbacks in $\cSet$ preserve colimits, any one of the assumptions (1) through (3) implies that the diagram
$E_{\bullet}$ is a homotopy colimit diagram in $\pr{\CS_{/\mcal B_{\infty}}}_{\mrm{cov}}$; i.e., its image in the underlying $\infty$-category is a colimit
diagram
(for case (2), we use the fact that every colimit
diagram indexed by $\omega$ is a homotopy colimit diagram, because
covariant equivalences are stable under filtered colimits \cite[Prop.~4.1]{RR15}).
By \cref{thm:lfib_descent}, the diagram $\mcal E_{\bullet}$
is a homotopy colimit diagram in $\pr{\CS_{/\mcal B_{\infty}}}_{\mrm{cov}}$.
Consequently, we are reduced to showing that for each $i\in\mcal I$,
the map $E_{i}\to\mcal E_{i}$ is a covariant equivalence over $\mcal B_{\infty}$.
But this is clear, because this map is a covariant equivalence over
$\mcal B_{i}$ since each $B_{i}$ is good.

Next, we show that $\mbb RF$ is essentially surjective. Let $E_{\infty}\to B_{\infty}$
be a left fibration. Set $E_{i}=E_{\infty}\times_{B_{\infty}}B_{i}$
for each $i\in\mcal I$, and choose a fibrant replacement 
\[\begin{tikzcd}
	{E_i} & {\mathcal{E}_i} \\
	{B_i} & {\mathcal{B}_i}
	\arrow["{\iota_i}", from=1-1, to=1-2]
	\arrow[from=1-1, to=2-1]
	\arrow["{\pi _i}", from=1-2, to=2-2]
	\arrow["\simeq"', from=2-1, to=2-2]
\end{tikzcd}\]which is natural in $i\in\mcal I$, where $\iota_{i}$ is a covariant
equivalence over $\mcal B_{i}$ and $\pi_{i}$ is a left fibration.
If we are in case (3), we further assume that the map $\mcal E_{0}\to\mcal E_{1}$
is a monomorphism. (This is possible, e.g., by factoring the map $\mcal E_{0}\to\mcal E_{1}$
as a monomorphism $\mcal E_{0}\to\mcal E'_{1}$ followed by a trivial
fibration $\mcal E'_{1}\to\mcal E_{1}$, and then choosing a section
of the latter map in $\CS_{/\mcal B_{0}}$.) 
We then factor the map
$\colim_{i\in\mcal I}\mcal E_{i}\to\mcal B_{\infty}$ as
\[
\colim_{i\in\mcal I}\mcal E_{i}\xrightarrow{j}\mcal E_{\infty}\xrightarrow{\pi_{\infty}}\mcal B_{\infty},
\]
where $j$ is a weak categorical equivalence and $\pi_{\infty}$ is
a categorical fibration. We then complete the proof by verifying the
following pair of assertions:

\begin{enumerate}[label=(\alph*)]

\item The map $\pi_\infty \from \mcal E_{\infty}\to\mcal B_{\infty}$ is a left fibration.

\item The map $\theta \from E_{\infty}\to B_{\infty}\times_{\mcal B_{\infty}}\mcal E_{\infty}$
is a covariant equivalence over $B_{\infty}$.

\end{enumerate}

We start with (a). By construction, the diagram $\mcal E_{\bullet} \from \mcal I^{\rcone}\to\pr{\CS}_{\mrm{Joyal}}$
is a homotopy colimit diagram (for case (2), this again
follows from \cite[Prop.~4.1]{RR15}). 
Therefore, by \cref{thm:lfib_descent}, it suffices to check that for each morphism
$i\to j$ in $\mcal I$, the square 
\[\begin{tikzcd}
	{\mathcal{E}_i} & {\mathcal{E}_j} \\
	{\mathcal{B}_i} & {\mathcal{B}_j}
	\arrow[from=1-1, to=1-2]
	\arrow[from=1-1, to=2-1]
	\arrow[from=1-2, to=2-2]
	\arrow[from=2-1, to=2-2]
\end{tikzcd}\]is homotopy cartesian in the cubical Joyal model structure. Since
left fibrations over cubical quasicategories are categorical fibrations,
this is equivalent to saying that the map
\[
\mcal E_{i}\to\mcal B_{i}\times_{\mcal B_{j}}\mcal E_{j}
\]
be a categorical equivalence. 
The domain and codomain are left fibrations over $\mathcal{B}_i$, hence by \cref{rem:cat_inv,prop:fiberwise_eq_cov}, it suffices to show that,
for each vertex $b\in B_{i}$, the map $\mcal E_{i}\to\mcal E_{j}$
induces a homotopy equivalence
\[
\mcal E_{i,b}\to\mcal E_{j,b},
\]
where the subscript indicates taking fibers over the images of $b$.
Since $B_{i}$ and $B_{j}$ are good, the maps $E_{i,b}\to\mcal E_{i,b}$
and $E_{j,b}\to\mcal E_{j,b}$ are homotopy equivalences. Consequently,
we are reduced to showing that the map $E_{i,b}\to E_{j,b}$ is a
homotopy equivalence.
Since $E_i$ and $E_j$ are defined via pullback ($E_i = B_i \times_{B_{\infty}} E_\infty$ and $E_j = B_j \times_{B_\infty} E_\infty$), the map between their fibers over $b$ is an isomorphism of cubical sets, which suffices.

Next, we prove (b). Again, by \cref{prop:fiberwise_eq_cov},
it suffices to show that for each $i\in\mcal I$ and each vertex $b\in B_{i}$,
the map $E_{\infty,b}\to\mcal E_{\infty,b}$ is a homotopy equivalence.
This map factors as
\[
E_{\infty,b}\cong E_{i,b}\xrightarrow{\phi}\mcal E_{i,b}\xrightarrow{\psi}\mcal E_{\infty,b}.
\]
The map $\phi$ is a homotopy equivalence, so it suffices to show
that $\psi$ is a homotopy equivalence. This follows from an argument
similar to the one in the previous paragraph and \cref{thm:lfib_descent},
the latter of which says that the square 
\[\begin{tikzcd}
	{\mathcal{E}_i} & {\mathcal{E}_\infty} \\
	{\mathcal{B}_i} & {\mathcal{B}_\infty}
	\arrow[from=1-1, to=1-2]
	\arrow[from=1-1, to=2-1]
	\arrow[from=1-2, to=2-2]
	\arrow[from=2-1, to=2-2]
\end{tikzcd}\]is homotopy cartesian in the cubical Joyal model structure. 
\end{proof}
\begin{proof}[Proof of \cref{thm:cat_inv}]
By \cref{rem:cat_inv},
it suffices to show that every cubical set is good in the sense of \cref{def:good_lfib}. By \cref{lem:good_colim_lfib},
we need only prove each representable $\cube{n}$ is
good, which is exactly \cref{lem:square_good}.
\end{proof}

\subsection{\label{subsec:cat_inv_m}Proof of \cref{thm:cat_inv_cc} }

We now take up the proof of \cref{thm:cat_inv_cc}. The proof
is very similar to the case of left fibrations, so we will only indicate
where changes are necessary.
\begin{proof}
[Proof of \cref{thm:cat_inv_cc}]We will say that a cubical
set $X$ is\emph{ good} if it satisfies the following equivalent
(by \cref{prop:cat_inv_cc}) conditions:
\begin{enumerate}
\item There is \textit{some} weak categorical equivalence $X\to\mcal X$,
where $\mcal X$ is a cubical quasicategory, such that the adjunction
\[
\CS_{/X^{\sharp}}^{+}\bigadj\CS_{/\mcal X^{\sharp}}^{+}
\]
is a Quillen equivalence.
\item For \textit{every} weak categorical equivalence $X\to\mcal X$, where
$\mcal X$ is a cubical quasicategory, the adjunction
\[
\CS_{/X^{\sharp}}^{+}\bigadj\CS_{/\mcal X^{\sharp}}^{+}
\]
is a Quillen equivalence.
\end{enumerate}
We want to show that every cubical set is good.

Arguing as in \cref{lem:square_good}, we find that the standard
cubes $\square^{n}$ are all good. With minor modifications, \cref{lem:good_colim_lfib}
remains true for the definition of good cubical sets given above. 
This shows every cubical set is good, as desired.
\end{proof}

\section{\label{sec:applications}Applications of categorical invariance}

In this section, we prove several structural results of left fibrations
and cocartesian fibrations, using \cref{thm:cat_inv,thm:cat_inv_cc}. Some of these results are cubical analogs of
the corresponding results for simplicial sets \cite[$\S$3.3.1]{lurie:htt}.

One immediate consequence of \cref{thm:cat_inv,thm:cat_inv_cc} is that the \emph{alternative} adjunctions on slice categories
induced by triangulation are also Quillen equivalences. 
This allows us to compare the covariant and cocartesian model structures over arbitrary cubical sets with their simplicial versions.
\begin{theorem}\label{TUcomparison_for_all}
	The adjunctions 
	\[ \begin{tikzcd}[sep = small]
		\cSet_{/ B} \ar[rr, "T"{name=Upper}] & {} & \sSet_{/ TB} \ar[ld, "U"] &[2ex] \mcSet_{/ \maxm{B}} \ar[rr, "T^+"{name=UpperT}] & {} & \msSet_{/ \maxm{(TB)}} \ar[ld, "U^+"] \\
		{} & \cSet_{/ UTB} \ar[ul, "(\eta_{B})^*"] \ar[to=Upper, phantom, "\perp"] & {} & {} & \mcSet_{/ \maxm{(UTB)}} \ar[ul, "(\maxm{\eta_{B}})^*"] \ar[to=UpperT, phantom, "\perp"] & {}
	\end{tikzcd} \]
	are Quillen equivalences between the covariant and cocartesian model structures, respectively.
\end{theorem}
\begin{proof}
	We prove the claim for the left adjunction; the proof for the right adjunction is analogous. 

	Choose a weak categorical equivalence $i \from B \xrightarrow{\sim} \mathcal{B}$, where $\mathcal{B}$ is a cubical quasicategory. We further choose a weak categorical equivalence $j \from T \mathcal{B} \xrightarrow{\sim} \widetilde{T \mathcal{B}}$, where $\widetilde{T \mathcal{B}}$ is a quasicategory. 
	The diagram
\[\begin{tikzcd}
	{\cSet_{/B}} && {\sSet_{TB}} \\
	{\cSet_{/\mathcal{B}}} && {\sSet_{/T\mathcal{B}}} \\
	{\cSet_{U \widetilde{T\mathcal{B}}}} & {\sSet_{TU(\widetilde{T\mathcal{B}})}} & {\sSet_{/\widetilde{T\mathcal{B}}}.}
	\arrow["T", from=1-1, to=1-3]
	\arrow["{i_!}"', from=1-1, to=2-1]
	\arrow["{(Ti)_!}", from=1-3, to=2-3]
	\arrow["{\tilde\eta_!}"', from=2-1, to=3-1]
	\arrow["{j_!}", from=2-3, to=3-3]
	\arrow["T"', from=3-1, to=3-2]
	\arrow["{\varepsilon_!}"', from=3-2, to=3-3]
\end{tikzcd}\]
commutes up to natural isomorphism, where $\tilde{\eta}$ denotes the derived unit (associated with $j$), and $\varepsilon$ denotes the ordinary counit (which coincides with the derived counit at fibrant objects). 
\cref{thm:cat_inv} and its simplicial analog (cf. \cite[Thm.~5.2.14]{cisinski:higher-categories} and \cite[Prop. 3.3.1.1]{lurie:htt}) imply that the vertical arrows are left Quillen equivalences, and \cref{prop:TUcomparison} says that the bottom composite $\varepsilon \circ T$ is also a left Quillen equivalence. 
This implies that the right outer composite $j_! \circ (Ti)_! \circ T$ is a left Quillen equivalence.
Since left Quillen equivalences preserve and reflect weak equivalences of cofibrant objects, this implies $T \from \cSet_{/ B}\to \cSet_{/ TB}$ preserves and reflects covariant equivalences.
It also preserves monomorphisms, so we find that it is left Quillen. 
By the 2-out-of-3 property of left Quillen equivalences, we deduce that it is in fact a left Quillen equivalence.
\end{proof}

We can also use \cref{thm:cat_inv,thm:cat_inv_cc} to show that the $(Q \adj \int)$- and $(Q^+ \adj \int^+)-$adjunctions are Quillen equivalences with respect to the covariant and cocartesian model structures, respectively.
\begin{theorem} \label{Q-equiv-cc}
	For a simplicial set $K$, the functors
	\[ \begin{array}{c@{\ }c}
		Q & \from \Big( \sSet_{/ K} \Big)_{\mathrm{cov}} \to \Big( \cSet_{/ QK} \Big)_{\mathrm{cov}} \\
		Q^+ & \from \Big( \msSet_{/ \maxm{K}} \Big)_{\mathrm{cc}} \to \Big( \mcSet_{/ \maxm{(QK)}} \Big)_{\mathrm{cc}}
	\end{array} \]
	are left Quillen equivalences.
\end{theorem}
\begin{proof}
	The proof is entirely analogous to that of \cref{TUcomparison_for_all}, using \cref{thm:int-equiv-cc-qcat} instead of \cref{prop:TUcomparison}.
\end{proof}
As explained in \cref{rem:int-is-not-quillen-equiv-always}, the analog of \cref{Q-equiv-cc} for the alternate slice adjunction is false unless the base is a quasicategory.

One structural result is that the covariant model structure is a left Bousfield localization of the Joyal model structure.
We also provide an analog of item (2) of \cref{prop:cat_inv_cc} over an arbitrary base.
\begin{theorem}
\label{thm:cc_Joy}Let $B$ be a cubical set. 
\begin{enumerate}
\item The covariant model structure on $\CS_{/B}$ is a left Bousfield localization
of the Joyal model structure.
\item The forgetful functor
\[
R_{B} \from \pr{\CS_{/B^{\sharp}}^{+}}_{\cc}\to\pr{\CS_{/B}}_{\Joyal}
\]
is right Quillen, and its total right derived functor is faithful
and conservative.
\end{enumerate}
\end{theorem}

\begin{proof}
We will prove (2) and leave  the proof of part (1) to the reader. 

We first show that $R_{B}$ is right Quillen. It suffices to show
that its left adjoint $L_{B} \from \pr{\CS_{/B}}_{\Joyal}\to\pr{\CS_{/B^{\sharp}}^{+}}_{\cc}$,
given by $\Phi\pr X=X^{\flat}$, is left Quillen. 
It is clear that $L_B$ preserves monomorphisms, so it suffices to show $L_{B}$ preserves weak equivalences of cofibrant objects.
To prove this, choose a weak categorical equivalence $\Psi \from B\to\mcal B$,
where $\mcal B$ is a cubical quasicategory, and consider the following
commutative diagram:
\[\begin{tikzcd}
	{(\mathsf{cSet}_{/B})_{\mathrm{Joyal}}} & {(\mathsf{cSet}^+_{/B^\sharp})_{\mathrm{cc}}} \\
	{(\mathsf{cSet}_{/\mathcal{B}})_{\mathrm{Joyal}}} & {(\mathsf{cSet}^+_{/\mathcal{B}^\sharp})_{\mathrm{cc}}}
	\arrow["{L_B}", from=1-1, to=1-2]
	\arrow["\Psi_*"', from=1-1, to=2-1]
	\arrow["{\Psi'_*}", from=1-2, to=2-2]
	\arrow["{L_{\mathcal{B}}}"', from=2-1, to=2-2]
\end{tikzcd}\]The functor $\Psi'$ is a left Quillen equivalence, so $\Psi'$ reflects
weak equivalences of cofibrant objects (i.e., all objects) by
\cref{thm:cat_inv_cc}. Thus, it suffices to show that $\Psi'\circ L_{B}$
preserves weak equivalences. But $\Psi$ is obviously left Quillen,
and \cref{prop:cat_inv_cc} says that $L_{\mcal B}$ is
left Quillen, so their composite $L_{\mcal B}\circ\Psi=\Psi'\circ L_{B}$
is also left Quillen, as required.

To prove that the total right derived functor $\mbb RR_{B}$ is faithful,
let $X\to B$ and $Y\to B$ be cocartesian fibrations, and let $f_{0},f_{1} \from X^{\natural}\to Y^{\natural}$
be maps over $B^{\sharp}$ such that $f$ and $g$ are homotopic in
$\CS_{/B}$. We wish to show that these maps are homotopic in the
cocartesian model structure. Since $X \otimes E[1] \to X \to B$ is a cylinder object for $X \to B$, there is a map $h \from X\otimes E[1] \to Y$
over $B$ such that $h\vert X\otimes\{\varepsilon\}=f_{\varepsilon}$
for $\varepsilon\in\{0,1\}$. 
Since the inclusion $X^{\natural}\otimes E[1]^{\flat}\subset X^{\natural}\otimes E[1]^{\sharp}$
is marked left anodyne (\cref{prop:mlan_pp}),
the map $h$ extends to a map $X^{\natural}\otimes E[1]^{\sharp}\to Y^{\natural}$
of marked cubical sets. This, in turn, means that $f_{0}$ and $f_{1}$
are homotopic in the cocartesian model structure, as desired.

Finally, we prove that $\mbb RR_{B}$ is conservative. For each vertex
$b\in B$, the pullback functor $\pr{\CS_{/B}}_{\Joyal}\to\pr{\CS_{/\{b\}}}_{\Joyal}$
is right Quillen, so the composite
\[
\CS_{/B^{\sharp}}^{+}\to\pr{\CS_{/B}}_{\Joyal}\to\prod_{b\in B}\pr{\CS_{/\{b\}}}_{\Joyal}
\]
is also right Quillen. Since the total right derived functor of this
composite is conservative (\cref{prop:fiberwise_eq_cc}),
the total right derived functor of $R_{B}$ is conservative as well.
\end{proof}
\begin{corollary}
Every cocartesian fibration of cubical sets is a categorical fibration.
\end{corollary}

\begin{proof}
This follows from \cref{thm:cc_Joy}.
\end{proof}
\begin{corollary}
\label{cor:cc_equiv_cat_equiv}
Given a commutative
diagram of cubical sets
\[\begin{tikzcd}
	X && Y \\
	& B,
	\arrow["f", from=1-1, to=1-3]
	\arrow["p"', from=1-1, to=2-2]
	\arrow["q", from=1-3, to=2-2]
\end{tikzcd}\]
where $p$ and $q$ are cocartesian fibrations and $f$ carries $p$-cocartesian
edges to $q$-cocartesian edges, the following conditions are equivalent:
\begin{enumerate}
\item The map $f$ is a weak categorical equivalence.
\item The map $\natm{f} \from X^{\natural}\to Y^{\natural}$ is a cocartesian equivalence
over $B^{\sharp}$.
\item For each $b\in B$, the map $X_{b}\to Y_{b}$ is a categorical equivalence
of cubical quasicategories.
\end{enumerate}
\end{corollary}

\begin{proof}
The equivalence (1) $\iff$ (2) follows from \cref{thm:cc_Joy},
and the equivalence (1) $\iff$ (3) follows from \cref{prop:fiberwise_eq_cc}.
\end{proof}

Before we state our next result, we introduce a bit of terminology.

\begin{definition}
Let $f \from X\to Y$ be a map of cubical sets. We say that $f$ is\emph{
fully faithful} if there is a commutative diagram 
\[\begin{tikzcd}
	X & Y \\
	{\mathcal{X}} & {\mathcal{Y}}
	\arrow["f", from=1-1, to=1-2]
	\arrow["i"', from=1-1, to=2-1]
	\arrow["\sim", from=1-1, to=2-1]
	\arrow["j", from=1-2, to=2-2]
	\arrow["\sim"', from=1-2, to=2-2]
	\arrow["g"', from=2-1, to=2-2]
\end{tikzcd}\]such that $i,j$ are weak categorical equivalences, $\mcal X$ and
$\mcal Y$ are cubical quasicategories, and $g$ is fully faithful.
\end{definition}

\begin{theorem}
\label{thm:fully_faithful}Let $f \from X\to Y$ be a map of cubical sets.
The following conditions are equivalent:
\begin{enumerate}
\item The map $f$ is fully faithful.
\item The total left derived functor of $f_{!} \from \pr{\CS_{/X}}_{\cov}\to\pr{\CS_{/Y}}_{\cov}$ is full and faithful.
\item The total left derived functor of $f_{!} \from \pr{\CS_{/X^{\sharp}}^{+}}_{\cc}\to\pr{\CS_{/Y^{\sharp}}^{+}}_{\cc}$ is full and faithful.
\end{enumerate}
\end{theorem}

\begin{proof}
We will prove (1)$\iff$(3); the equivalence (2)$\iff$(3) can be
proved similarly. By \cref{thm:cat_inv_cc}, we may assume
that $f$ has the form $U(g) \from U\pr{\mcal X}\to U\pr{\mcal Y}$, where
$g \from \mcal X \to\mcal Y$ is a map of quasicategories. By \cref{prop:TUcomparison}, condition (3) is equivalent to the condition
that the functor $g_{!} \from \pr{\SS_{/\mcal X^{\sharp}}^{+}}_{\cc}\to\pr{\CS_{/\mcal Y^{\sharp}}^{+}}_{\cc}$
have fully faithful left derived functor. By the naturality of straightening and unstraightening
\cite[Appendix A]{gepner-haugseng-nikolaus}, we can rephrase this condition as saying
that the left Kan extension functor
\[
g_{!} \from \Fun\pr{\mcal X,\QC}\to\Fun\pr{\mcal Y,\QC},
\]
which is the right adjoint of the pullback functor $g^{*}$, is fully
faithful. 

Objects of the form $\mcal C \times \mcal X\pr{x,-}$, where $\mcal C$
is a quasicategory and $x\in\mcal X$ is an object, generate $\Fun\pr{\mcal X,\QC}$
under small colimits; this follows, for instance, from the description
of generating cofibrations of the projective model structure \cite[Prop.~A.3.3.2]{Hirschhorn},
or by using \cite[Prop.~4.8.1.17]{lurie:higher-algebra}. 
Since $g^{*}$ preserves
small colimits (being a left adjoint), we find that $g_{!}$ is
full and faithful if and only if the unit is an equivalence at objects of the form $\mathcal{C} \times \mathcal{X}(x, -)$.
In other words, $g_{!}$ is full and faithful precisely when the map
\[
\mcal C\times\mcal X\pr{x,-}\to\mcal C\times\mcal Y\pr{f\pr x,-}
\]
is an equivalence for every $x\in\mcal X$ and $\mcal C\in\QC$, which is equivalent to condition (1).
\end{proof}

We have now come to the final result of this section, which says that strict pullbacks involving cocartesian fibrations are homotopy pullbacks.
\begin{theorem}
\label{cor:cfib_hpb}Suppose we are given a pullback square of cubical
sets 
\[\begin{tikzcd}
	X & {X'} \\
	Y & {Y'}
	\arrow["f", from=1-1, to=1-2]
	\arrow["p"', from=1-1, to=2-1]
	\arrow["{p'}", from=1-2, to=2-2]
	\arrow["g"', from=2-1, to=2-2]
	\ar[from=1-1, to=2-2, phantom, "\pbtick" very near start]
\end{tikzcd}\]
If $p'$ is a cocartesian fibration, then this square is a homotopy pullback
in the cubical Joyal model structure.
\end{theorem}

The proof of \cref{cor:cfib_hpb} requires the following ``gluing
lemma'' for cocartesian fibrations.
\begin{lemma}
\label{HMS22}Consider a commutative diagram of cubical sets
\[\begin{tikzcd}
	X && Y \\
	& Z.
	\arrow["f", from=1-1, to=1-3]
	\arrow["p"', from=1-1, to=2-2]
	\arrow["q", from=1-3, to=2-2]
\end{tikzcd}\]Suppose that:
\begin{enumerate}
\item $p$ and $q$ are cocartesian fibrations, and $f$ preserves cocartesian
edges.
\item The map $f \from X^{\natural}\to Y^{\natural}$ is a marked left fibration.
\item For each vertex $z\in Z$, the map $f_{z} \from X_{z}\to Y_{z}$ is a cocartesian
fibration.
\item For each edge $e \from z\to z'$ in $Z$, the induced map $e_{!} \from X_{z}\to X_{z'}$
carries $f_{z}$-cocartesian edges to $f_{z'}$-cocartesian edges.
\end{enumerate}
Then $f$ is a cocartesian fibration.
\end{lemma}

\begin{proof}
Choose a weak categorical equivalence $k \from Z\to\mcal Z$ with $\mcal Z$
a cubical quasicategory, and find a commutative diagram 
\[\begin{tikzcd}
	{X^\natural} & {\mathcal{X}^{\natural}} \\
	{Y^\natural} & {\mathcal{Y}^{\natural}} \\
	{Z^\sharp} & {\mathcal{Z}^{\sharp}}
	\arrow["i", from=1-1, to=1-2]
	\arrow["f"', from=1-1, to=2-1]
	\arrow["{f'}", from=1-2, to=2-2]
	\arrow["j", from=2-1, to=2-2]
	\arrow["q"', from=2-1, to=3-1]
	\arrow["{q'}", from=2-2, to=3-2]
	\arrow["k", from=3-1, to=3-2]
\end{tikzcd}\]in which $i$ and $j$ are marked left anodyne and $f'$ and $q'$
are marked left fibrations in $\mcSet_{ / \maxm{\mathcal{Z}}}$. 
Since $k$ is a monomorphism, the composite of functors $k^* \circ k_!$ is naturally isomorphic to the identity.
With this, \cref{thm:cat_inv_cc} shows the maps $X^{\natural}\to Z^{\sharp}\times_{\mcal Z^{\sharp}}\mcal X^{\natural}$
and $Y^{\natural}\to Z^{\sharp}\times_{\mcal Z^{\sharp}}\mcal Y^{\natural}$
are cocartesian equivalences of $\CS_{/Z^{\sharp}}^{+}$, since they are obtained by applying $k^*$ to the maps $i$ and $j$, which are weak equivalences between fibrant objects (\cref{cor:cc_equiv_cat_equiv}).
Combining this with \cref{thm:cc_Joy}, we find that in the commutative
diagram 
\[\begin{tikzcd}
	X & {Z\times_\mathcal{Z}\mathcal{X}} \\
	Y & {Z\times_\mathcal{Z}\mathcal{Y}}
	\arrow["\simeq", from=1-1, to=1-2]
	\arrow["f"', from=1-1, to=2-1]
	\arrow["{f'\vert Z}", from=1-2, to=2-2]
	\arrow["\simeq", from=2-1, to=2-2]
\end{tikzcd}\]in $\pr{\CS_{/Z}}_{\Joyal}$, all objects are fibrant, the vertical
arrows are fibrations, and that the horizontal arrows are weak equivalences.
Repeating the argument of \cref{prop:cc_equiv} relative to $Z$,
we find that $f$ is a cocartesian fibration if and only if $f'\vert_{Z}$
is one. Therefore, it suffices to show that $f'$ is a cocartesian
fibration. Since $f'$ satisfies the hypotheses of the proposition,
we are therefore reduced to the case where $Z$ is a cubical quasicategory.

By \cref{lem:cc_fibrations}, it suffices to show $f$ is a categorical fibration.
Since marked left fibrations are inner fibrations, it remains only to show $f$ is an isofibration.
Using \cref{prop:cc_edges_vs_cc_edges} and \cref{prop:TUcomparison},
we may reduce this to the analogous statement in simplicial sets, where the claim is well-known (see, e.g., \cite[Lem.~A.1.8]{HMS22}).
\end{proof}

\begin{proof}[Proof of \cref{cor:cfib_hpb}]
Suppose first that $g$ is inner anodyne. We must show that $f$ is a weak categorical equivalence.
Since $f$ is bijective on vertices, it suffices to show that it is
fully faithful. According to \cref{thm:fully_faithful}, it suffices 
to show that the functor $f_! \from ({\CS_{/X^{\sharp}}^{+}} )_{\cc} \to ({\CS_{/X^{\p\sharp}}^{+}})_{\cc}$
has full and faithful total left derived functor. 

Let $q \from E\to X$ be a cocartesian fibration. We will show that the
derived unit of the adjunction $\mathbb{L}f_{!}\dashv\mathbb{R}f^{*}$
at $q \from \pr{E,\cc\pr q}\to X^{\sharp}$ is an isomorphism, where $\cc\pr q$
denotes the set of $q$-cocartesian edges. To this end, find a commutative
diagram
\[\begin{tikzcd}[row sep = 2.7em]
	{(E,\mathrm{cc}(pq))} & {(E',\mathrm{cc}(p'q'))} \\
	{(X,\mathrm{cc}(p))} & {(X',\mathrm{cc}(p'))} \\
	{Y^\sharp} & {\maxm{(Y')},}
	\arrow["u", from=1-1, to=1-2]
	\arrow["q"', two heads, from=1-1, to=2-1]
	\arrow["{q'}", two heads, from=1-2, to=2-2]
	\arrow[from=2-1, to=2-2]
	\arrow["{p}"', two heads, from=2-1, to=3-1]
	\arrow["p'", two heads, from=2-2, to=3-2]
	\arrow[from=3-1, to=3-2]
\end{tikzcd}\]
such that $u$ is marked left anodyne and the two-headed arrows are marked
left fibrations. We observe that $q'$ is a cocartesian fibration.
Indeed, by \cref{HMS22}, this is implied by the following claims:

\begin{enumerate}[label=(\alph*)]

\item for each vertex $y\in Y$, the map $q'_{g\pr y} \from E'_{g\pr y}\to X'_{g\pr y}$
is a cocartesian fibration; and

\item for each edge $e \from y_{0}\to y_{1}$ in $Y$, the map $e_{!} \from E'_{g\pr{y_{0}}}\to E'_{g\pr{y_{1}}}$
carries $q'_{g\pr{y_{0}}}$-cocartesian edges to $q_{g\pr{y_{1}}}$-cocartesian
edges.

\end{enumerate}

These claims follow from \cref{prop:cc_equiv}, because the map $E_{y}\to E'_{g\pr y}$
is a categorical equivalence by \cref{thm:cat_inv_cc,prop:fiberwise_eq_cc}, and because the map $X_{y}\to X'_{g\pr y}$ is an isomorphism of cubical sets.

The above argument actually proves something stronger: every $q$-cocartesian
edge is mapped to a $q'$-cocartesian edge, and every $q'$-cocartesian
edge is a composite of a $p'q'$-cocartesian edge and the image of
a $q$-cocartesian edge. 
It follows that the map
\[
	\big( E',\cc(p'q') \big) \push\limits_{(E,\cc(pq))} \big( E,cc(q) \big) \to \big(E',  cc(q') \big)
\]
is well-defined and is marked left anodyne (as a composite of maps in (ML5)).
Therefore, the top face of the diagram
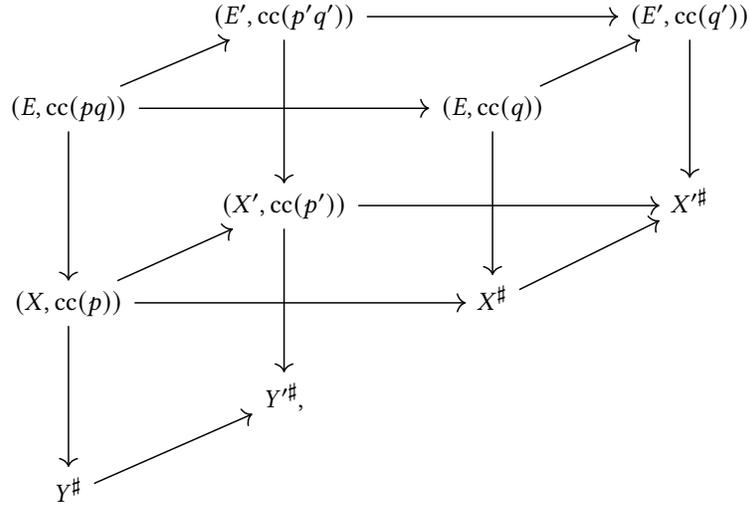
\[\begin{tikzcd}
	& {(E',\mathrm{cc}({p'q'}))} && {(E',\mathrm{cc}(q'))} \\
	{(E,\mathrm{cc}({pq}))} && {(E,\mathrm{cc}(q))} \\
	& {(X',\mathrm{cc}({p'}))} && {X^{\prime\sharp}} \\
	{(X,\mathrm{cc}(p))} && {X^\sharp} \\
	& {Y^{\prime\sharp},} \\
	{Y^\sharp}
	\arrow[from=1-2, to=1-4]
	\arrow[from=1-2, to=3-2]
	\arrow[from=1-4, to=3-4]
	\arrow[from=2-1, to=1-2]
	\arrow[from=2-1, to=2-3]
	\arrow[from=2-1, to=4-1]
	\arrow[from=2-3, to=1-4]
	\arrow[from=2-3, to=4-3]
	\arrow[from=3-2, to=3-4]
	\arrow[from=3-2, to=5-2]
	\arrow[from=4-1, to=3-2]
	\arrow[from=4-1, to=4-3]
	\arrow[from=4-1, to=6-1]
	\arrow[from=4-3, to=3-4]
	\arrow[from=6-1, to=5-2]
\end{tikzcd}\]is homotopy cocartesian in $\pr{\CS_{/Y^{\prime\sharp}}^{+}}_{\cc}$.
In particular, the map $u \from \big( {E,\cc\pr q} \big) \to \big( E',\cc\pr{q'} \big)$ is
a cocartesian equivalence over $Y'$, so the derived unit at $\big( {E,\cc\pr q} \big)$
is realized by the map $\big( {E,\cc\pr q} \big) \to \big( {E',\cc\pr{q'}} \big) \times_{X^{\sharp}} X^{\prime\sharp}$.
This is a cocartesian equivalence over $X$ by \cref{cor:cc_equiv_cat_equiv},
because its underlying map can be identified with that of $\big( {E,\cc\pr{pq}} \big) \to \big( {E',\cc\pr{q'}} \big) \times_{Y^{\sharp}} Y^{\prime\sharp}$.

Before we proceed, we record an important consequence of what we have
just proved:
\begin{itemize}
\item [($\ast$)]For any cocartesian fibration $\pi \from E\to B$ of cubical
sets and every inner anodyne extension $\iota \from B \xrightarrow{\sim}B'$, there is a commutative diagram 
\[\begin{tikzcd}
	E & {E'} \\
	B & {B'}
	\arrow["\alpha", from=1-1, to=1-2]
	\arrow["\sim"', from=1-1, to=1-2]
	\arrow["\pi"', from=1-1, to=2-1]
	\arrow["{\pi'}", from=1-2, to=2-2]
	\arrow["\iota"', from=2-1, to=2-2]
	\arrow["\sim", from=2-1, to=2-2]
\end{tikzcd}\]satisfying the following conditions:
\begin{itemize}
\item The map $\pi'$ is a cocartesian fibration.
\item The map $\alpha$ is a weak categorical equivalence carrying $\pi$-cocartesian
edges to $\pi'$-cocartesian edges.
\end{itemize}
\end{itemize}
To prove ($\ast$), just take $\pi'$ as a cocartesian fibration which
admits a cocartesian equivalence $E^{\natural}\xrightarrow{\sim}B^{\sharp}\times_{B^{\p\sharp}}E^{\p\natural}$,
which exists by \cref{thm:cat_inv_cc}. The map $E\to B\times_{B'}E'$
is then a weak categorical equivalence by \cref{cor:cc_equiv_cat_equiv},
while the projection $B\times_{B'}E'\to E'$ is a weak categorical
equivalence by what we have just shown.

We can now prove the corollary in the general case. Find a commutative
diagram 
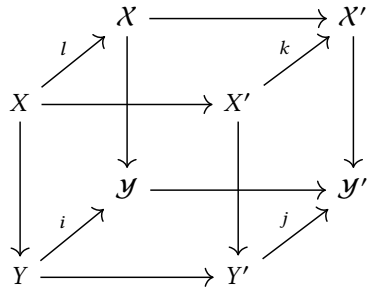
\[\begin{tikzcd}
	& {\mathcal{X}} && {\mathcal{X}'} \\
	X && {X'} \\
	& {\mathcal{Y}} && {\mathcal{Y}'} \\
	Y && {Y'}
	\arrow[from=1-2, to=1-4]
	\arrow[from=1-2, to=3-2]
	\arrow[from=1-4, to=3-4]
	\arrow["l", from=2-1, to=1-2]
	\arrow[from=2-1, to=2-3]
	\arrow[from=2-1, to=4-1]
	\arrow["k", from=2-3, to=1-4]
	\arrow[from=2-3, to=4-3]
	\arrow[from=3-2, to=3-4]
	\arrow["i", from=4-1, to=3-2]
	\arrow[from=4-1, to=4-3]
	\arrow["j", from=4-3, to=3-4]
\end{tikzcd}\]where all vertical maps are cocartesian fibrations, $i,j$ are inner
anodyne extensions, $k$ is a weak categorical equivalence mapping
cocartesian edges over $Y'$ to those over $\mcal Y'$, and the back
square is a pullback (such a diagram exists by ($\ast$)). The map
$u \from X'\to j^{*}\mcal X'$ is a weak categorical equivalence, since
$j^{*}\mcal X'\to\mcal X'$ is a weak categorical equivalence by the
special case we treated above. By \cref{cor:cc_equiv_cat_equiv},
this means that $u$ is a fiberwise categorical equivalence. In particular,
its pullback $v \from X\to i^{*}\mcal X$ has the same property. Again
by \cref{cor:cc_equiv_cat_equiv}, $v$ is a weak categorical
equivalence. Another application of the special case we treated above
says that the map $i^{*}\mcal X\to\mcal X$ is a weak categorical
equivalence. In conclusion, all the maps $i,j,k,l$ are weak categorical
equivalences. Consequently, it will suffice to show that the back
face is a homotopy pullback. This is clear, because the square
consists of cubical quasicategories and the map $\mcal X'\to\mcal Y'$
is a categorical fibration.
\end{proof}

\section{\label{sec:formula}Cubical Bousfield--Kan formula}

As another (unexpected) application of the machinery we have developed
so far, we give a cubical version of the classical Bousfield--Kan
formula.

Computing homotopy colimits in general model categories can be
challenging. However, for simplicial model categories, there is a
clean formula for homotopy colimits, called the \emph{Bousfield--Kan
formula}. To recall, let $\mbf M$ be a simplicial model
category, and let $F \from \mcal C\to\mbf M$ be a small diagram in $\mbf M$.
The Bousfield--Kan formula for the homotopy colimit of $F$ is given by the weighted
colimit of $F$ with weight $N_{\Delta}\pr{\mcal C_{\bullet/}} \from \mcal C^{\op}\to\SS$;
that is, by the functor tensor product 
\[
N_{\Delta}\pr{\mcal C_{\bullet/}}\otimes_{\mcal C}F=\int^{c\in\mcal C}N_{\Delta}\pr{\mcal C_{c/}}\otimes F(c).
\]
Notice that the weighted colimit of $F$ with trivial weight (constant
at $\Delta^{0}$) computes the ordinary colimit. Since each $N_{\Delta}\pr{\mcal C_{C/}}$
is weakly contractible, we can think of this formula as obtained from
the ordinary colimit by ``fattening up'' the weight $N_{\Delta}\pr{\mcal C_{\bullet/}}$;
more precisely, it is a projective cofibrant replacement of the trivial
weight.

It is natural to ask to what extent the Bousfield--Kan formula is
valid in the cubical setting. We can answer this question by using
the machinery developed earlier:
\begin{theorem}
\label{thm:BK}Let $\mbf M$ be a cofibrantly generated model category
equipped with a left Quillen bifunctor 
\[ \otimes \from \CS_{\mathrm{Kan-Quillen}}\times\mbf M\to\mbf M, \]
and let $F \from \mcal C\to\mbf M$ be a small diagram. The functor tensor
product $N\pr{\mcal C_{\bullet/}}\otimes_{\mcal C}F$ computes the
homotopy colimit of $F$.
\end{theorem}

\begin{proof}
The dual of \cref{thm:str} shows that the weight $W=N\pr{\mcal C_{\bullet/}} \from \mcal C^{\op}\to\CS$
is projectively cofibrant. As explained in \cite[Rem.~A.2.9.27]{lurie:htt},
this implies that the functor $W \otimes_{\mcal C} - \from \Fun\pr{\mcal C,\mbf M}\to\mbf M$
preserves weak equivalences of pointwise cofibrant diagrams. Since
$W(c)$ is weakly contractible for all $c \in\mcal C$, its total
left derived functor is isomorphic to that of $\colim \from \Fun\pr{\mcal C,\mbf M}\to\mbf M$.
\end{proof}
\begin{remark}
\label{rem:BK}In the situation of \cref{thm:BK}, the functor
tensor product $N\pr{\mcal C_{\bullet/}}\otimes_{\mcal C}F$ can be
written as $\square^{\bullet}\otimes_{\square} B_{\bullet}\pr{\ast,\mcal C,F}=|B_\bullet\pr{\ast,\mcal{C},F}|$,
where $B_{\bullet}\pr{\ast,\mcal C,F}$ denotes the \emph{cubical
Bar construction}, defined by
\[
B_{n}\pr{\ast,\mcal C,F}=\coprod_{d:[1]^{n}\to\mcal C}F\pr{d\pr{0^{n}}}.
\]
(The rewriting of the simplicial case is explained in \cite[Thm.~6.6.1]{cathtpy},
and the argument there is valid in the cubical case as well.) We can
use this to give an alternative proof of \cref{thm:BK}, as
follows.

Let $\pi \from \int N\pr{\mcal C}\to\square^{\op}$ denote the category of elements of the functor $N\pr{\mcal C} \from \square^{\op}\to\msf{Set}$.
Its objects are diagrams $d \from [1]^{n}\to\mcal C$, with a morphism
$\pr{d \from [1]^{n}\to\mcal C}\to\pr{d' \from [1]^{m}\to\mcal C}$ given by a
map $u \from \square^{m}\to\square^{n}$ such that $du=d'$. There is an
``evaluation at $0$'' functor
\[ \begin{array}{r c c c}
	\ev_{0} \from & \int N\pr{\mcal C} & \to & \mcal C \\
 	{} & (d \from [1]^{n}\to\mcal C) & \mapsto & d(0^n)
\end{array} \]
and $B_{\bullet}\pr{\ast,\mcal C,F}$ is the left Kan extension of
$F\circ\ev_{0}$ along $\pi$:
\[\begin{tikzcd}
	{\int N(\mathcal{C})} & {\mathcal{C}} & {\mathbf{M}} \\
	{\square^{\mathrm{op}}}
	\arrow[""{name=0, anchor=center, inner sep=0}, "{\mathrm{ev}_0}", from=1-1, to=1-2]
	\arrow["\pi"', from=1-1, to=2-1]
	\arrow["F", from=1-2, to=1-3]
	\arrow[""{name=1, anchor=center, inner sep=0}, "{B_\bullet (\ast,\mathcal{C},F)}"', dotted, from=2-1, to=1-3]
\end{tikzcd}\]This left Kan extension is a homotopy left Kan extension because coproducts
of cofibrant objects are homotopy coproducts, and the object $B_{\bullet}\pr{\ast,\mcal C,F}$
is Reedy cofibrant for a similar reason. Thus, to prove \cref{thm:BK},
it suffices to show that $\ev_{0}$ is homotopy final, i.e., that
for each $c \in\mcal C$, the fiber product $N \left( \int \mcal{C} \times_{\mcal C}\mcal C_{c/} \right) \cong\int N\pr{\mcal C_{c/}}$
is weakly contractible. 
By the naturality of \cref{thm:str} 
in $\mcal C$, the right hand side computes the homotopy colimit of
the diagram $N\pr{\mcal C_{c/}} \from \square^{\op}\to\Set\hookrightarrow\CS$.
This homotopy colimit is weakly equivalent to $N\pr{\mcal C_{c/}}$ (\cite[Cor.~3.14]{arakawa-carranza-kapulkin:derived-mapping-space}), so $\int N\pr{\mcal C_{c/}}$ is indeed weakly contractible.
\end{remark}

\begin{remark}\label{rem:fatBK}
	In the situation of \cref{rem:BK}, it is tempting to compute the homotopy colimit of a cubical diagram $X$ in $\mathbf{M}$ by the ``fat realization'' 
\[
	\| X \| = \square^\bullet\otimes_{\square_{+}} X=\int^{\square^n\in \square_{+}}\square^{n}\otimes X_n,	
\]
where $\square_{+} \subset \square$ denotes the subcategory spanned by the face maps.
Unfortunately, this does \emph{not} compute the homotopy colimit of the cubical diagram $X$ in general. 
In contrast to the simplicial case, the inclusion $i \from (\square_{+})^{\op}\to \square^\op$ is not homotopy final. In fact, the classifying space $B\square_{+}$ is homotopy equivalent to $\Omega S^2$, so $i$ fails to induce a homotopy equivalence between classifying spaces.

To see this, for each $n \geq 0$, let $S(n)$ denote the set of morphisms $[1]^n \to [1]^m$ in $\boxcat$ which are surjective. The inclusion
\[
	S(n)\to \square_+\times_{\square}\square_{[1]^n/}
\]
is initial. Combining this with the fact that coproducts of cubical sets are already homotopy coproducts, we deduce that the left Kan extension functor $i_! \from \CS^{(\square_{+})^\op}\to \CS^{\square^\op}$ agrees with the homotopy left Kan extension functor (where $\cSet$ is equipped with the Grothendieck model structure). 
Thus, we have weak homotopy equivalences of cubical sets
\[
	B\square_+\simeq B((\square_+)^\op)\simeq \hocolim\limits_{(\square_+)^\op}\ast \simeq \hocolim\limits_{\square^\op}i_!(\ast) \simeq i_!(\ast),
\]
where the last equivalence is an instance of \cite[Prop.~3.14]{arakawa-carranza-kapulkin:derived-mapping-space}.

The cubical set $i_!(\ast)$ has a unique nondegenerate cube in each dimension, and nondegenerate cubes are stable under the face operations. So the topological geometric realization of $i_!(\ast)$ is the quotient of $\coprod_{n\geq0}[0,1]^n$ with respect to the relation $(t_1,\dots,t_n) \sim (t_1,\dots ,\hat{t_i} ,\dots, t_n)$ when $t_i\in\{0,1\}$ (the hat notation indicates that $t_i$ is discarded). 
This recovers the James construction $J(S^1)$, which is homotopy equivalent to $\Omega(\Sigma S^1)=\Omega S^2$ \cite[Thm.~4J.1]{hatcher}.

\end{remark}

\renewcommand{\plain}{\varnothing}
Let $\square_{\plain}$ denote the minimal cubical
category, i.e., the cubical category without connections. 
Our discussion so far remain valid if we replace $\square$ by
$\square_{\plain}$. 
In particular, \cref{thm:BK} goes through
if we replace $\CS$ by the category $\CS_{\plain}$ of presheaves over $\square_{\plain}$
(with the Kan--Quillen model structure), since we only use that the rectification functor of \cref{thm:str} preserves
cofibrant objects (or equivalently, that its right adjoint preserves
trivial fibrations). 
\begin{proposition} \label{rem:plain}
	Let $\mbf M$ be a cofibrantly generated model category
	equipped with a left Quillen bifunctor 
	\[ \otimes \from {\big( {\cSet}_{\plain} \big)}_{ \mathrm{Kan-Quillen}} \times {\mathbf{M}} \to {\mathbf{M}}, \]
	and let $F \from \mcal C\to\mbf M$ be a diagram. The functor tensor
	product $N\pr{\mcal C_{\bullet/}}\otimes_{\mcal C}F$ computes the
	homotopy colimit of $F$. \qed
\end{proposition}

\cref{rem:plain} is particularly useful in the context of monoidal
model categories, because the category of plain cubical sets acts
on essentially every monoidal category we meet in practice. Recall
that a \emph{monoidal model category} $\mbf M$ is a monoidal category
equipped with a monoidal structure, satisfying the pushout-product
axiom. We say that it satisfies \emph{Muro's unit axiom} if there
is a cofibrant replacement $q \from \widetilde{\mbf 1}\to\mbf 1$ of the
unit object, such that $q\otimes X$ and $X\otimes q$ are weak equivalences
for every object $X\in\mbf M$. (To our knowledge, all known monoidal
model categories satisfy Muro's axiom.)
\begin{theorem} \label{cor:BK_mon}
Let $\mbf M$ be a combinatorial monoidal model category satisfying Muro's unit axiom.
Let $\mbf 1\sqcup\mbf 1\to I\to\mbf 1$ be a cylinder of the unit object $\mbf 1\in\mbf M$ (i.e., a factorization of the codiagonal into a cofibration followed by a weak equivalence).
Given a small diagram $F \colon \mcal C\to\mbf M$ taking values in the full
subcategory of cofibrant objects, the homotopy colimit of $F$ may be computed by the formula
\[
	\int^{[1]^n\in\square_{\plain}}I^{\otimes n}\otimes B_{n}\pr{\ast,\mcal C,F}.
\]
\end{theorem}

\begin{proof}
By \cite[Thm.~1]{Mur15}, we can add cofibrations to $\mbf M$
without changing the class of weak equivalences and so that $\mbf 1$
becomes cofibrant. Since this procedure does not affect homotopy colimits,
we may assume that $\mbf 1$ is cofibrant to begin with.

Using the universal properties of $\square_{\plain}$ \cite[Thm.~4.2]{GM03}
and the Day convolution product, we can extend the cylinder
$I$ to a cocontinuous monoidal functor $\Phi \from \CS_{\plain}\to\mbf M$
carrying $\square^{n}$ to $I^{\otimes n}$. This functor is left
Quillen: to show that it preserves cofibrations, we only need to show
that it maps $i_{n} \from \partial\square^{n}\to\square^{n}$ to a cofibration
for all $n\geq0$. 
For $n=0$, this follows from the cofibrancy of
$\mbf 1$. 
For $n=1$, this exactly says that the map $\mbf 1\sqcup\mbf 1\to I$
is a cofibration. 
Since $i_{n}$ is an iterated pushout product of
$i_{1}$, we deduce that $\Phi\pr{i_{n}}$ is a cofibration in general.
An analogous argument shows that $\Phi$ preserves trivial cofibrations.

Since $\mbf M$ satisfies the pushout-product axiom, the functor 
\[ {\cSet}_{\plain}\times {\mathbf{M}} \xrightarrow{\Phi\times\id} {\mathbf{M}} \times {\mathbf{M}} \xrightarrow{\otimes} {\mathbf{M}} \]
is a left Quillen bifunctor.
The claim now follows by applying
\cref{thm:BK} (and \cref{rem:BK}) to
this bifunctor.
\end{proof}

The observation that $\Phi$ is left Quillen, although likely known to experts before, was first made by Lawson in \cite[Cor.~1.5]{Law17}.

\begin{remark}
To adopt the proof of \cref{cor:BK_mon} to the
simplicial Bousfield--Kan formula, one would need a \textit{cosimplicial
resolution} of the unit object. 
In practice, there is often no canonical choice for such a resolution, and so the resulting formula for homotopy colimits will likely not be as useful as that of \cref{cor:BK_mon}.
\end{remark}

 \bibliographystyle{R-amsalphaurlmod}
 \bibliography{R-all-refs.bib}

\end{document}